\documentclass[10pt]{article}
\usepackage{amsmath,amssymb,amscd,amsthm}

\nonstopmode

\usepackage[dvips]{graphicx}
\usepackage[dvips]{color}
\usepackage{upgreek}

\usepackage{times}
\definecolor{MyDarkBlue}{rgb}{0,0.08,0.4}
\definecolor{Pomegranade}{rgb}{0.6,0.1,0.15}
\definecolor{purple}{rgb}{0.6,0.1,0.15}
\usepackage[backref=none]{hyperref}
\hypersetup{pdfborder={0 0 0},
  colorlinks,
  urlcolor={black},
  linkcolor={MyDarkBlue},
  citecolor={MyDarkBlue},
  breaklinks=true}
\providecommand{\url}[1]{\small\textcolor{MyDarkBlue}{#1}}
\usepackage{doi}
\providecommand{\eprint}[1]{}
\renewcommand{\eprint}[1]{arXiv:\href{http://arxiv.org/abs/#1}{#1}}

\usepackage{bm}
\usepackage{mathrsfs}
\usepackage{graphicx}

\newcommand{\notyet}[1]{}

\usepackage[hmargin=2.5 cm,vmargin=2.0 cm]{geometry}

\makeatletter\@addtoreset {equation}{section}\makeatother

\DeclareSymbolFont{EUR}{U}{eur}{m}{n}
\SetSymbolFont{EUR}{bold}{U}{eur}{b}{n}
\DeclareSymbolFontAlphabet{\eur}{EUR}

\DeclareSymbolFont{EUB}{U}{eur}{b}{n}
\SetSymbolFont{EUB}{bold}{U}{eur}{b}{n}
\DeclareSymbolFontAlphabet{\eub}{EUB}

\def\C{\mathbb{C}}
\def\R{\mathbb{R}}

\def\bmupphi{\bm\upphi}

\def\Eta{\mathrm{H}}

\newcommand{\bD}{\mathbf{D}}

\newcommand{\bJ}{\mathbf{J}}
\newcommand{\bW}{\mathbf{W}}
\newcommand{\bL}{\mathbf{L}}
\newcommand{\bN}{\mathbf{N}}

\newcommand{\X}{\mathbf{X}}
\newcommand{\bX}{\mathbf{X}}

\newcommand{\f}{\mathbf{f}}
\newcommand{\g}{\mathbf{g}}

\newcommand{\wei}[1]{\langle #1 \rangle}
\newcommand{\End}{\mathop{\mathrm{End}}}

\newcommand{\sgn}{\mathop{\rm sgn}}
\newcommand{\supp}{\mathop{\rm supp}}
\newcommand\norm[1]{\left\| #1\right\|}
\newcommand\Norm[1]{\Big\| #1\Big\|}
\newcommand\at[1]{\vert\sb{#1}}

\newcommand{\tr}{\mathop{\rm tr}}
\newcommand{\bmupalpha}{\bm\upalpha}
\newcommand{\bmupbeta}{\bm\upbeta}

\renewcommand{\Im}{\mathop{\rm Im}}
\renewcommand{\Re}{\mathop{\rm Re}}
\newcommand{\F}{\mathbf{F}}
\newcommand{\G}{\mathbf{G}}

\def\ve{\varepsilon}

\newtheorem{theorem}{Theorem}[section]
\newtheorem{lemma}[theorem]{Lemma}
\newtheorem{proposition}[theorem]{Proposition}
\newtheorem{remark}[theorem]{Remark}
\newtheorem{assumption}[theorem]{Assumption}

\newtheorem{definition}[theorem]{Definition}
\newtheorem{corollary}[theorem]{Corollary}

\newtheorem{claim}[theorem]{Claim}

\newcommand{\abs}[1]{\lvert#1\rvert}
\newcommand{\Abs}[1]{\left\lvert#1\right\rvert}

\newcommand{\om}{\omega}

\newcommand{\rone}{\mathbb{R}^1}

\newcommand{\ck}{\mathcal K}

\newcommand{\p}{\partial}

\newcommand{\beq}{\begin{equation}}
\newcommand{\eeq}{\end{equation}}
\newcommand{\beqna}{\begin{eqnarray*}}
\newcommand{\eeqna}{\end{eqnarray*}}
\newcommand{\beqn}{\begin{equation*}}
\newcommand{\eeqn}{\end{equation*}}
\newcommand{\bp}{\begin{proof}}
\newcommand{\bprop}{\begin{proposition}}
\newcommand{\eprop}{\end{proposition}}
\newcommand{\bt}{\begin{theorem}}
\newcommand{\et}{\end{theorem}}
\newcommand{\bex}{\begin{Example}}
\newcommand{\eex}{\end{Example}}
\newcommand{\bc}{\begin{corollary}}
\newcommand{\ec}{\end{corollary}}
\newcommand{\bcl}{\begin{claim}}
\newcommand{\ecl}{\end{claim}}
\newcommand{\bl}{\begin{lemma}}
\newcommand{\el}{\end{lemma}}

\usepackage{color}

\usepackage[dvips]{epsfig}
\usepackage{graphicx}

\begin{document}

\title{Asymptotic stability of solitary waves
in generalized Gross--Neveu model
}

\author{
{\sc Andrew Comech}
\\
{\small\it Texas A\&M University, College Station, TX 77843, USA
and IITP, Moscow 101447, Russia}
\\~\\
{\sc Tuoc Van Phan}
\\
{\small\it
Department of Mathematics, University of Tennessee, Knoxville, TN 37996, USA}
\\~\\
{\sc Atanas Stefanov\footnote{Stefanov is supported in part by NSF-DMS under contract $\#$ 1313107.}}
\\
{\small\it Department of Mathematics, University of Kansas, Lawrence, KS 66045, USA}
}

\date{July 3, 2014}

\maketitle

\begin{abstract}
For the nonlinear Dirac equation in (1+1)D
with scalar self-interaction (Gross--Neveu model),
with quintic and higher order nonlinearities
(and within certain range of the parameters),
we prove that solitary wave solutions are asymptotically stable
in the ``even'' subspace of perturbations
(to ignore translations and eigenvalues $\pm 2\omega i$).
The asymptotic stability is proved for initial data in $H^1$.
The approach is based on the spectral information
about the linearization at solitary waves
which we justify by numerical simulations.
For the proof,
we develop the spectral theory for the linearized operators
and obtain appropriate estimates
in mixed Lebesgue spaces, with and without weights.
\end{abstract}

\section{Introduction}

Models of self-interacting spinor fields have been
appearing in particle physics for many years
\cite{jetp.8.260,PhysRev.83.326,PhysRev.103.1571,
RevModPhys.29.269}.
The most common examples
of nonlinear Dirac equation
are the massive Thirring model \cite{MR0091788}
(vector self-interaction)
and the Soler model \cite{PhysRevD.1.2766}
(scalar self-interaction).
The (1+1)D analogue of the latter model
is widely known as the massive Gross--Neveu model \cite{PhysRevD.12.3880}.
In the present paper,
we address the asymptotic stability of solitary waves in this model.
We require that the nonlinearity in the equation
vanishes of order at least five;
the common case of cubic nonlinearity seems out of reach with the current technology;
there is a similar situation with other popular dispersive models
in one spatial dimension,
such as the Schr\"odinger and Klein-Gordon equations (see 
\cite{MR1221351,MR1199635, MR2511047, MR2422523,MR2968226}
and the references therein).

We only consider perturbations in the class of
``even'' spinors (same parity as the solitary waves under consideration).
The restriction to this subspace
allows us to ignore spatial translations
and the $\pm 2\omega i$ eigenvalues which are present
in the spectrum of the linearization at solitary waves \cite{dirac-vk}.
This paper therefore may be considered as the extension of \cite{MR2985264}
to the translation-invariant systems
(in that paper, the potential was needed to obtain the desired
spectrum of linearization at small amplitude solitary waves).

A similar result -- asymptotic stability of solitary waves
in the translation-invariant
nonlinear Dirac equation in three spatial dimensions --
is obtained in \cite{MR2924465}.
Authors base their highly technical approach
on a series of assumptions about the
spectrum of the linearizations at solitary waves;
these assumptions
can not be verified yet for a particular model.
The authors also restrict the perturbations
to a certain subspace to avoid spatial translations and
issues caused by the presence of $\pm 2\omega i$ eigenvalues
\cite{dirac-vk}
and only consider the solitary waves with $\omega>m/3$.
Contrary to \cite{MR2924465},
our results are obtained for models
for which the spectrum is known (albeit numerically);
our technical restriction is $\abs{\omega}<m/3$.

We briefly review the related research on stability of 
solitary waves in nonlinear Dirac equation.
There have been numerous approaches to this question
based on considering the energy minimization
at particular families of perturbations,
but the scientific relevance of these conclusions
has never been justified;
see the review and references
in e.g. \cite{linear-a,2014arXiv1405.5547S}.
The linear (spectral) stability
of the nonlinear Dirac equation
is still being settled.
According to
\cite{linear-a},
the linear stability properties of solitary waves
in the nonrelativistic limit
of the nonlinear Dirac equation
(solitary waves with $\omega\lesssim m$)
are similar to linear stability of nonlinear Schr\"odinger
equation;
in particular, the stability
of the \emph{ground states}
(no-node solutions)
is described by the Vakhitov--Kolokolov stability criterion
\cite{VaKo},
$\p\sb\omega Q(\omega)<0$, with
$Q(\omega)=\norm{\phi\sb\omega}_{L^2}^2$
the charge of a solitary wave.
Away from the nonrelativistic limit,
the border of the instability region
can be indicated by the conditions
$\p\sb\omega Q(\omega)=0$
or $E(\omega)=0$ (the value of the energy functional
at a solitary wave),
see \cite{2013arXiv1306.5150C}.
The instability could also develop
from the bifurcation
of the quadruple of complex eigenvalues
from the embedded thresholds $\pm i(m+\abs{\omega})$
as in \cite{PhysRevLett.80.5117},
which in particular
can take place
at the collision of thresholds
at $\lambda=\pm i m$ when $\omega=0$
as in \cite{MR1897705}.
We do not have a good criterion
when such bifurcation takes place.

Let us mention that our results are at odds
with the numerical simulations in \cite{2014arXiv1405.5547S}
which are interpreted as
instability of the cubic Gross--Neveu model ($k=1$)
for $\omega\le\omega_c\approx 0.56$,
of the quintic model ($k=2$)
for $\omega\le\omega_c\approx 0.92$,
and of the $k=3$ model for all $\omega<m$.
We expect that the observed instability is related
to the boundary effects, when certain harmonics,
instead of being dispersed, are reflected into the
bulk of the solution, where
the nonlinearity creates higher harmonics;
this process keeps repeating,
and eventually the space-time discretization
becomes insufficient.
This explanation is corroborated by the fact
that the characteristic instability times
\emph{grow almost proportionally
with the size of the domain}
(see the
instability times for the one-humped solitary wave
with $k=1$, $\omega=0.5$
in \cite[TABLE II]{2014arXiv1405.5547S}),
suggesting the link not to the linear instability
but to the boundary contribution.
Our numerics show no complex
eigenvalues away from the union
of real and imaginary axes
in the Gross--Neveu model with $1\le k\le 9$.
The presence of real eigenvalues
(as on Figure~\ref{fig-gn-3})
agrees with the Vakhitov--Kolokolov stability criterion,
$dQ(\omega)/d\omega>0$.

\medskip

The approach in our paper is standard,
being based on modulation equations,
dispersive wave decay estimates,
and the Strichartz inequalities.
Instead of explaining our approach,
we provide a detailed outline of the paper,
which will elucidate the main steps and ideas involved in the proof.
In Section~\ref{sec:1.1},
we describe the Gross--Neveu model
and formulate our main results.
In Section~\ref{sec:2}, we describe the standing wave solutions of the GN model, as well as the linearized operator around the solitary wave for the corresponding nonlinear evolution.
Here, we provide numerics, which suggest that, at least for certain range of the parameters, we have a favorable for us spectral picture:
that is, {\it the absence of unstable spectrum,
as well as the absence of marginally stable point spectrum, except at zero}.
Section~\ref{sec:4} is the most challenging from a technical point of view.
Therein, we develop the spectral theory for the linearized operator.
We use the four linearly independent Jost solutions
to construct the resolvent explicitly.
This allows us to obtain (among other things) a limiting absorption principle for the linearized operator
(Proposition~\ref{prop9}),
which is crucial for the types of estimates required to establish asymptotic stability.
(Let us mention a related result \cite{MR2857360}
on local energy decay for the Dirac equation on one dimension,
which we will also need.)
In Section~\ref{disp-section}, we use the spectral theory developed in the previous section to establish various dispersive estimates for the linearized Dirac evolution semigroup.
Namely, we establish weighted decay estimates, which in turn imply Strichartz estimates.
We also state and prove estimates between Strichartz spaces and weighted $L^\infty_t L^2$ spaces -- in all this, we have been greatly helped by the Christ--Kiselev lemma and Born expansions.
In Section~\ref{sec:pmt},
we set up the modulation equations for the residuals/radiation term.
We follow this by the fixed point argument in the appropriate spaces,
which finally shows well-posedness for small data for the equation of the residuals.

\section{Main results}
\label{sec:1.1}

The generalized Soler model
(classical fermionic field with scalar self-interaction)
corresponds to the Lagrangian density
\begin{equation}\label{gn-lagrangian}
\mathcal{L}
=\bar\psi(i\gamma\sp\mu\p\sb\mu-m)\psi
+F(\bar\psi\psi),
\qquad
\psi(x,t)\in\C^N,
\quad
x\in\R^n,
\end{equation}
where
$F\in C^\infty(\R)$, $F(0)=0$,
\begin{equation}\label{def-psi-bar}
\bar\psi
=\psi\sp\ast\gamma\sp 0,
\end{equation}
and $\gamma\sp\mu$, $0\le\mu\le n$,
are the Dirac gamma-matrices:
\[
\gamma\sp\mu\gamma\sp\nu
+\gamma\sp\nu\gamma\sp\mu
=2h\sp{\mu\nu}I_n,
\qquad
0\le \mu,\,\nu\le n,
\]
with
$h\sp{\mu\nu}=\mathop{\rm diag}[1,-1,\dots,-1]$
(the inverse of) the Minkowski metric tensor
and $I_n$ the identity matrix.
The one-dimensional analogue of \eqref{gn-lagrangian}
with $n=1$, $N=2$
is called the Gross--Neveu model.
The equation of motion corresponding to the Lagrangian \eqref{gn-lagrangian}
is then given by the following
nonlinear Dirac equation:
\begin{equation}\label{GN.eqn}
i\p_t\psi=D_m\psi-f(\psi^*\beta\psi)\beta\psi,
\qquad
\psi(x,t)\in\C^2,
\quad
x\in\R,
\end{equation}
where $f=F'$,
$\alpha=\gamma\sp 0\gamma\sp 1$, $\beta=\gamma\sp 1$,
and
$D_m=-i\alpha\frac{\p}{\p x}+\beta m$
is the Dirac operator,
with $\alpha$, $\beta$ the self-adjoint Dirac matrices
satisfying
\[
\alpha^2=\beta^2=I_2,
\qquad
\alpha\beta+\beta\alpha=0.
\]
A particular choice of the Dirac matrices is irrelevant;
for definiteness, we take
\[
\alpha =-\sigma_2= \begin{bmatrix} 0 & i \\ -i & 0 \end{bmatrix},
\qquad
\beta =\sigma_3= \begin{bmatrix} 1 & 0 \\ 0 & -1 \end{bmatrix}.
\]
Without loss of generality,
we will also assume that the mass is equal to $m=1$.
Then one has
\begin{equation}\label{def-D}
D_m
=-i\alpha\frac{\p}{\p x}+m\beta
=
\begin{bmatrix}
1&\p\sb x
\\
-\p\sb x&-1
\end{bmatrix}.
\end{equation}
The hamiltonian density
derived from the Lagrangian density
\eqref{gn-lagrangian}
is given by
\begin{equation}\label{def-energy-density}
\mathcal{E}(\psi,\dot\psi)
=\frac{\p\mathcal{L}}{\p\dot\psi}\dot\psi-\mathcal{L}.
\end{equation}
The value of the energy functional
\begin{equation}\label{def-energy}
E(\psi)
=\int\sb{\R}\mathcal{E}(\psi,\dot\psi)\,dx
\end{equation}
is (formally) conserved for the solutions to \eqref{GN.eqn}.
Due to the $\mathbf{U}(1)$-invariance
of the Lagrangian \eqref{gn-lagrangian},
the total charge of the solutions to \eqref{GN.eqn},
\begin{equation}\label{def-charge}
Q(\psi)=\int\sb{\R}\psi\sp\ast(x,t)\psi(x,t)\,dx,
\end{equation}
is also (formally) conserved.


Let
\begin{equation}\label{def-x}
X=
\big\{
\phi\in L^2(\R,\C^2)
;\;
\quad \phi_1(x)=\phi_1(-x),
\ \ \phi_2(x)=-\phi_2(-x)
\,\big\}.
\end{equation}

\begin{assumption}\label{ass-1}
Assume that $f\in C^\infty(\R)$
is such that
$f(s)=\mathcal{O}(s^k)$ for $s\in[0,1]$,
with $k\ge 2$,
and that there is an open interval
$\varOmega$,
\[
\varOmega\subset\Big(-\frac 1 3,\ \frac 1 3\Big),
\]
such that the following takes place:
\begin{enumerate}
\item
For each $\omega\in\varOmega$,
there are solitary wave solutions
$\psi\sb\omega(x,t)=\phi\sb\omega(x)e^{-i\omega t}$,
$\phi\sb\omega\in H^1(\R,\C^2)$,
to \eqref{GN.eqn},
with the map
$\varOmega\to H^1$,
$\omega\mapsto \phi\sb\omega$
being $C^2$.
\item
Non-degeneracy:
\[
\p\sb\omega
Q(\omega)\ne 0,
\qquad
\omega\in\varOmega.
\]
Here
$Q(\omega)$
is the value
of the charge functional \eqref{def-charge}
evaluated at the solitary wave $\phi\sb\omega(x)e^{-i\omega t}$.
\item
The linearization of \eqref{GN.eqn}
at a solitary wave with $\omega\in\varOmega$
has no eigenvalues with nonzero real part
and no purely imaginary eigenvalues $\lambda\in i\R$
with eigenfunctions from $X$
(of the same parity as $\phi\sb\omega$),
and no resonances
at $\lambda=1\pm\abs{\omega}$
with generalized eigenfunctions
of the same parity as $\phi\sb\omega$.
\item
For $\omega\in\varOmega$,
the Evans function
$E(\lambda,\omega)$
of the linearization operator
does not vanish at $\lambda\in i\R$
with $\abs{\lambda}\ge 1-\abs{\omega}$.
\end{enumerate}
\end{assumption}

The following theorem is the main result of our paper.

\begin{theorem}[Asymptotic stability
of solitary waves in nonlinear Dirac equation]
\label{theorem-main}
Assume that Assumption~\ref{ass-1} holds.
Let $\omega_0 \in\varOmega$
and $\phi_{\omega_0}(x)e^{-i\omega\sb 0 t}$
be the corresponding solitary wave
with $\phi_{\omega_0}\in X\cap H^1(\R,\C^2)$.
There exist $\epsilon >0$ and $C<\infty$ such that if
$\psi_0\in X$ satisfies
\[
\inf_{\gamma \in [0,2\pi]}\norm{\psi_0-e^{i\gamma}\phi_{\omega_0}}_{H^1}
\leq \epsilon^2,
\]
then the solution $\psi$ of \eqref{GN.eqn}
with $\psi\at{t=0}=\psi_0$
exists globally in time and satisfies the estimate
\[
\lim_{t \rightarrow \infty}\norm{\psi(\cdot, t)-e^{-i\theta(t)}\phi_{\omega_\infty}(\cdot)-
e^{-i D_m t}h_{\infty}(\cdot)}_{H^1} =0,
\]
for some $\omega_\infty \in\varOmega$,
$\theta \in C^1(\R, \R)$ and $h_\infty \in H^1(\R,\C^2)$,
with $\norm{h_\infty}_{H^1} \leq C\epsilon$
and $|\omega_0-\omega_\infty| \leq C \epsilon$.
\end{theorem}

\begin{remark}
The precise structure of the nonlinearity
of the Gross--Neveu model,
$f(\psi\sp\ast\beta\psi)\beta\psi$,
does not play any particular role in
our considerations.
In fact, because of the eigenvalues $\pm 2\omega i$
of the linearized operator,
which are specific for this model
(to avoid the associated problems,
we need to restrict to $\abs{\omega}<1/3$
and only consider perturbations from $X$).
Yet, we choose this model
since it is the focus of many other recent papers.
\end{remark}

\begin{remark}
Assumption~\ref{ass-1}
is satisfied, for example,

\begin{enumerate}
\item
For the Gross--Neveu model
with $f(s)=s^2$ and
$\varOmega=(0.23,0.33)$
(see Fig.~\ref{fig-gn-2});
\item
For the Gross--Neveu model
with $f(s)=s^3$ and
$\varOmega=(0.14,0.33)$
(see Fig.~\ref{fig-gn-3}).
\end{enumerate}
We also mention
that in the Gross--Neveu model with $k=1,\,2,\,\dots,\,9$,
we found no complex eigenvalues
for the linearizations at solitary waves with
$\omega=0.1,\,0.2,\,\dots,\,0.9$
in the domain
$0.0008<\abs{\Re\lambda}<0.59$,
$\abs{\Im\lambda}<2.5$.
Moreover,
according to \cite{linear-a},
the bifurcations of point eigenvalues off the imaginary axis
could result only from the collision
of purely imaginary eigenvalues
or from eigenvalues embedded into the continuous spectrum,
and also from resonances
at the embedded thresholds, $\lambda=\pm i(m+\abs{\omega})$
(in one-dimensional case, the resonances correspond to
the generalized, $L^\infty$ eigenfunctions).
Our numerics show that there are no resonances
at the embedded thresholds
in the Gross--Neveu model with $k=2$ and $k=3$
for all $\omega\in(0,m)$,
justifying the observed absence
of complex eigenvalues away from $\R\cup i\R$.
\end{remark}

\begin{remark}
The solitary waves
to classical Gross--Neveu model ($k=1$, cubic nonlinearity)
are known to be linearly stable
\cite{MR2892774}
but our argument does not apply to this situation.
\end{remark}

\begin{remark}
The assumption $k\ge 2$
allows us to close the argument
in Section~\ref{sect-further-linear-estimates}
using the Strichartz estimates,
making the argument sufficiently compact.
Similar requirements
on the order of vanishing of the nonlinearity
being sufficiently high
are common in the research
on asymptotic stability
of solitons in nonlinear Schr\"odinger equation,
starting with the seminal papers
\cite{MR1221351,MR1199635}.
\end{remark}

\begin{figure}
\begin{center}
\includegraphics[width=10cm,height=7cm]{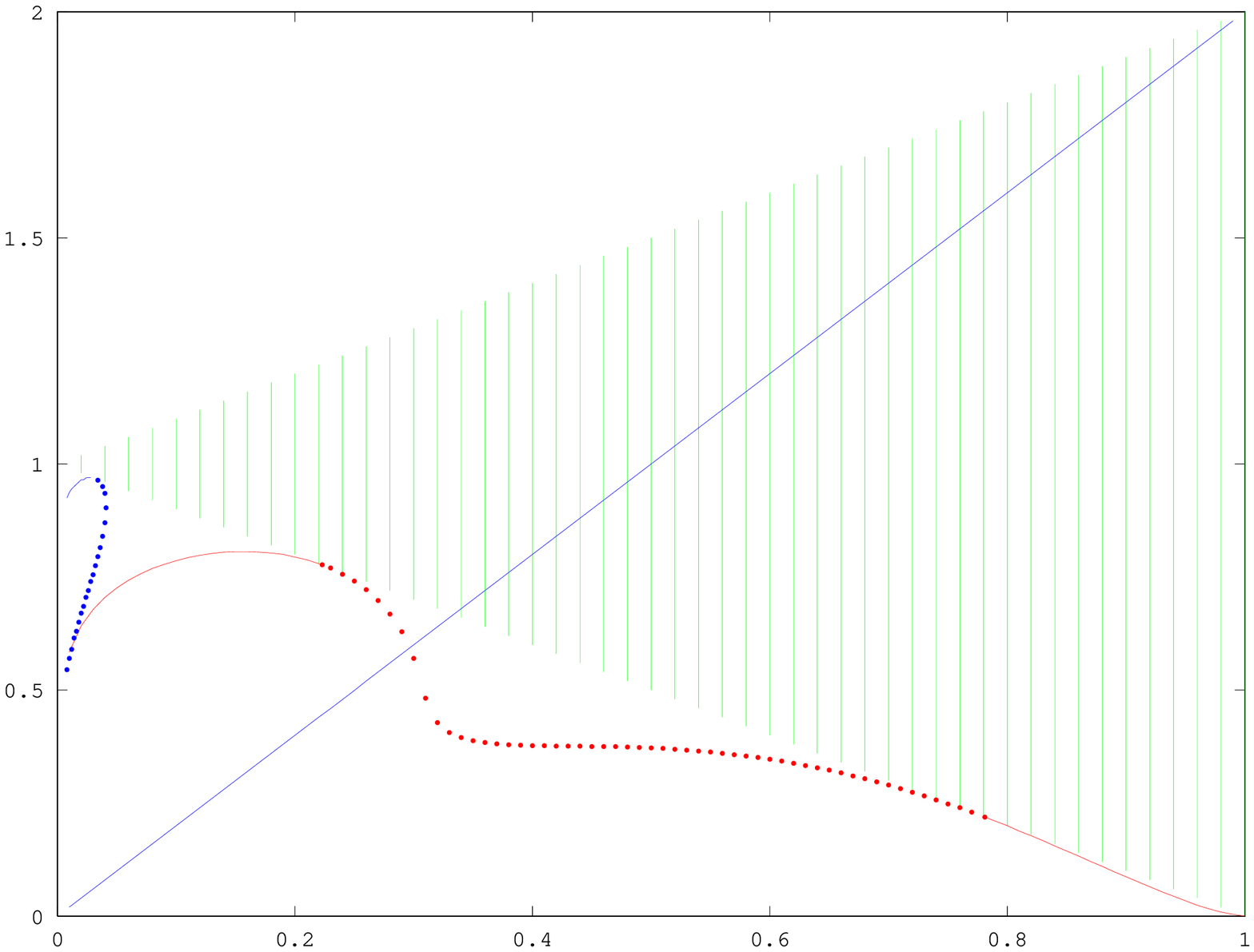}
\caption{Gross--Neveu model, $k=2$ (the quintic case).
Linearization at a solitary wave.
Horizontal axis: $\omega\in(0,1)$.
Vertical axis: spectrum on the upper half of the
imaginary axis.
Solid vertical (green) lines: part of the continuous spectrum
between the threshold $i(1-|\omega|)$ and the embedded threshold $i(1+|\omega|)$.
Solid red curves: eigenvalues
with eigenfunctions from $\bX$
(of the same parity as $\phi\sb\omega$; see \eqref{def-xx}),
which we can not ignore; our result
holds in the regions where such eigenvalues are absent.
Solid blue curve (near $\omega=0$ and $\lambda=i$)
and the line $\lambda=2\omega i$ denote
eigenvalues
with eigenfunctions from $\bX\sp\perp$
(see \eqref{def-xx-perp}),
which remain orthogonal to our perturbation.
\hfill
\break
Dotted red and blue curves: antibound states
of different parity (from $\bX$ and $\bX\sp\perp$);
we do not mention them in the argument.
Antibound states correspond to
zeros of Evans functions on the ``wrong'' Riemann sheet,
which corresponds to generalized eigenfunctions
with exponential growth at infinity.
}
\label{fig-gn-2}

\bigskip
\bigskip

\includegraphics[width=10cm,height=7cm]{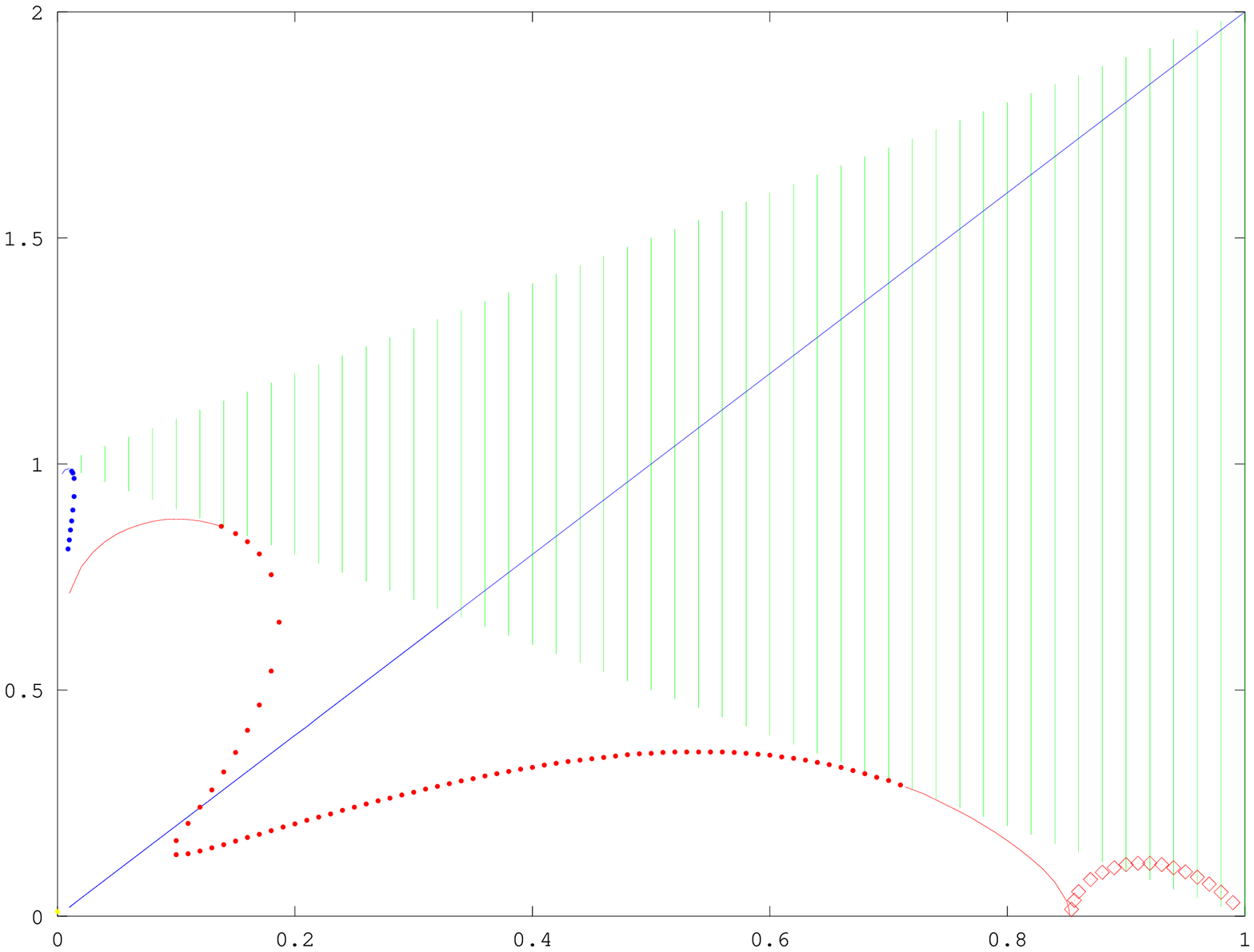}
\caption{Gross--Neveu model, $k=3$.
Hollow red diamonds (on bottom right)
denote positive eigenvalues (thus linear
instability) present in the spectrum
for $\omega\in(0.85,1)$.
These eigenvalues are overimposed on the imaginary axis.
Theorem~\ref{theorem-main} on asymptotic stability
applies for solitary waves with $\omega$ such that
there are neither
hollow red diamonds (linear instability)
nor
solid red curves
(purely imaginary eigenvalues
with eigenfunctions from $\bX$)
in the spectrum.
Note that the dotted kink
indicates collision of antibound states at $\omega_b\approx 0.1$
on the imaginary axis and their bifurcation
off the imaginary axis for $\omega<\omega_b$.
(Location of  these values of $\lambda$
off the imaginary axis does not lead to instability
since the corresponding antibound states have infinite $L^2$-norm.)
}
\label{fig-gn-3}
\end{center}
\end{figure}


\begin{remark}
We discuss
the necessary and sufficient conditions
for the existence and properties of solitary waves
in the generalized Gross--Neveu model
in Section~\ref{sect-sw-gn}.
\end{remark}

\begin{remark}
By \cite{dirac-vk,2013arXiv1306.5150C},
the assumptions
$E(\omega)\ne 0$
and $\p\sb\omega Q(\omega)\ne 0$
guarantee that the generalized null space
of the linearization operator
is (exactly) four-dimensional.
(Above, $E(\omega)$ and $Q(\omega)$
are the values of the energy and charge functionals
\eqref{def-energy}
\eqref{def-charge}
at the solitary wave $\phi\sb\omega e^{-i\omega t}$.)
We do not need to impose the condition
$E(\omega)\ne 0$
since although the vanishing of $E(\omega)$
leads to the increase of the Jordan block
of the linearization operator,
this increase is absent when we restrict
the operator to the subspace $X$.
\end{remark}

\begin{remark}
The restriction of the linearized
operator to $X$
requires some details.
Since this operator is $\R$-linear but not $\C$-linear,
by its restriction to $X$
we imply its restriction
onto $\bX=\C\otimes\sb{\R}(\Re X\times \Im X)$;
see \eqref{def-xx} below.
More details are in Section~\ref{sect-linearization}.
\end{remark}

\begin{remark}
In Assumption~\ref{ass-1}, we require that
$\varOmega\subset(-1/3,1/3)$
to avoid the situation when the eigenvalues
$\pm 2\omega i$
(see Remark~\ref{remark-2-omega-i} below)
become embedded
into the essential spectrum
($\lambda\in i\R$, $\abs{\lambda}\ge 1-\abs{\omega}$).
In that case, our construction
of the resolvent in Section~\ref{sect-resolvent}
does not allow to
obtain the necessary estimates.
Yet, the restriction to $\abs{\omega}<1/3$
seems to be merely technical;
we still expect that
for $1/3\le \abs{\omega}<1$,
the resolvent
of the linearized operator
restricted to $X$
has the same properties as stated
in Proposition~\ref{prop9}
even in the vicinity of the embedded eigenvalues $\pm 2\omega i$
and that the asymptotic stability could be proved.
\end{remark}

\begin{remark}
We expect that the Evans function
never has zeros at $\lambda\in i\R$,
$\abs{\Im\lambda}\ge 1+\abs{\omega}$,
but could not prove this.
Instead, we check this assumption
numerically;
all the zeros of the Evans function
which we found are plotted
as solid curves
on Figures~\ref{fig-gn-2} and~\ref{fig-gn-3}
(these zeros correspond to the point eigenvalues
of the linearized operator).
The absence of zeros of the Evans function
for $\lambda\to \pm i\infty$ follows from
Lemma~\ref{lemma-large-lambda}.
\end{remark}

\begin{remark}
Let us summarize that
most of our assumptions are technical;
the only essential assumption is
that the spectrum of the linearized operator
has no eigenvalues in the right half-plane
and that the Jordan block of $\lambda=0$
is (exactly) four-dimensional.
We expect that the presence of purely imaginary eigenvalues
does not lead to instability
unless these eigenvalues are of higher algebraic multiplicity.
More generally, we expect that,
similarly to the case of the
nonlinear Schr\"odinger equation and similar systems,
the (dynamic) instability
takes place when either there is a linear instability
or when the eigenvalues on the imaginary axis
are of higher algebraic multiplicities
(when we are at the threshold of linear instability).
\end{remark}

\section{Solitary waves in generalized Gross--Neveu model}
\label{sec:2}

\subsection{Properties of solitary waves}
\label{sect-sw-gn}

Equation \eqref{GN.eqn} can be written explicitly as
\begin{equation}\label{Dirac.ex}
\left \{
\begin{array}{ll}
i \p_t \psi_1 & = \p_x \psi_2 + \psi_1-f(|\psi_1|^2-|\psi_2|^2)\psi_1, \\
i \p_t\psi_2 & =-\p_x\psi_1-\psi_2 + f(|\psi_1|^2-|\psi_2|^2) \psi_2.
\end{array} \right.
\end{equation}
In the abstract form, we write \eqref{GN.eqn} as
\begin{equation}\label{N.eqn.gn}
i\p_t \psi = D_m \psi + \bN(\psi),
\end{equation}
with the Dirac operator
\[
D_m
=-i\alpha\p\sb x+\beta
=\begin{bmatrix}
1 & \p_x \\ -\p_x & -1
\end{bmatrix}
\]
and the nonlinearity
\begin{equation}\label{def-n-gn}
\bN(\psi) =
\begin{bmatrix}
-f(|\psi_1|^2-|\psi_2|^2)&0
\\0&f(|\psi_1|^2-|\psi_2|^2)
\end{bmatrix} \psi.
\end{equation}

\begin{definition}\label{def-sw}
Solitary waves are solutions of the form
\begin{equation}\label{sw}
\psi\sb\omega(x,t)
=\phi_{\omega}(x)e^{-i\omega t},
\qquad\phi_{\omega}\in H^1(\R,\C^2),
\qquad\omega\in\R.
\end{equation}
\end{definition}

Substituting this Ansatz into \eqref{N.eqn.gn},
we see that $\phi_\omega$ solves
\begin{equation}\label{phi.eqn.gn}
\omega\phi_\omega
=
D_m\phi_\omega + \bN(\phi_\omega).
\end{equation}
The existence of solitary waves follows
from \cite{MR847126,MR2892774}:

\begin{proposition}\label{Solitons}
Let $F(s)$ be the antiderivative of $f(s)$ such that $F(0)=0$.
Assume that for given
$\omega \in\R$, $0 < \omega < 1$, there exists
$\varGamma_\omega > 0$ such that
\[
\omega \varGamma_\omega = \varGamma_\omega-F(\varGamma_\omega),
\quad \omega \ne 1-f(\varGamma_\omega), \qquad
\omega s < s-F(s), \qquad \text{for} \quad s \in (0, \varGamma_\omega).
\]
Then there is a solitary wave solution
$\psi\sb\omega(x,t) = \phi_\omega(x) e^{-i\omega t}$, where
\begin{equation}\label{sw1}
\phi_\omega(x) = \begin{bmatrix}
v(x,\omega) \\ u(x,\omega)
\end{bmatrix}, \qquad v(\cdot,\omega),\,u(\cdot,\omega)\in H^1(\R).
\end{equation}
This solution is unique if we require that
$v$, $u$ are real-valued, $v$ even and positive, and $u$ odd.
Both
$v$ and $u$ are exponentially decaying as $|x|\to\infty$
and satisfy
$\abs{u(x,\omega)}<\abs{v(x,\omega)}$, $x\in\R$.

Moreover, there is $c\sb\omega<\infty$
such that
\begin{equation}\label{phi-bounded.gn}
\abs{\phi\sb\omega(x)}
\le c\sb\omega e^{-\updelta_\omega\abs{x}},
\qquad
x\in\R,
\end{equation}
where
\begin{equation}\label{def-up-delta-omega}
\updelta_\omega
=\sqrt{1-\omega^2}.
\end{equation}
Similarly, there is $c\sb\omega<\infty$
such that
\begin{equation}\label{phi-bounded.gn-2}
\abs{\p\sb\omega\phi\sb\omega(x)}
\le c\sb\omega\langle x\rangle e^{-\updelta_\omega\abs{x}},
\qquad
\abs{\p\sb\omega^2\phi\sb\omega(x)}
\le c\sb\omega\langle x\rangle^2 e^{-\updelta_\omega\abs{x}},
\qquad
x\in\R.
\end{equation}
\end{proposition}

\begin{proof}
The proof is given in e.g. \cite[Lemma 3.2]{MR2892774}.
The sharp rate of decay \eqref{phi-bounded.gn}
can be proved as in e.g. \cite[Appendix A]{MR2318373}.
The bounds on
$\abs{\p\sb\omega\phi\sb\omega(x)}$
and
$\abs{\p\sb\omega^2\phi\sb\omega(x)}$
follow from differentiating
\eqref{phi.eqn.gn} with respect to $\omega$.
\end{proof}

\subsection{Linearization at a solitary wave}
\label{sect-linearization}

To study stability of a solitary wave $\phi_\omega(x)e^{-i\omega t}$,
with
$\phi_\omega(x)=\begin{bmatrix}v(x,\omega)\\u(x,\omega)\end{bmatrix} \in\R^2$,
we consider the solution in the form
\[
\psi(x,t)=\big(\phi_\omega(x)+\rho(x,t)\big)e^{-i\omega t},
\qquad
\rho(x,t)
\in\C^2.
\]
Substituting this Ansatz into \eqref{N.eqn.gn}, we obtain:
\begin{equation}\label{rho.eqn.gn}
i\p_t \rho = (D_m -\omega I_2)\rho
+\bN(\phi_\omega+\rho)-\bN(\phi_\omega).
\end{equation}
The linearization of \eqref{rho.eqn.gn}
can be written as follows:
\begin{equation}\label{R.eqn.gn}
\dot R
=\bJ\bL R,
\qquad
R=\begin{bmatrix}\Re\rho\\\Im\rho\end{bmatrix}\in\R^4,
\end{equation}
where
\begin{equation}
\label{bL.def.gn}
\bJ=\begin{bmatrix}0&I_2\\-I_2&0\end{bmatrix},
\qquad
\bL(\omega)=\bD_m-\omega I_4+\bW(x,\omega),
\end{equation}
with
\begin{eqnarray}\label{def-w.gn}
&
\bW(x,\omega)
=\begin{bmatrix}W_1(x,\omega)&0\\0&W_0(x,\omega)\end{bmatrix},
\\
&
W_0(x,\omega)
=
\begin{bmatrix}-f(v^2-u^2)&0\\0&f(v^2-u^2)\end{bmatrix},
\quad
W_1(x,\omega)=
W_0(x,\omega)
-
2f'(v^2 -u^2)\begin{bmatrix} v^2&-v u\\-v u&u^2\end{bmatrix}.
\nonumber
\end{eqnarray}
The free Dirac operator takes the form
\begin{equation}\label{def-DD}
\bD_m=\bJ\bmupalpha\p\sb x+\bmupbeta,
\end{equation}
with
\begin{equation}\label{def-alpha-beta}
\bmupalpha
=\begin{bmatrix}\Re\alpha&-\Im\alpha\\\Im\alpha&\Re\alpha\end{bmatrix}
=\begin{bmatrix}0&\Im\sigma_2\\-\Im\sigma_2&0\end{bmatrix},
\qquad
\bmupbeta
=\begin{bmatrix}\Re\beta&-\Im\beta\\\Im\beta&\Re\beta\end{bmatrix}
=\begin{bmatrix}\sigma_3&0\\0&\sigma_3\end{bmatrix};
\end{equation}
$\bJ$, $\bmupalpha$, and $\bmupbeta$
represent $-i$, $\alpha$, and $\beta$ when acting on
$\begin{bmatrix}\Re\psi\\\Im\psi\end{bmatrix}$,
with $\psi\in\C^2$.
We then have
\begin{equation}
\bD_m
=\begin{bmatrix}D_m&0\\0&D_m\end{bmatrix},
\qquad
\mbox{where}
\quad
D_m
=\begin{bmatrix}1&\p\sb x\\-\p\sb x&-1\end{bmatrix}.
\end{equation}
Note that since $v,\,u$ both depend on $\omega$,
the potentials $W_1, W_0$ also depend on it.
We will often omit this dependence in our notations.

\begin{lemma}
\label{lemma10}
There is $C\sb\omega<\infty$
such that
the matrix-valued potential $\bW$
satisfies
\begin{equation}\label{w-small}
\norm{\bW(x,\omega)}\sb{\C^4\to\C^4}
\le C\sb\omega e^{-2k\abs{x}\updelta_\omega},
\qquad
x\in\R.
\end{equation}
\end{lemma}

\begin{proof}
This bound is an immediate consequence of
the exponential decay of $\phi\sb\omega$ in 
Proposition~\ref{Solitons}
(see \eqref{phi-bounded.gn}),
the assumption
$f(s)=\mathcal{O}(s^k)$,
and \eqref{def-w.gn}.
\end{proof}

\begin{lemma}
\[
\sigma\sb{\rm ess}(\bJ\bL)
=i\R\backslash\big(-i(1-\abs{\omega}),i(1-\abs{\omega})\big).
\]
\end{lemma}

\begin{proof}
This is an immediate consequence
of Weyl's theorem on the essential spectrum.
\end{proof}

Denote
\begin{equation}\label{def-upphi}
\bmupphi(x)
=\bmupphi\sb\omega(x)
=\begin{bmatrix}\Re\phi\sb\omega(x)\\\Im\phi\sb\omega(x)\end{bmatrix}
=\begin{bmatrix}\phi\sb\omega(x)\\0\end{bmatrix}.
\end{equation}
Thanks to the invariance of \eqref{phi.eqn.gn} with respect to
the phase rotation and the translation, we have
\[
\bJ\bL\bJ\bmupphi
=0,
\qquad
\bJ\bL
\p\sb x\bmupphi
=0.
\]
Analyzing the Jost solutions of
\begin{equation}\label{def-l0-l1}
L_1(\omega)=D_m-\omega I_2+W_1,
\qquad
L_0(\omega)=D_m-\omega I_2+W_0
\end{equation}
(for each of $L_1$ and $L_0$,
there are two Jost solutions: one decreasing and one increasing),
one concludes that
the null space of $\bL$ is given by
\begin{equation}\label{n-l}
N(\bL)=
\left(
\bJ\bmupphi,\ \p\sb x\bmupphi
\right).
\end{equation}
Moreover,
\begin{equation}\label{jordan-block-1}
\bJ\bL
\p\sb\omega\bmupphi
=\bJ\bmupphi,
\end{equation}
\begin{equation}\label{jordan-block-2}
\bJ\bL
\Big(
\omega x\bJ\bmupphi-\frac 1 2\bmupalpha\bmupphi
\Big)
=
\p\sb x\bmupphi,
\end{equation}
where
\[
\omega x\bJ\bmupphi-\frac 1 2\bmupalpha\bmupphi
=
\begin{bmatrix}
0
\\
\frac i 2\alpha\phi-\omega x\phi
\end{bmatrix}
=
\begin{bmatrix}
0
\\
-\frac i 2\sigma_2\phi-\omega x\phi
\end{bmatrix}.
\]
Therefore,
\begin{equation}
\left \{
\bJ\bmupphi,
\ \p\sb x\bmupphi,
\ \p\sb\omega\bmupphi,
\ \ \omega x\bJ\bmupphi-\frac 1 2\bmupalpha\bmupphi
\right \} \subset N_g(\bJ\bL).
\end{equation}
By \cite{2013arXiv1306.5150C},
if $\p\sb\omega Q(\omega)$
and $E(\omega)\ne 0$,
then the above vectors form a basis in the generalized null space
$N\sb g(\bJ\bL)$:
\begin{equation}\label{ng}
N_g(\bJ\bL)
=
\mathop{\rm Span}
\left(
\bJ\bmupphi,
\ \p\sb x\bmupphi,
\ \p\sb\omega\bmupphi,
\ \ \omega x\bJ\bmupphi-\frac{1}{2}\bmupalpha\bmupphi
\right).
\end{equation}

Following the definition \eqref{def-x}, we define
\begin{equation}\label{def-xx}
\bX = \left \{\psi\in L^2(\R, \C^4)
;\;
\ \psi_k(x)=\psi_k(-x),\ \ k=1,\,3;
\quad
\psi_k(x)=-\psi_k(-x),\ \ k=2,\,4
\right \};
\end{equation}
\begin{equation}\label{def-xx-perp}
\bX\sp\perp = \left \{\psi\in L^2(\R, \C^4)
;\;
\ \psi_k(x)=\psi_k(-x),\ \ k=2,\,4;
\quad
\psi_k(x)=-\psi_k(-x),\ \ k=1\,3
\right \}.
\end{equation}
    From now on, we shall
restrict $\bJ\bL(\omega)$ to $\bX$.
This restriction has the following
null space and generalized null space:
\begin{equation}\label{n-ng}
N(\bJ\bL\at{\bX})
=
\mathop{\rm Span}
\left(
\bJ\bmupphi
\right),
\qquad
N_g(\bJ\bL\at{\bX})
=
\mathop{\rm Span}
\left(
\bJ\bmupphi,
\ \p\sb\omega\bmupphi
\right).
\end{equation}
The linearization operator
$\bJ\bL$ acts invariantly in $\bX$ and in $\bX\sp\perp$.

\begin{remark}\label{remark-2-omega-i}
The restriction of $\bJ\bL(\omega)$ onto $\bX$
allows one to exclude certain eigenvalue directions,
significantly simplifying the problem.
In particular,
by \cite{dirac-vk}, one has
\begin{equation}\label{2-omega-i}
\bJ\bL
\begin{bmatrix}\sigma_1\phi\\i\sigma_1\phi\end{bmatrix}
=2i\omega
\begin{bmatrix}\sigma_1\phi\\i\sigma_1\phi\end{bmatrix},
\qquad
\bJ\bL
\begin{bmatrix}\sigma_1\phi\\-i\sigma_1\phi\end{bmatrix}
=-2i\omega
\begin{bmatrix}\sigma_1\phi\\-i\sigma_1\phi\end{bmatrix},
\end{equation}
where $\sigma_1$ is the Pauli matrix;
this shows that
$\pm 2\omega i\in\sigma\sb{p}(\bJ\bL(\omega))$.
On the other hand, the restriction to $\bX$
satisfies
$\pm 2\omega i\not\in\sigma_d(\bJ\bL\at{\bX})$.
\end{remark}

Since $(\bJ\bL)^* = -\bL\bJ$, it follows
from \eqref{jordan-block-1}, \eqref{jordan-block-2}
that the corresponding generalized kernel for the adjoint is
\[
\bX_g((\bJ\bL)^*)
=N_g((\bJ\bL)^*)\cap \bX
=\left\{
\ \bJ
\p\sb\omega\bmupphi, \bmupphi
\right\}.
\]
We decompose the space $\bX$ as follows:
\begin{equation}\label{def-x-decomposition}
\bX = \bX_g(\bJ\bL)
\oplus \bX_c(\bJ\bL) ,
\quad \text{where} \quad
\bX_c(\bJ\bL)
=
\bX_g((\bJ\bL)^*)^\perp.
\end{equation}
The subspaces $\bX_g(\bJ\bL)$
and $\bX_c(\bJ\bL)$ are invariant under the action of $\bJ\bL$,
and any $R_1\in\bX_g(\bJ\bL)$, $R_2\in\bX_c(\bJ\bL)$
satisfy the following symplectic orthogonality condition:
\[
\wei{\bJ R_1, R_2} =0.
\]
It then follows that any $R\in\bX$
can be uniquely decomposed into
\begin{equation}\label{r-rd-rc}
R
=2\frac{\wei{\bmupphi, R}}{Q'(\omega)}\p\sb\omega\bmupphi
+2\frac{\wei{\bJ\p\sb\omega\bmupphi,R}}{Q'(\omega)}\bJ\bmupphi
+U,
\qquad U\in\bX_c(\bJ\bL),
\end{equation}
where $Q(\omega)$
is the charge functional \eqref{def-charge}
evaluated at $\phi\sb\omega e^{-i\omega t}$.
Thus, a vector function $U\in X_c(\bJ\bL)$ satisfies
the following two symplectic orthogonality conditions:
\begin{equation}\label{symplectic-orth}
\langle\bmupphi,U\rangle=0,
\qquad
\langle\bJ\p\sb\omega\bmupphi,U\rangle=0.
\end{equation}

\begin{remark}
Note that $Q'(\omega)\ne 0$
by Assumption~\ref{ass-1}.
\end{remark}

Let $P_d(\omega)$ denote the symplectically orthogonal
projection onto the generalized null space of $\bJ\bL(\omega)$
restricted onto the space $\bX$ from \eqref{def-xx}.
By \eqref{r-rd-rc},
\begin{equation}\label{def-pd}
P_d(\omega)R
=2\frac{\wei{\bmupphi, R}}{Q'(\omega)}\p\sb\omega\bmupphi
+2\frac{\wei{\bJ\p\sb\omega\bmupphi,R}}{Q'(\omega)}\bJ\bmupphi,
\end{equation}
while the projection onto $\bX_c$ is
\begin{equation}\label{def-pc}
P_c(\omega)
=1-P_d(\omega).
\end{equation}

\section{Spectral theory for the linearized operator}
\label{sec:4}

In this section, we consider dispersive estimates for the
complexification of the linearized equation \eqref{R.eqn.gn},
\begin{equation}\label{R.eqn}
\p_t R = \bJ\bL R,
\qquad
R\in\C^4.
\end{equation}
More precisely, we will show that similarly to the free Dirac evolution, the linear evolution of \eqref{R.eqn}
projected onto the continuous spectrum of $\bJ\bL$
scatters the initial data.
This phenomena in the related Schr\"odinger equation context manifests itself in a variety of useful estimates; see for example the work of Mizumachi \cite{MR2511047}.

Before proceeding to specific estimates for the solution of \eqref{R.eqn.gn}, let us take a moment to properly define
$e^{\bJ\bL t}P_c$.
Since
\[
\sigma\sb{\mathrm{ess}}(\bJ\bL(\omega))
=(-i\infty,-i(1-\abs{\omega})]\cup[i(1-\abs{\omega}),i\infty),
\]
we define $e^{\bJ\bL(\omega)t}P_c(\omega)$
by the following Cauchy formula:
\begin{eqnarray}\label{e-p-f}
e^{\bJ\bL t}P_c(\omega) f
&=&
-\frac{1}{2\pi i}
\oint\sb\Gamma R_{\bJ\bL}(\lambda)f\,d\lambda
\nonumber
\\
&=&-\frac{1}{2\pi i}
\left(\int_{-i\infty}^{-i(1-\abs{\omega})}+\int_{i(1-\abs{\omega})}^{i\infty}\right)
e^{\lambda t}
\Big(
\big[R_{\bJ\bL}^{+}(\lambda)-R_{\bJ\bL}^{-}(\lambda)\big] f
\Big)
\, d\lambda
\nonumber
\\
&=&-\frac{1}{2\pi}
\left(\int_{-\infty}^{\abs{\omega}-1}+\int_{1-\abs{\omega}}^{+\infty}\right)
e^{i\Lambda t}
\Big(
\big[R_{\bJ\bL}^{+}(i\Lambda)-R_{\bJ\bL}^{-}(i\Lambda)\big] f
\Big)
\, d\Lambda,
\end{eqnarray}
where
$\Gamma$ is a positively-oriented contour
around the essential spectrum of $\bJ\bL$.
For $\lambda\in i\R$
the operators
\[
R_{\bJ\bL}^\pm(\lambda):=\lim_{\varepsilon\to 0+}(\bJ\bL-(\lambda\pm \varepsilon))^{-1}
\]
are to be interpreted in a certain appropriate sense
(for example, as operators from $L^2_\alpha\to L^2_{-\alpha}$,
for certain $\alpha>0$, by the limiting absorption principle).

\subsection{The Jost solutions
and the Evans function
of the linearization operator $\bJ\bL$}

The eigenvalue problem for
the operator $\bJ\bL(\omega)$,
\[
\bJ(\bD_m-\omega+\bW(x,\omega))\psi
=\bJ(\bJ\bmupalpha\p\sb x+\bmupbeta-\omega+\bW(x,\omega))\psi
=\lambda\psi,
\]
can be rewritten as
\begin{equation}\label{paw}
(\p\sb x
-\bmupalpha\bJ\bmupbeta
+\omega\bmupalpha\bJ
+\bmupalpha\lambda
-\bmupalpha\bJ\bW(x,\omega)
\big)\psi=0.
\end{equation}
The construction of Jost solutions
is based on considering solutions
to the constant coefficient equation
\begin{equation}\label{def-M0}
(\p\sb x-\mathcal{M}_0(\lambda,\omega)
\big)\psi
=0,
\qquad
\mathcal{M}_0(\lambda,\omega):=
\bmupalpha\bJ\bmupbeta
-\omega\bmupalpha\bJ-\bmupalpha\lambda,
\end{equation}
and using the Duhamel representation to
construct solutions
to equation \eqref{paw} with variable coefficients,
written in the form
\begin{equation}\label{def-M}
\big(\p\sb x-\mathcal{M}(x,\lambda,\omega)\big)\psi=0,
\qquad
\mathcal{M}(x,\lambda,\omega)
:=\mathcal{M}_0(\lambda,\omega)
+\bmupalpha\bJ\bW(x,\omega).
\end{equation}

\begin{lemma}\label{lemma-m-zero}
Let $\omega\in[-1,1]$, $\lambda\in\C$.
Then the eigenvalues of $\mathcal{M}_0(\lambda,\omega)$
are given by
\begin{equation}\label{lemma-m-zero-eigenvalues}
\sigma(\mathcal{M}_0(\lambda,\omega))
=\{\pm\sqrt{1-(\omega\pm i\lambda)^2}\}.
\end{equation}
These eigenvalues satisfy
\begin{equation}\label{lemma-m-zero-re}
\sup\sb{\lambda\in i\R}
\big\{
\abs{\Re\zeta};\ \zeta\in\sigma(\mathcal{M}_0(\lambda,\omega))
\big\}
=1;
\qquad
\sup\sb{\lambda\in\sigma\sb{\mathrm{ess}}(\bJ\bL(\omega))}
\big\{
\abs{\Re\zeta};\ \zeta\in\sigma(\mathcal{M}_0(\lambda,\omega))
\big\}
=2\sqrt{\omega-\omega^2}.
\end{equation}
\end{lemma}

\begin{proof}
We need to find all $z\in\C$ such that
\[
\mathcal{M}_0(\lambda,\omega)-z
=\bmupalpha\bJ\bmupbeta-\omega\bmupalpha\bJ-\bmupalpha\lambda-z
=-\bmupalpha\lambda-\omega\bmupalpha\bJ-z+\bmupalpha\bJ\bmupbeta
\]
is degenerate.
Multiplying the above matrix in the right-hand side
by $-\bmupalpha\bJ$,
we need to find out when the matrix
\[
\bJ\lambda-\omega+\bmupalpha\bJ z+\bmupbeta
\]
is degenerate.
Since $\bmupalpha\bmupbeta$
anticommutes with both $\bmupalpha$ and $\bmupbeta$,
while $\det\bmupalpha=\det\bmupbeta=1$,
one has:
\[
\det(\bJ\lambda-\omega+\bmupalpha\bJ z+\bmupbeta)^2
=
\det
\big(
(\bJ\lambda-\omega+\bmupalpha\bJ z+\bmupbeta)
\bmupalpha\bmupbeta
(\bJ\lambda-\omega+\bmupalpha\bJ z+\bmupbeta)
\bmupbeta\bmupalpha
\big)
\]
\[
=
\det
\big(
(\bJ\lambda-\omega+\bmupalpha\bJ z+\bmupbeta)
(\bJ\lambda-\omega-\bmupalpha\bJ z-\bmupbeta)
\big)
=\det
\big(
(\bJ\lambda-\omega)^2
-(-z^2+1)
\big).
\]
Since $\sigma(\bJ)=\{\pm i\}$,
we conclude that
the above determinant vanishes
(hence $z\in\sigma(\mathcal{M}_0(\lambda,\omega))$)
if and only if
\[
z^2-1+(\pm i\lambda-\omega)^2=0.
\]
The conclusion about the spectrum of
$\mathcal{M}_0$ follows.

Other statements are checked by direct
computation.
\end{proof}

Due to the symmetry of the potential $\bW$
(see \eqref{w-symmetry} below),
we have the following results.

\begin{lemma}\label{lemma-paw}
If $\psi(x)$ solves \eqref{def-M} for $\lambda\in\C$,
then $\theta(x)= \bmupbeta\psi(-x)$ also solves \eqref{def-M}
for the same $\lambda\in\C$.
\end{lemma}

\begin{proof}
Since $v$ is even and $u$ is odd,
and since $\beta=\sigma_3$ anticommutes with $\sigma_1$,
there are the relations
\begin{equation}\label{w-symmetry}
W_0(x)\beta=\beta W_0(-x),
\qquad
W_1(x)\beta=\beta W_1(-x),
\end{equation}
for $W_0$, $W_1$
from
the Gross--Neveu model \eqref{def-w.gn}.
(It is convenient to notice that
for each of these models,
$W_0$ and $W_1$ can be written as linear combinations
of the form $w_a(x)\sigma_1+w_b(x)\sigma_3+w_c(x)I_2$,
with scalar-valued functions
$w_b$ and $w_c$ symmetric in $x$
and $w_a$ skew-symmetric.)
The conclusion follows.
\end{proof}

\begin{lemma}\label{lemma-trace-zero}
For any $x\in\R$, $\omega\in\varOmega$,
and $\lambda\in\C$,
the matrix $\mathcal{M}$ from \eqref{def-M}
satisfies
one has
$\tr
\mathcal{M}(x,\lambda,\omega)=0$.
\end{lemma}

\begin{proof}
The statement is immediate for
all the terms from $\mathcal{M}_0$
(Cf. \eqref{def-M0}).
The remaining relation
$\tr\bmupalpha\bJ\bW=0$
is checked with the explicit expressions
\eqref{def-w.gn}.
\end{proof}

\bigskip

We now turn to the construction of the Jost solutions,
which are defined as eigenfunctions of $\bJ\bL(\omega)$
with the same asymptotic behavior as
eigenfunctions of $\bJ(\bD_m-\omega)$.
To do this, for $\lambda\in\C$, we first define
\begin{equation}\label{xi12.def}
\xi_1(\lambda,\omega)=\sqrt{(\omega-i\lambda)^2-1},
\qquad \xi_2(\lambda,\omega)=\sqrt{(\omega+i\lambda)^2-1},
\end{equation}
so that
$\sigma(\mathcal{M}_0(\lambda,\omega))
=\{\pm\xi\sb 1(\lambda,\omega),\,
\pm\xi\sb 2(\lambda,\omega)\}
$
(Cf. Lemma~\ref{lemma-m-zero}).
Without loss of generality, we will only consider the case
\begin{equation}\label{omega-lambda-positive}
\omega\ge 0,
\qquad
\Re\lambda\le 0,
\qquad
\Im\lambda\ge 0;
\end{equation}
in each of the two square roots in \eqref{xi12.def},
we choose the branch that is positive
for $\lambda\in i\R$, $\Im\lambda\gg 1$.

We define
\begin{equation}\label{def-xi1-eta1}
\Xi_1(\lambda,\omega)
=
\frac{1}{c_1(\lambda,\omega)}
\begin{bmatrix}
i\xi_1
\\
-i\lambda-1+\omega
\\
-\xi_1
\\
\lambda-i(1-\omega)
\end{bmatrix},
\qquad
\Eta_1(\lambda,\omega)
=
\frac{1}{c_1(\lambda,\omega)}
\begin{bmatrix}
i\xi_1
\\
i\lambda+1-\omega
\\
-\xi_1
\\
-\lambda+i(1-\omega)
\end{bmatrix},
\end{equation}
\begin{equation}\label{def-xi2-eta2}
\Xi_2(\lambda,\omega)
=
\frac{1}{c_2(\lambda,\omega)}
\begin{bmatrix}
i\xi_2
\\
i\lambda-1+\omega
\\
\xi_2
\\
\lambda+i(1-\omega)
\end{bmatrix},
\qquad
\Eta_2(\lambda,\omega)
=
\frac{1}{c_2(\lambda,\omega)}
\begin{bmatrix}
i\xi_2
\\
-i\lambda+1-\omega
\\
\xi_2
\\
-\lambda-i(1-\omega)
\end{bmatrix},
\end{equation}
with the constants
\begin{equation}\label{def-c1-c2}
c_1(\lambda,\omega)>0,
\qquad
c_2(\lambda,\omega)>0
\end{equation}
chosen so that
$\abs{\Xi\sb j}=\abs{\Eta\sb j}=1$, $\ j=1,\,2$.
Note that
$\Eta\sb j
=
\bmupbeta
\Xi\sb j$;
$j=1,\,2$.
The functions
\[
\Xi\sb 1(\lambda,\omega)e^{i\xi\sb 1(\lambda,\omega) x},
\qquad
\Xi\sb 2(\lambda,\omega)e^{i\xi\sb 2(\lambda,\omega) x},
\qquad
\Eta\sb 1(\lambda,\omega)e^{-i\xi\sb 1(\lambda,\omega) x},
\qquad
\Eta\sb 2(\lambda,\omega)e^{-i\xi\sb 2(\lambda,\omega) x}
\]
satisfy the equation $(\bJ(\bD_m-\omega)-\lambda)\psi(x)=0$
(and thus \eqref{def-M0}).

By \eqref{omega-lambda-positive}, we see that
\[
\xi_1>\xi_2\ge 0
\quad \text{for} \quad
\lambda\in i\R,\ \ \abs{\lambda}\ge 1+\abs{\omega};
\qquad
\abs{\xi_2}>\abs{\xi_1}
\quad \text{for} \quad
\lambda\in i\R,\ \ \abs{\lambda}\le 1-\abs{\omega}.
\]
We denote
\begin{equation}\label{def-kappa}
\kappa_1=\abs{\Im\xi_1},
\qquad
\kappa_2=\abs{\Im\xi_2};
\end{equation}
then one has
\[
\kappa_2>\kappa_1\ge 0
\qquad
\mbox{for}
\quad
\lambda\in i\R,
\quad
\abs{\lambda}\le 1-\abs{\omega}.
\]

\begin{proposition}\label{prop-jost}
Let $\omega\in\varOmega$.
There are the Jost solutions
$\f\sb j(x,\lambda,\omega)$,
$\g\sb j(x,\lambda,\omega)$,
$\F\sb j(x,\lambda,\omega)$,
$\G\sb j(x,\lambda,\omega)$,
$j=1,\,2$,
which satisfy the equation
$(\bJ\bL(\omega)-\lambda)u=0$
and have the following properties.
There is $c(\omega)<\infty$
such that
\begin{itemize}
\item
For
$\lambda\in i\R$,
$\abs{\lambda}\le 1-\abs{\omega}$,
\begin{equation}\label{p10}
\abs{e^{\kappa_j x}\f_j(x,\lambda,\omega)-\Xi_j(\lambda,\omega)}
+
\abs{e^{-\kappa_j x}\F_j(x,\lambda,\omega)-\Eta_j(\lambda,\omega)}
\le
c(\omega)e^{-2k\updelta_\omega x},
\qquad
x\ge 0,
\qquad
j=1,\,2.
\end{equation}
\item
For
$\lambda\in i\R$,
$1-\abs{\omega}\le\abs{\lambda}\le 1+\abs{\omega}$,
\[
\abs{e^{-i\xi_1 x}\f_1(x,\lambda,\omega)-\Xi_1(\lambda,\omega)}
+
\abs{e^{i\xi_1 x}\F_1(x,\lambda,\omega)-\Eta_1(\lambda,\omega)}
\le
c(\omega)e^{-2k\updelta_\omega x},
\qquad
x\ge 0,
\]
\[
\abs{e^{\kappa_2 x}\f_2(x,\lambda,\omega)-\Xi_2(\lambda,\omega)}
+
\abs{e^{-\kappa_2 x}\F_2(x,\lambda,\omega)-\Eta_2(\lambda,\omega)}
\le
c(\omega)e^{-2k\updelta_\omega x},
\qquad
x\ge 0.
\]
\item
For
$\lambda\in i\R$,
$\abs{\lambda}\ge 1+\abs{\omega}$,
\[
\abs{e^{-i\xi_j x}\f_1(x,\lambda,\omega)-\Xi_j(\lambda,\omega)}
+
\abs{e^{i\xi_j x}\F_1(x,\lambda,\omega)-\Eta_j(\lambda,\omega)}
\le
c(\omega)e^{-2k\updelta_\omega x},
\qquad
x\ge 0,
\qquad
j=1,\,2.
\]
\item
For $\lambda\in i\R$,
\[
\abs{\f_j(x,\lambda,\omega)}
+
\abs{\F_j(x,\lambda,\omega)}
\le
c(\omega)
\big(
\langle x\rangle+e^{\kappa_2\abs{x}}
\big),
\qquad
x\le 0,
\qquad
j=1,\,2.
\]
\item
For $\lambda\in i\R$,
$\abs{\lambda}\ge 3$,
\begin{equation}\label{f-f-c}
\abs{\f_j(x,\lambda,\omega)}
+
\abs{\F_j(x,\lambda,\omega)}
\le
c(\omega),
\qquad
x\in\R,
\qquad
j=1,\,2.
\end{equation}
\item
One can define the Jost solutions
with appropriate asymptotics
as $x\to-\infty$ by
\begin{equation}\label{f-is-g}
\g\sb j(x)
=
\bmupbeta
\f\sb j(-x),
\qquad
\G\sb j(x)
=
\bmupbeta
\F\sb j(-x),
\qquad
x\in\R,
\qquad
j=1,\,2.
\end{equation}
\end{itemize}
\end{proposition}

Above,
$\updelta_\omega=\sqrt{1-\omega^2}$
(Cf. \eqref{def-up-delta-omega}).

\begin{proof}
The proof is quite standard.
However,
since the decay rate of the potential $\bW$
depends on $\omega$ and $k$ (Cf. Assumption~\ref{ass-1}),
we choose to provide the details.
Given $\varkappa\in\sigma(\mathcal{M}_0(\lambda,\omega))$,
with $\omega\in\varOmega$ and $\lambda\in i\R$,
let
$\Xi\in\C^4$ be a corresponding eigenvector,
with $\abs{\Xi}=1$.
To find a solution
$\psi(x)\sim \Xi e^{\varkappa x}$,
$x\to+\infty$ of \eqref{def-M0}, we
define $\xi(x)$ by
\[
\psi(x)=e^{\varkappa x}\xi(x),
\qquad \text{so that} \quad
\xi\at{x=+\infty}=\Xi;
\]
then
$
\p\sb x\xi
=(\mathcal{M}_0-\varkappa)\xi+\bmupalpha\bJ\bW\xi,
$ and hence we can write
\[
\p\sb x(e^{-(\mathcal{M}_0-\varkappa)x}\xi)
=e^{-(\mathcal{M}_0-\varkappa)x}\bmupalpha\bJ \bW\xi.
\]
We construct $\xi(x)$ in the form of the power series
$
\xi(x)=\sum\sb{n=0}\sp\infty\xi_n(x),
$
where
$\xi_0=\Xi$
and
\[
\p\sb x(e^{-(\mathcal{M}_0-\varkappa)x}\xi_n(x))
=e^{-(\mathcal{M}_0-\varkappa)x}\bmupalpha\bJ\bW(x,\omega)\xi_{n-1}(x),
\qquad
\xi_n(+\infty)=0,
\qquad
n\ge 1;
\]
hence
\[
\xi_n(x)
=
-\int_x^{+\infty}
e^{(\mathcal{M}_0-\varkappa)(x-y)}
\bmupalpha\bJ\bW(y,\omega)\xi_{n-1}(y)\,dy,
\qquad n\ge 1.
\]
Let $P\sb\zeta$
denote the Riesz projector onto the eigenspace
corresponding to $\zeta\in\sigma(\mathcal{M}_0)$.
Then, for $x\ge 0$,
\begin{eqnarray}\label{int-y}
\abs{P\sb\zeta \xi\sb n(x)}
&\le&
\int_x^{+\infty}
\norm{P\sb\zeta e^{(\mathcal{M}_0-\varkappa)(x-y)}}
\norm{\bW(y,\omega)}\sb{\End(\C^4)}
\abs{\xi_{n-1}(y)}\,dy
\nonumber
\\
&\le&
\sup\sb{y\ge x}\abs{\xi_{n-1}(y)}
\int_x^{+\infty}
a
e^{(x-y)\Re(\zeta-\varkappa)}
\langle x-y\rangle
K e^{-2k\updelta_\omega y}\,dy
\le
c
e^{-2k\updelta_\omega x}
\sup\sb{y\ge x}\abs{\xi_{n-1}(y)},
\end{eqnarray}
for some $c=c(\omega,K)<\infty$.
Above, we used the bound
$\norm{P_\zeta e^{(\mathcal{M}_0-\varkappa)x}}
\le a e^{x\Re(\zeta-\varkappa)}\langle x\rangle$,
with some $a<\infty$
(which
depends on $\omega$ but
does not depend on $\zeta$),
with the factor $\langle x\rangle$
due to the possibility of the Jordan block of $\mathcal{M}_0$
(when $\zeta=0$).
For the convergence of the integration in $y$,
we used the bound \eqref{w-small}
and the inequalities
\begin{equation}\label{t-i}
\abs{\Re\zeta}\le k\updelta_\omega,
\qquad
\abs{\Re\varkappa}\le k\updelta_\omega,
\end{equation}
which are trivially satisfied under conditions
of Assumption~\ref{ass-1}:
one has $k\ge 2$, $\abs{\omega}<1/3$,
hence $k\updelta_\omega\ge 2\sqrt{8/9}$,
while $\zeta\in\sigma(\mathcal{M}_0(\lambda,\omega))$
for any $\lambda\in i\R$, $\omega\in(-1,1)$,
satisfy $\abs{\Re\zeta}\le 1$
(Cf. Lemma~\ref{lemma-m-zero}).


Then the integration in $y$
in \eqref{int-y}
can be estimated as follows:
\[
\int_x^{+\infty} e^{(x-y)\Re(\zeta-\varkappa)}
\langle x-y\rangle e^{-2k\updelta_\omega y}
\,dy
=e^{-2k\updelta_\omega x}\int_0^{+\infty} e^{-z\Re(\zeta-\varkappa)}
\langle z\rangle e^{-2k\updelta_\omega z}
\,dz
\]
\[
\le
e^{-2k\updelta_\omega x}
\Big(
\frac{1}{2(k-1)\updelta_\omega}
+\frac{1}{(2(k-1)\updelta_\omega)^2}
\Big).
\]
We conclude that
\[
\sup\sb{y\ge x}\abs{\xi\sb n(y)}
=
\sup\sb{y\ge x}
\Abs{
\sum\sb{\zeta\in\sigma(\mathcal{M}_0)}
P\sb\zeta\xi\sb n(y)
}
\le
\sum\sb{\zeta\in\sigma(\mathcal{M}_0)}
\sup\sb{y\ge x}\abs{P\sb\zeta\xi\sb n(y)}
\le
4c
e^{-2k\updelta_\omega x}\sup\sb{y\ge x}\abs{\xi\sb{n-1}(y)},
\qquad
x\ge 0.
\]
Therefore, there is $C<\infty$
such that
$
\sum\sb{n=1}\sp{\infty}\abs{\xi\sb n(x)}
\le
C e^{-2k\updelta_\omega x},
$
for all $x\ge 0$,
hence
\[
\abs{\xi(x)-\Xi}
=C e^{-2k\updelta_\omega x},
\qquad
x\ge 0.
\]

Let us prove the uniform bounds \eqref{f-f-c}.
Let us write \eqref{def-M} in the form
\begin{equation}\label{def-M1}
(\p\sb x-\mathcal{M}_0(\lambda,\omega))\psi
=\bmupalpha\bJ\bW(x,\omega)\psi.
\end{equation}
Using the Green function for the operator
$\p\sb x-\mathcal{M}_0(\lambda,\omega)$,
which is given by
\[
\mathscr{G}(x,y,\lambda,\omega)
=
\left(\Xi_1\otimes \theta_1\sp\ast e^{i(x-y)\xi_1}
+\Xi_2\otimes \theta_2\sp\ast e^{i(x-y)\xi_2}
+\Eta_1\otimes\eta_1\sp\ast e^{-i(x-y)\xi_1}
+\Eta_2\otimes\eta_2\sp\ast e^{-i(x-y)\xi_2}
\right)
\Theta(x-y),
\]
where
$\Theta$ is the Heaviside step-function
and
$\theta_j,\,\eta_j\in\C^4$, $j=1,\,2$, is the basis
dual to $\Xi_j,\,\Eta_j\in\C^4$,
one can construct the solutions
$\f_j(x,\lambda,\omega)$,
$\F_j(x,\lambda,\omega)$,
in the form of the
power series
\begin{equation}\label{series}
\psi=\sum\sb{n=0}\sp{\infty}\psi_n,
\end{equation}
with $\psi_0(x)=\Xi_j e^{i\xi_j x}$
or $\psi_0(x)=\Eta_j e^{-i\xi_j x}$
(according to \eqref{f-is-g},
these are asymptotics of $\f_j(x),\,\F_j(x)$ for $x\gg 1$),
and with $\psi_n(x)$, $n\ge 1$ solving
\[
(\p\sb x-\mathcal{M}_0(\lambda,\omega))\psi_n(x)
=\bmupalpha\bJ\bW(x,\omega)\psi_{n-1}(x).
\]
For definiteness, we will consider $\f_1(x)$ only
(other functions are considered in the same way).
For any $x\in\R$,
the series
\eqref{series}
converges due to the estimate
\begin{eqnarray}
\abs{\psi_n(x)}
&\le&
\int_x^{\infty}
\norm{\bW(x_1)}\sb{\End(\C^4)}
\abs{\psi_{n-1}(x_1)}\,dx_1
\nonumber
\\
&\le&
\int_{x}^{\infty}\int_{x_1}^{\infty}
\Norm{\bW(x_1)}\sb{\End(\C^4)}
\Norm{\bW(x_2)}\sb{\End(\C^4)}
\abs{\psi_{n-2}(x_2)}\,dx_1\,dx_2
\le\dots
\nonumber
\\
&\le&
\mathop{\int\dots\int}\sb{x<x_1<\dots<x_n<\infty}
\Big(\prod\sb{l=1}\sp{n}
\norm{\bW(x_l)}\sb{\End(\C^4)}\Big)
\abs{\psi_0(x_n)}
\,dx_1\dots dx_n
\nonumber
\\
&\le&
\frac{1}{n!}
\int\sb{x_1>x}\dots \int\sb{x_n>x}
\Big(\prod\sb{l=1}\sp{n}
\norm{\bW(x_l)}\sb{\End(\C^4)}\Big)
\abs{\psi_0(x_n)}
\,dx_1\dots dx_n
\le
\frac{(\int_{x}^{\infty}\norm{\bW(y)}\sb{\End(\C^4)}\,dy)^n}{n!},
\nonumber
\end{eqnarray}
where we represented the integration over the simplex
$x<x_1<\dots<x_n<\infty$
in $\R^n$
as a fraction of the integration over the quadrant
$x_l>x$, $1\le l\le n$,
and substituted
$\abs{\psi_0(x_n)}=\abs{\Xi_1}=1$.
Therefore,
$
\abs{\psi(x)}\le\sum\sb{n\ge 0}\abs{\psi_n(x)}
\le \mathop{\rm exp}\left(\int_{\R}
\norm{\bW(y)}\sb{\End(\C^4)}\,dy\right),
$
for any $x\in\R$.
This proves \eqref{f-f-c}.

Finally, Lemma~\ref{lemma-paw}
allows to use \eqref{f-is-g} to obtain
the Jost solutions with required asymptotic
behaviour at $-\infty$.
\end{proof}

\begin{definition}\label{def-evans}
We define the Evans function by
\begin{equation}\label{def-Z}
E(\lambda,\omega)
=
\det\big[\f_1( x,\lambda,\omega),\f_2( x,\lambda,\omega),
\g_1( x,\lambda,\omega),\g_2( x,\lambda,\omega)\big].
\end{equation}
By Lemma~\ref{lemma-trace-zero}
and Liouville's formula,
the right-hand side of \eqref{def-Z}
does not depend on $x\in\R$.
\end{definition}

\begin{lemma}\label{lemma-fg}
Fix
$\omega\in\varOmega$.
\begin{enumerate}
\item
\label{lemma-fg-1}
Let $\lambda\in i\R$,
$\abs{\lambda}\in(1-\abs{\omega},1+\abs{\omega})$.
Then
$
E(\lambda,\omega)
=0
$
at some $\lambda\in i\R$,
$\abs{\lambda}\in(1-\abs{\omega},1+\abs{\omega})$,
if and only if
$\lambda$ is an $L^2$ eigenvalue of $\bJ\bL$.

\item
\label{lemma-fg-2}
At $\lambda=\pm i(1 +\abs{\omega})$,
one has
$E(\lambda,\omega)=0$
only if
there is a generalized $L\sp\infty$-eigenfunction
corresponding to $\lambda$,
which has the asymptotics
$\psi\sim a\Xi_2$ as $x\to+\infty$,
$\psi\sim b\Eta_2$ as $x\to-\infty$.
\end{enumerate}
\end{lemma}

\begin{remark}
The statement of the lemma at the thresholds
is non-trivial
since at the threshold points the solution to
$(\bJ\bL-\lambda)\psi=0$
which is bounded for $x\to +\infty$
could be linearly growing as $x\to-\infty$.
\end{remark}

\begin{proof}
Let us prove Part~\ref{lemma-fg-1}.
Consider the case
$\pm\lambda\in i(1-\abs{\omega},1+\abs{\omega})$.
Due to the asymptotics of the Jost solutions,
if $\f_1$ and $\g_1$ are linearly dependent,
then $\f_1=C\g_1=\psi$ is the exponentially decaying
solution to \eqref{def-M}
and thus $\lambda$ is an $L^2$ eigenvalue.
This proves the ``if'' statement of the lemma.

Let us prove the converse statement.
If $\det[\f_1,\f_2,\g_1,\g_2]=0$
for some $\lambda\in i\R$,
then there are
$a\sb 1,\,a\sb 2,\,b\sb 1,\,b\sb 2\in\C$,
not all of them equal to zero, one has
\begin{equation}\label{def-varphi}
\varPhi(x):=
\sum\sb{j=1}\sp 2 a\sb j \f\sb j(x,\lambda,\omega)
=\sum\sb{j=1}\sp 2 b\sb j \g\sb j(x,\lambda,\omega),
\qquad
x\in\R.
\end{equation}
Clearly, $\varPhi$ thus defined
is not identically zero.

Define
\[
\varSigma
=i\bJ
=
\begin{bmatrix}0&i I_2\\-i I_2&0\end{bmatrix}.
\]
Let us consider the auxiliary Dirac equation
\begin{equation}\label{aux-eqn}
i\varSigma\p\sb t\Psi
=\bL\Psi,
\qquad
\Psi(x,t)\in\C^4,
\quad x\in\R,
\end{equation}
where
$
\bL=\bJ\bmupalpha\p\sb x+\bmupbeta+\bW-\omega.
$
This is a Hamiltonian system
with the Hamiltonian density
\[
\eur{h}
=\Psi\sp\ast\bL\Psi
=\Psi\sp\ast(\bJ\bmupalpha\p\sb x+\bmupbeta+\bW-\omega)\Psi
\]
and the Lagrangian density
\[
\eur{l}
=\Psi\sp\ast(i\varSigma\p\sb t-\bL)\Psi.
\]

If $\varPhi\in C^1(\R,\C^4)$ satisfies
$\lambda\varPhi=\bJ\bL\varPhi$,
which we write as
$(i\lambda)(i\bJ)\varPhi
=
-\lambda\bJ\varPhi=\bL\varPhi$,
then we have
\[
\Omega\varSigma\varPhi=\bL\varPhi,
\qquad
\Omega:=i\lambda\in\R.
\]
Thus,
$\Psi(x,t)=\varPhi(x)e^{-i\Omega t}$
is a ``solitary wave solution'' to \eqref{aux-eqn},
except that $\varPhi$ is not necessarily in $L^2$.

Equation \eqref{aux-eqn} conserves
\emph{the Krein charge};
its density is
\begin{equation}\label{aux-krein-charge-density}
\sigma(x,t)=\Psi\sp\ast\varSigma\Psi
=\varPhi\sp\ast\varSigma\varPhi,
\end{equation}
while the density of the corresponding current is
\begin{equation}\label{aux-krein-current-density}
\eur{j}(x,t)=\Psi\sp\ast\varSigma\bmupalpha\Psi
=\varPhi\sp\ast\varSigma\bmupalpha\varPhi.
\end{equation}

\begin{remark}
We call the quantity
$\langle\Psi,\varSigma\Psi\rangle$
the ``Krein charge''
in view of its relation
to the Krein index considerations.
Namely, the relation $\bJ\bL\varPhi=\lambda\varPhi$ implies that
$\langle\varPhi,\bL\varPhi\rangle
=-\lambda\langle\varPhi,\bJ\varPhi\rangle$,
with $\langle\varPhi,\bL\varPhi\rangle$
real and $\langle\varPhi,\bJ\varPhi\rangle$
purely imaginary;
hence $\Re\lambda\ne 0$ leads to
$\langle\varPhi,\bJ\varPhi\rangle=0$,
$\langle\varPhi,\bL\varPhi\rangle=0$.
Thus, the Krein signature is zero
($\bL$ is not sign-definite
on the corresponding eigenspace)
for any eigenvalue away from the imaginary axis.
(The above could also be interpreted as follows.
We could say that
if $\Psi=\varPhi(x)e^{-i\Omega t}$
(with $\Omega=i\lambda$)
is a solitary wave solution
to \eqref{aux-eqn} and $\Im\Omega=\Re\lambda\ne 0$,
then the conservation of the ``Krein charge''
$\langle\Psi(t),\varSigma\Psi(t)\rangle
=\langle\varPhi,\varSigma\varPhi\rangle e^{2\Im\Omega t}$
requires that this charge is zero,
$\langle\varPhi,\varSigma\varPhi\rangle=0$.)
It follows that purely imaginary eigenvalues
$\lambda\in i\R\setminus 0$
with nonzero Krein  signature,
$\langle\varPhi,i\bJ\varPhi\rangle\ne 0$,
can not bifurcate off the imaginary axis
into the complex plane.
\end{remark}



Since the Krein charge density does not depend on time,
the local conservation of the Krein charge
in the system \eqref{aux-eqn}
leads to the equality of the Krein current
\eqref{aux-krein-current-density}
evaluated at the endpoints of the interval $(-l,l)$, $l>0$.
Therefore, taking into account that
\[
\varSigma\Xi_1=-\Xi_1,
\qquad
\varSigma\Xi_2=\Xi_2,
\qquad
\varSigma\Eta_1=-\Eta_1,
\qquad
\varSigma\Eta_2=\Eta_2,
\]
we compute
for $\varPhi$ from \eqref{def-varphi}:
\begin{eqnarray}\label{krein-current-conservation}
0
=\lim\sb{l\to+\infty}
\varPhi\sp\ast\varSigma\bmupalpha\varPhi\vert\sp l\sb{-l}
=
\lim\sb{l\to+\infty}
\big(
(a_1\Xi_1e^{i\xi_1 x}
+
a_2\Xi_2 e^{-\kappa_2 x})\sp\ast
\bmupalpha
(
-a_1\Xi_1 e^{i\xi_1 x}
+
a_2\Xi_2 e^{-\kappa_2 x}
)
\big)
\vert\sb{x=l}
\nonumber
\\
-
\lim\sb{l\to+\infty}
\big(
(b_1\Eta_1 e^{-i\xi_1 x}
+
b_2\Eta_2 e^{-\kappa_2\abs{x}})\sp\ast
\bmupalpha
(
-b_1\Eta_1 e^{-i\xi_1 x}
+
b_2\Eta_2 e^{-\kappa_2\abs{x}}
)
\big)
\vert\sb{x=-l}
\nonumber
\\
=
\lim\sb{l\to+\infty}
\big(
(a_1\Xi_1e^{i\xi_1 x}
+
a_2\Xi_2 e^{-\kappa_2 x})\sp\ast
\bmupalpha
(
-a_1\Xi_1 e^{i\xi_1 x}
+
a_2\Xi_2 e^{-\kappa_2 x}
)
\big)
\vert\sb{x=l}
\nonumber
\\
-
\lim\sb{l\to+\infty}
\big(
(b_1\Xi_1 e^{-i\xi_1 x}
+
b_2\Xi_2 e^{-\kappa_2\abs{x}})\sp\ast
\bmupalpha
(
b_1\Xi_1 e^{-i\xi_1 x}
-
b_2\Xi_2 e^{-\kappa_2\abs{x}}
)
\big)
\vert\sb{x=-l}.
\end{eqnarray}
In the last relation, we took into account that
$\Eta_j=\bmupbeta\Xi_j$
and that $\bmupbeta$ anticommutes with
$\bmupalpha$.
Taking into account that
$\Xi_1\sp\ast\bmupalpha\Xi_2
=\Xi_2\sp\ast\bmupalpha\Xi_1
=0$,
we rewrite the above as
\begin{eqnarray}\label{krein-current-conservation-2}
0=
(\abs{a_1}^2
+\abs{b_1}^2
)\Xi_1\sp\ast\bmupalpha\Xi_1
+(\abs{a_2}^2+\abs{b_2}^2
)
\Xi_2\sp\ast\bmupalpha\Xi_2
\lim\sb{l\to+\infty}e^{-2\kappa_2 l}.
\end{eqnarray}
For Part~\ref{lemma-fg-1},
when $\lambda\in i\R$
and
$1-\abs{\omega}<\abs{\lambda}<1+\abs{\omega}$,
one has $\kappa_2>0$,
hence the second term in the right-hand side of
\eqref{krein-current-conservation-2}
vanishes.
On the other hand,
\[
\Xi\sb 1\sp\ast\bmupalpha\Xi\sb 1
=
\frac{4i(\lambda-i(1-\omega))\xi_1}{c_1^2}>0
\qquad
\mbox{for}
\quad
\lambda\in i\R,\ \ \Im\lambda>1-\abs{\omega}.
\]
Then it follows from \eqref{krein-current-conservation-2}
that
$a_1=b_1=0$,
and we conclude that
$\varPhi$ is exponentially decaying for $x\to\pm\infty$,
so that $\lambda$ is an $L^2$ eigenvalue.
This finishes the proof of Part~\ref{lemma-fg-1}.

For Part~\ref{lemma-fg-2},
when $\lambda=\pm (1+\abs{\omega})i$,
one has $\kappa_2=0=\xi_2$,
and, using \eqref{def-xi2-eta2},
one computes
\[
\Xi\sb 2\sp\ast\bmupalpha\Xi\sb 2
=
\frac{4i(\lambda+i(1-\omega))\xi_2}{c_2^2}
=
0;
\]
therefore,
\eqref{krein-current-conservation-2}
again yields $a_1=b_1=0$.
\end{proof}

For $\lambda\in i\R$, $\abs{\lambda}>1+\abs{\omega}$,
the functions
$\f\sb 1(\lambda,\omega)$, $\f\sb 2(\lambda,\omega)$,
$\F\sb 1(\lambda,\omega)$, $\F\sb 2(\lambda,\omega)$
are linearly independent,
and there are
$A(\lambda,\omega)$,
$B(\lambda,\omega)\in\C^{4\times 4}$,
locally bounded in $\lambda$, $\omega$,
such that
\begin{equation}\label{g-is-f}
\g_j(x,\lambda,\omega)
=
\sum\sb{k=1}\sp{2}
\f_k(x,\lambda,\omega)
A_{kj}(\lambda,\omega)
+
\sum\sb{k=1}\sp{2}
\F_k(x,\lambda,\omega)
B_{kj}(\lambda,\omega),
\qquad
j=1,\,2.
\end{equation}
We note that, by \eqref{f-is-g},
applying $\bmupbeta$ to \eqref{g-is-f} and flipping $x$,
we also have
\begin{equation}\label{f-is-g-2}
\f_j(x,\lambda,\omega)
=
\sum\sb{k=1}\sp{2}
\g_k(x,\lambda,\omega)
A_{kj}(\lambda,\omega)
+
\sum\sb{k=1}\sp{2}
\G_k(x,\lambda,\omega)
B_{kj}(\lambda,\omega),
\qquad
j=1,\,2.
\end{equation}

\begin{lemma}\label{lemma-fg-infty}
For each $\omega\in\varOmega$,
the matrices $A(\lambda,\omega)$, $B(\lambda,\omega)$
from \eqref{g-is-f} satisfy
\begin{equation}\label{A-zero}
\lim\sb{\lambda\to \pm i\infty}\norm{A(\lambda,\omega)}
\sb{\End(\C^4)}=0,
\qquad
\lim\sb{\lambda\to \pm i\infty}B(\lambda,\omega)
=B_\infty(\omega),
\end{equation}
with
$\norm{B_\infty(\omega)}\sb{\End(\C^N)}<\infty$.
Moreover,
\begin{equation}\label{det-b-1}
\det B_\infty(\omega)=1.
\end{equation}
\end{lemma}

\begin{proof}
The bound \eqref{f-f-c}
(which is also 
valid for $\g_j$, $\G_j$
in view of \eqref{f-is-g}),
together with \eqref{g-is-f}
and with the asymptotic behaviour of $\f$, $\F$
for $x\gg 1$
(Proposition~\ref{prop-jost})
and linear independence of $\Xi_j$, $\Eta_j$,
$1\le j\le 2$,
leads to
\[
\lim\sb{\lambda\to\pm i\infty}
(\norm{A(\lambda,\omega)}\sb{\End(\C^2)}
+\norm{B(\lambda,\omega)}\sb{\End(\C^2)})
<\infty.
\]
Following the proof of \eqref{f-f-c}
from Proposition~\ref{prop-jost}
and using the stationary phase method,
which yields
\[
\int_x^{\infty}
\Xi_j\otimes\theta_k\sp\ast e^{i(x-y)\xi_j}\bW(y)
\Eta_l e^{-iy\xi_l}\,dy=\mathcal{O}(\frac{1}{\xi_j})\to 0
\qquad
\mbox{as $\lambda\to\pm i\infty$},
\]
one shows that
$\norm{A(\lambda,\omega)}\sb{\End(\C^4)}\to 0$
as $\lambda\to\pm i\infty$.

Let us show that $\det B\sb\infty(\omega)=1$.
First, we note from
\eqref{def-xi1-eta1}, \eqref{def-xi2-eta2}
that
\[
\lim\sb{\lambda\to\pm i\infty}
\Xi_1(\lambda,\omega)=M
\lim\sb{\lambda\to\pm i\infty}\Eta_2(\lambda,\omega),
\qquad
\lim\sb{\lambda\to\pm i\infty}
\Xi_2(\lambda,\omega)=M
\lim\sb{\lambda\to\pm i\infty}\Eta_1(\lambda,\omega),
\]
where $M=\begin{bmatrix}I_2&0\\0&-I_2\end{bmatrix}$.
Therefore,
taking into account that for each $x\in\R$
one has $\lim\sb{\lambda\to \pm i\infty}
\abs{x}\abs{\xi_1-\xi_2}\to 0$,
we have
$\lim\sb{\lambda\to\pm i\infty}
\norm{
(\f_1,\f_2)-M (\F_2,\F_1)e^{2i\xi_1 x}}\sb{\C^4\times\C^4}
\to 0$,
for each fixed $x\gg 1$, and hence
(due to continuous dependence of solutions to
\eqref{def-M1} on the initial data)
for each fixed $x\in\R$:
\[
\lim\sb{\lambda\to\pm i\infty}
\norm{(\f_1(x,\lambda,\omega),\f_2(x,\lambda,\omega))
-M (\F_2(x,\lambda,\omega),\F_1(x,\lambda,\omega))
e^{2i\xi_1 x}}\sb{\C^4\times\C^4}
=0,
\qquad
x\in\R.
\]
Similarly,
comparing asymptotics for $x\ll -1$,
we conclude that
\[
\lim\sb{\lambda\to\pm i\infty}
\norm{(\G_1(x,\lambda,\omega),\G_2(x,\lambda,\omega))
-M (\g_2(x,\lambda,\omega),\g_1(x,\lambda,\omega))
e^{2i\xi_1 x}}\sb{\C^4\times\C^4}
=0,
\qquad
x\in\R.
\]
Therefore, besides \eqref{g-is-f},
which yields
$\lim\sb{\lambda\to \pm i\infty}
\norm{(\g_1,\g_2)-(\F_1,\F_2)B}_{L^\infty_x}=0$
(due to \eqref{A-zero}),
we also have
\[
\lim\sb{\lambda\to \pm i\lambda}
\norm{(\G_1(x,\lambda,\omega),\G_2(x,\lambda,\omega))
-(\f_1(x,\lambda,\omega),\f_2(x,\lambda,\omega))B(\lambda,\omega)}
\sb{\C^4\times \C^4}=0,
\qquad
x\in\R.
\]
On the other hand,
from \eqref{f-is-g-2},
taking into account \eqref{A-zero},
we also have
\begin{equation}
\lim\sb{\lambda\to\pm i\infty}
\norm{(\f_1(x,\lambda,\omega),\f_2(x,\lambda,\omega))
-(\G_1(x,\lambda,\omega),\G_2(x,\lambda,\omega))B(\lambda,\omega)}
\sb{\C^4\times\C^4}=0,
\qquad
x\in\R.
\end{equation}
It follows that
$\lim\sb{\lambda\to\pm i\infty}
B(\lambda,\omega)^2=I_2$, hence
$\lim\sb{\lambda\to\pm i\infty}
\det B(\lambda,\omega)=\pm 1$.
The relation
$\lim\sb{\lambda\to\pm i\infty}\det B(\lambda,\omega)=\pm 1$
can be obtained by substituting the
``interaction term'' $\bW$
with $s\bW$, $s\in[0,1]$,
and using the continuity argument when changing
$s$ from $0$ to $1$.
\end{proof}

\begin{lemma}\label{lemma-large-lambda}
For each $\omega\in\varOmega$,
$
\lim\sb{\lambda\to \pm i\infty} \abs{E(\lambda,\omega)}=1.
$
\end{lemma}

\begin{proof}
Using \eqref{def-xi1-eta1} and \eqref{def-xi2-eta2},
we compute:
\begin{equation}\label{det-is-1}
\det[\Xi_1,\Xi_2,\Eta_1,\Eta_2]
=1+\mathcal{O}(\abs{\lambda}^{-1}),
\qquad
\lambda\to\pm i\infty.
\end{equation}
On the other hand, by \eqref{g-is-f},
\[
E(\lambda,\omega)
=
\det[\f_1,\f_2,\g_1,\g_2]
=
\det
[\f_1,\f_2,\sum\sb{j=1}\sp{2}\F_j B_{j1},
\sum\sb{j=1}\sp{2}\F_j B_{j2}]
=
\det B(\lambda,\omega)\,\det[\f_1,\f_2,\F_1,\F_2]
\]
\[
=
\det B(\lambda,\omega)
\lim_{x\to+\infty}\det[\f_1,\f_2,\F_1,\F_2]
=
\det B(\lambda,\omega)\,\det[\Xi_1,\Xi_2,\Eta_1,\Eta_2],
\]
where we used the asymptotics of $\f_j$, $\F_j$
from Proposition~\ref{prop-jost}.
Therefore,
by Lemma~\ref{lemma-fg-infty} and \eqref{det-is-1},
\[
\lim\sb{\lambda\to\pm i\infty}E(\lambda,\omega)
=
\lim\sb{\lambda\to\pm i\infty}
\det B(\lambda,\omega)
\lim\sb{\lambda\to\pm i\infty}
\det[\Xi_1,\Xi_2,\Eta_1,\Eta_2]=1.
\]
This finishes the proof.
\end{proof}

\begin{remark}
For $\omega\in\varOmega$,
$\lambda\in i\R$
with $\abs{\lambda}>1+\abs{\omega}$,
the Jost solutions
$\f_j,\,\F_j$, $j=1,\,2$
(and similarly $\g_j,\,\G_j$, $j=1,\,2$)
are linearly independent
(since so are the vectors
$\Xi_j$, $\Eta_j$, $j=1,\,2$
from \eqref{def-xi1-eta1},
\eqref{def-xi2-eta2});
hence there is a
``scattering matrix''
$S(\lambda,\omega)\in\C^{4\times 4}$
such that
\[
\big(
\g_1(x,\lambda,\omega),\g_2(x,\lambda,\omega),
\G_1(x,\lambda,\omega),\G_2(x,\lambda,\omega)
\big)
=
\big(
\f_1(x,\lambda,\omega),\f_2(x,\lambda,\omega),
\F_1(x,\lambda,\omega),\F_2(x,\lambda,\omega)
\big)
S(\lambda,\omega).
\]
Taking into account the relations
\eqref{f-is-g}
between
$\f_j$ and $\g_j$
and between $\F_j$ and $\G_j$,
we conclude that one also has
\[
(\f_1,\f_2,\F_1,\F_2)
=(\g_1,\g_2,\G_1,\G_2)S,
\]
hence $S^2=I$, $\det S=\pm 1$.
Taking into account that
$S\to\begin{bmatrix}0&I_2\\I_2&0\end{bmatrix}$
in the limit of zero interaction
(when $\bW(x,\omega)$
in \eqref{paw} is substituted by zero),
we conclude that $\det S=1$.
\end{remark}

\subsection{Explicit construction of the resolvent
of the linearization operator $\bJ\bL$}
\label{sect-resolvent}

In this section,
we will not restrict $\bJ\bL$ onto $\bX$
and give a general construction of the resolvent
in the case when $E(\lambda,\omega)\ne 0$.

\begin{remark}
Although
for applications to asymptotic stability
we will only need the resolvent
of $\bJ\bL(\omega)$
for $\lambda$ in the essential spectrum,
we will make our construction
for all $\lambda\in i\R$.
\end{remark}

\begin{definition}
For $1\leq p\leq \infty$ and $s\in\R$,
we will use the weighted $L^p$ spaces
with polynomial weights:
\[
\|f\|_{L^p_s}:=\|\langle\cdot\rangle^s f\|_{L^p}.
\]
For $f(x,t)$, we will denote
\[
\|f\|_{(L^p_s)_x}:=\|f(\cdot,t)\|_{L^p_s}
=\|\langle\cdot\rangle^s f(\cdot,t)\|_{L^p}.
\]
\end{definition}


\begin{proposition}
\label{prop9}
Fix $\omega\in\varOmega$.
Assume that $\lambda\in i\R$,
$\abs{\lambda}\ge 1-\abs{\omega}$,
is such that $E(\lambda,\omega)\ne 0$.

\begin{itemize}
\item
There are resolvents
$G\sp\pm(x,y,\lambda,\omega)$ of the operator
$\bJ\bL
=-\bmupalpha(\p\sb x-\mathcal{M}(x,\lambda,\omega))$
which satisfy
\begin{equation}\label{d-g-delta}
-\bmupalpha(\p\sb x-\mathcal{M}(x,\lambda\pm 0,\omega))
G\sp\pm(x,y,\lambda\pm 0,\omega)
=\delta(x-y)I_4,
\end{equation}
and for some $C(\lambda,\omega)<\infty$
(locally bounded in $\lambda$ and $\omega$)
one has
\begin{equation}\label{g-small}
\abs{
G\sp\pm(x,y,\lambda,\omega)}
\le C(\lambda\pm 0,\omega)
\min(\langle x\rangle,\langle y\rangle)\langle y\rangle,
\qquad
(x,\,y)\in\R^2.
\end{equation}
\item
For each $\omega\in\varOmega$,
there is $C(\omega)<\infty$ such that
\begin{equation}\label{G-bound-away}
\mathop{\lim\sup}_{\Lambda\to\pm\infty}
\norm{G\sp\pm(x,y,i\Lambda\pm 0,\omega)}\sb{\End(\C^N)}
\le
C(\omega),
\qquad
(x,\,y)\in\R^2.
\end{equation}
\item
For every $s>3$ and $K>0$,
there is a constant $C_{s,K,\omega}<\infty$
such that for all $\lambda\in i\R$ with $|\lambda|<K$ one has
\begin{equation}
\label{t:102-0}
\sup_{\lambda\in i\R\pm 0,\,|\lambda|<K} \|(\bJ\bL(\omega)-\lambda)^{-1}
\|_{L^2_s\to L^2_{-s}}
\leq C_{s,K,\omega}.
\end{equation}
\item
There is a constant $C_\omega<\infty$ such that
\begin{equation}
\label{a230-0}
\limsup_{
\lambda\in i\R\pm 0,\,|\lambda|\to \infty}
\|(\bJ\bL(\omega)-\lambda)^{-1}
\|_{L^2\sb{1}\to L^2\sb{-1}}
\leq C\sb\omega.
\end{equation}
\end{itemize}
\end{proposition}


\begin{proof}
We will only provide a construction
of $G\sp{-}$; see Remark~\ref{remark-g-pm} below.

Recall that
$\f_1$, $\f_2$ are Jost solutions decaying
(or oscillating)
for $x\to +\infty$,
while $\F_1$, $\F_2$ are the growing ones
(or oscillating ones).
$\f_1$, $\F_1$ have $\kappa_1$ as the rate
of decay and growth, respectively;
$\f_2$ and $\F_2$ have the rate $\kappa_2$,
with $\kappa_2>\kappa_1\ge 0$
(Cf. \eqref{def-kappa}).
Similarly with $\g_1$, $\g_2$, $\G_1$, $\G_2$, for $x\to-\infty$.

Recall that
if $\xi\sb j=0$, then
$\F\sb j=\f\sb j$,
hence the set $\{\f\sb 1,\,\f\sb 2,\,\F\sb 1,\,\F\sb 2\}$
is no longer linearly independent.
To overcome this issue,
let us modify $\F\sb j$.
For $\xi\sb j\ne 0$, denote
\begin{equation}\label{def-f-tilde}
\tilde\F\sb j(x,\lambda,\omega)
=
\F\sb j(x,\lambda,\omega)
+\frac{\f\sb j(x,\lambda,\omega)-\F\sb j(x,\lambda,\omega)}{2i\xi\sb j},
\qquad
j=1,\,2;
\end{equation}
\begin{equation}\label{def-g-tilde}
\tilde\G\sb j(x,\lambda,\omega)
=
\G\sb j(x,\lambda,\omega)
+\frac{\g\sb j(x,\lambda,\omega)-\G\sb j(x,\lambda,\omega)}{2i\xi\sb j}.
\qquad
j=1,\,2.
\end{equation}
Note that by \eqref{f-is-g}
one has
$\tilde G(x,\lambda,\omega)
=\bmupbeta\tilde F(-x,\lambda,\omega)$.

For $\lambda\in i\R$ such that
$\xi\sb j(\lambda,\omega)=0$,
we define $\tilde\F\sb j(x,\lambda,\omega)$
by the pointwise limit:
\[
\tilde\F\sb j(x,\lambda,\omega)
=
\F\sb j(x,\lambda,\omega)
+
\lim\sb{\lambda'\to\lambda;\,\xi_j(\lambda)>0}
\Big\{
\frac{\f\sb j(x,\lambda',\omega)
-\F\sb j(x,\lambda',\omega)}{2i\xi\sb j(\lambda')}
\Big\},
\]
and similarly for $\tilde\G\sb j$;
then
one has
$\tilde\F\sb j(x,\lambda,\omega)\sim \Xi\sb j\langle x\rangle$
for $x\gg 1$
and
$\tilde\G\sb j(x,\lambda,\omega)\sim \Eta\sb j\langle x\rangle$
for $x\ll -1$.

By Proposition~\ref{prop-jost},
we have the following asymptotics
for $\tilde\F_j$, $\tilde\G_j$:

\begin{lemma}\label{lemma-f-tilde}
For each
$\omega\in\varOmega$,
$\lambda\in i\R$,
one has:
\[
\abs{\tilde\F\sb j(x,\lambda,\omega)}
\le C(\omega)\langle x\rangle e^{\kappa\sb j x},
\qquad
x\ge 0,
\qquad
j=1,\,2,
\]
\[
\abs{\tilde\G\sb j(x,\lambda,\omega)}
\le C(\omega)\langle x\rangle e^{\kappa\sb j \abs{x}},
\qquad
x\le 0,
\qquad
j=1,\,2,
\]
where $C(\omega)$ is locally bounded
in $\omega$.
\end{lemma}

\begin{remark}
In Lemma~\ref{lemma-f-tilde},
the estimates remain true
when $\lambda$ is above the corresponding threshold,
so that $\xi\sb j>0$
while $\kappa\sb j=0$
(Cf. definition \eqref{def-kappa}).
\end{remark}

\begin{proof}
This follows from
Proposition~\ref{prop-jost}
and definitions \eqref{def-f-tilde}, \eqref{def-g-tilde}.
\end{proof}

Abusing the notations (Cf. \eqref{g-is-f}),
we assume that
$A(\lambda,\omega),\,B(\lambda,\omega)\in\C^{4\times 4}$
are such that
\[
\g\sb k(x,\lambda,\omega)
=
\sum\sb{j=1}\sp{2}
\f\sb j(x,\lambda,\omega) A\sb{jk}(\lambda,\omega)
+
\sum\sb{j=1}\sp{2}
\tilde\F\sb j(x,\lambda,\omega) B\sb{jk}(\lambda,\omega),
\qquad
k=1,\,2,
\]
which we write as
\begin{equation}\label{A-B-2}
(\g\sb 1,\g\sb 2)
=
(\f\sb 1,\f\sb 2)A
+(\tilde\F\sb 1,\tilde\F\sb 2)B.
\end{equation}
Multiplying \eqref{A-B-2} by $\bmupbeta$,
flipping the sign of $x$,
and using \eqref{f-is-g},
we arrive at
\begin{equation}\label{A-B}
(\f_1,\f_2)=(\g_1,\g_2)A+(\tilde\G_1,\tilde\G_2)B,
\end{equation}
with the same $A$, $B$ as in \eqref{A-B-2}.

\begin{lemma}\label{lemma-b-nondegenerate}
If $E(\lambda,\omega)\ne 0$,
then the matrix
$B(\lambda,\omega)$ is non-degenerate.
\end{lemma}

\begin{proof}
By  \eqref{A-B-2},
$
E(\lambda,\omega)
=\det[\f_1,\f_2,\g_1,\g_2]
=\det[\f_1,\f_2,(\f_1,\f_2)A+(\tilde\F_1,\tilde\F_2)B]
=\det[\f_1,\f_2,\tilde\F_1,\tilde\F_2]\det B.
$
\end{proof}

If $\lambda$ is neither an eigenvalue
nor a resonance,
so that the Jost solutions
\begin{equation}\label{ffgg}
\big\{\f_1(x,\lambda,\omega),\,\f_2(x,\lambda,\omega),
\,\g_1(x,\lambda,\omega),\,\g_2(x,\lambda,\omega)
\big\}
\end{equation}
are linearly independent,
we define:
\begin{equation}\label{green}
G(x,y,\lambda,\omega)
=
-\bmupalpha
\sum\sb{j,k=1}\sp{2}
\Big[
\Theta(x-y)
\f_j(x)\Gamma\sb{jk}(\lambda,\omega)\otimes \g\sb k\sp\ast(y)
+
\Theta(y-x)
\g_j(x)\Gamma\sb{jk}(\lambda,\omega)\otimes \f\sb k\sp\ast(y)
\Big]\Delta(y,\lambda,\omega)^{-1},
\end{equation}
where
$\Theta$ is the Heaviside step-function,
the Jost solutions also depend on $(\lambda,\omega)$
(this is not explicitly indicated),
the matrix
$\Gamma(\lambda,\omega)$
is defined by
\begin{equation}\label{def-gamma}
\Gamma(\lambda,\omega)=
\frac{1}{\sqrt{\abs{B_{21}}^2+\abs{B_{22}}^2}}
\begin{bmatrix}\abs{B_{22}}&\abs{B_{21}}e^{-is}
\\
\abs{B_{21}}e^{is}&-\abs{B_{22}}\end{bmatrix},
\end{equation}
so that $\det\Gamma=-1$;
here $s\in\R$ is chosen so that
\begin{equation}\label{gamma-good}
\sum\sb{j=1}\sp 2 B\sb{2j}\Gamma\sb{j1}
=B_{21}\Gamma\sb{11}+B_{22}\Gamma\sb{21}
=\frac{B_{21}\abs{B\sb{22}}+B_{22}\abs{B\sb{21}}e^{is}}
{\sqrt{\abs{B_{21}}^2+\abs{B_{22}}^2}}=0.
\end{equation}
(This choice of $\Gamma$
is justified later by the need to
have appropriate estimates on $G(x,y,\lambda,\omega)$.)
The matrix $\Delta(y,\lambda,\omega)$ in \eqref{green}
is defined by
\begin{equation}\label{def-delta}
\Delta(y,\lambda,\omega)=
\f_j(y,\lambda,\omega)\Gamma\sb{jk}(\lambda,\omega)
\otimes\g\sb k\sp\ast(y,\lambda,\omega)
-
\g_j(y,\lambda,\omega)
\Gamma\sb{jk}(\lambda,\omega)
\otimes\f\sb k\sp\ast(y,\lambda,\omega).
\end{equation}
Since $\det\Gamma\ne 0$,
the matrix \eqref{def-delta} is invertible
as long as
$\{\f\sb 1,\,\f\sb 2,\,\g\sb 1,\,\g\sb 2\}$
are linearly independent.
Moreover,
\begin{equation}\label{det-m}
\det \Delta(y,\lambda,\omega)
=\abs{E(\lambda,\omega)}^2.
\end{equation}

The relation \eqref{det-m} follows from the following
identity:

\begin{lemma}\label{lemma-inverse}
For any $u_j,\,v_j\in\C^N$, $1\le j\le N$,
$A\in\C^{N\times N}$,
one has
\begin{equation}\label{eq-inverse}
\det
\Big(
\sum\sb{j,\,k=1}\sp{N}u_j A_{jk}\otimes v_k\sp\ast\Big)
=\det A \det[u_1,\dots,u_N]
\,\overline{\det[v_1,\dots,v_N]}.
\end{equation}
\end{lemma}

\begin{proof}
If $v_j$ are linearly dependent,
the rank of the matrix
in the left-hand side is smaller than $N$,
and both sides in \eqref{eq-inverse} vanish.
Otherwise, the proof follows
from computing the determinants
of both sides of the identity
\[
\big(\sum\sb{j,\,k=1}\sp{N}u_j A_{jk}\otimes v_k\sp\ast\big)
\big([v_1,\dots,v_N]\sp\ast\big)^{-1}
=
\left[
\sum_{j=1}\sp{N}u_j A_{j1},\,\dots,\,\sum_{j=1}\sp{N}u_jA_{jN}
\right].
\qedhere
\]
\end{proof}

Applying Lemma~\ref{lemma-inverse}
to \eqref{def-delta},
thus setting
$[u_1,\dots,u_4]=[v_1,\dots,v_4]=[\f_1,\f_2,\g_1,\g_2]$
and $A=\begin{bmatrix}0&\Gamma\\-\Gamma&0\end{bmatrix}$,
one derives:
\[
\det \Delta(y,\lambda,\omega)
=
\big(\det\Gamma(\lambda,\omega)\big)^2
\big|
\det[\f_1(y,\lambda,\omega),\f_2(y,\lambda,\omega),
\g_1(y,\lambda,\omega),\g_2(y,\lambda,\omega)]
\big|^2,
\]
arriving at \eqref{det-m}.


As follows from the definition,
one has
\[
-\bmupalpha
(\p\sb x-\mathcal{M}(x,\lambda,\omega))G(x,y,\lambda,\omega)
=\delta(x-y)I_4.
\]

\begin{remark}\label{remark-g-pm}
At this point,
we need to recall that the Green function
is not uniquely defined at the essential spectrum.
Since the expression \eqref{green}
has the asymptotics $\sim e^{i\xi x}$, $\xi\approx -i\lambda$
for $\lambda\in i\R$, $\Im\lambda\gg 1$
(Cf. \eqref{xi12.def} and our convention
that $\xi_1,\,\xi_2$ are positive for $\lambda\in i\R$,
$\Im\lambda\gg 1$),
we conclude that \eqref{green}
will remain bounded for $\lambda$ near $i\R$
with $\Re\lambda<0$;
thus,
\eqref{green} corresponds to the limit
$G\sp{-}(x,y,\lambda,\omega):=G(x,y,\lambda-0,\omega)$
of the Green function
to the left of the upper branch of the essential spectrum
(this is consistent with \eqref{omega-lambda-positive}).
To define the limit on the right of the essential
spectrum, one would need to interchange
in the above considerations
$\f_j\sim e^{i\xi_j x}$ and $\F_j\sim e^{-i\xi_j x}$,
as well as $\g_j$ and $\G_j$
(this is assuming that
$\Im\lambda$ is large enough
so that $\xi_j>0$, hence $\f_j$, $\F_j$
with particular $j$
oscillate as $x\to+\infty$).
\end{remark}

\medskip

Let us now find the bounds on $G(x,y,\lambda,\omega)$.
Our goal is to show that
\eqref{green} does not grow exponentially
when $x$ and or $y$ go to infinity.
For example, when $y\to +\infty$, the fastest
growing term is $\tilde \F_2(y)$.
We need to show that
when \eqref{green} is written solely in terms
of $\f_j$, $\tilde\F_k$,
then in the combinations
$\f_j(x)\otimes\tilde F_k\sp\ast(y)$
one always has $x\ge y$,
and moreover the coefficient at the term
$\f_1(x)\otimes\tilde\F_2\sp\ast(y)$
vanishes
(this is the only problematic term,
when the decay of $\f_j(x)$
with $x\ge y$, $x\gg 1$, $y\gg 11$,
does not compensate
for the growth of $\tilde\F_k(y)$).
We claim that the choice of $\Gamma$ in \eqref{def-gamma}
specifically guarantees this.

For $x\ge y$, we only need
to consider the first term
from \eqref{green}:
\begin{equation}\label{fg}
\sum\sb{j,k}
\f_j(x)\Gamma\sb{jk}\otimes \g\sb k\sp\ast(y)
,
\qquad
x\ge y.
\end{equation}
It is enough to consider the following
two (intersecting) cases:
(1) $x\ge y$, $y\le 0$
and
(2) $x\ge y$, $x\ge 0$.
(In the intersection,
one has $x\ge 0$, $y\le 0$,
hence \eqref{fg} is uniformly bounded.)

Let us consider the case $x\ge y$, $x\ge 0$.
By \eqref{A-B-2},
the factor at $\f_1(x)$
in \eqref{fg}
is given by
\begin{equation}\label{f1-f2-small}
\sum\sb{k}
\Gamma\sb{1k}\g\sb k\sp\ast(y)
=
\sum\sb{j,k}
\Big(
\f\sb j(y)A\sb{jk}\bar\Gamma\sb{1k}
+\tilde\F\sb j(y)B\sb{jk}\bar\Gamma\sb{1k}
\Big)\sp\ast
=
\sum\sb{j,k}
\big(\f\sb j(y)A\sb{jk}\bar\Gamma\sb{1k}\big)\sp\ast
+
\sum\sb{k}
\big(\tilde\F_1(y)B\sb{1k}\bar\Gamma\sb{1k}\big)\sp\ast;
\end{equation}
in the last equality, we took into account \eqref{def-gamma}
and \eqref{gamma-good}:
\[
\sum\sb{k}
B\sb{2k}\bar\Gamma\sb{1k}
=
B\sb{21}\bar\Gamma\sb{11}+B\sb{22}\bar\Gamma\sb{12}
=
B\sb{21}\Gamma\sb{11}+B\sb{22}\Gamma\sb{21}
=0.
\]
It follows that
when we rewrite \eqref{fg}
in terms of $\f$, $\tilde \F$ only,
then the only term which can become exponentially large
for $x\ge y$, $x\ge 0$,
namely $\f_1(x)\otimes\tilde\F_2(y)\sp\ast$,
drops out!
Hence, \eqref{fg} is bounded by
$C(\lambda,\omega)\langle y\rangle$
for $x\ge 0$, $x\ge y$.
The linear growth in $y$ may come
from
$\f\sb j(x)\otimes \tilde\F_j(y)\sp\ast$
when $0\ll y\le x$,
whenever $\lambda\in i\R$ is near $i(1\pm\abs{\omega})$,
so that $\xi_j\approx 0$.

Let $x\ge y$, $y\le 0$.
By \eqref{A-B},
the factor at $\g_1\sp\ast(y,\lambda,\omega)$
in \eqref{fg}
is given by
\begin{equation}
\sum\sb{j}
\f\sb j(x)\Gamma\sb{j1}
=
\sum\sb{j,k}
\big(
\g\sb k(x)A\sb{kj}+\tilde\G\sb k(x)B\sb{kj}
\big)
\Gamma\sb{j1}
=
\sum\sb{j,k}
\g\sb k(x)A\sb{kj}\Gamma\sb{j1}
+
\sum\sb{j}\tilde\G_1(x)B\sb{1j}\Gamma\sb{j1};
\end{equation}
in the last equality, we took into account
that the coefficient at
$\tilde\G_2(x)\otimes\g_1\sp\ast(y)$
is given by
$B_{21}\Gamma\sb{11}+B_{22}\Gamma\sb{21}=0$,
by \eqref{gamma-good}.
Thus,
when we rewrite \eqref{fg} in terms of $\g$ and $\tilde\G$,
the coefficient at the term
$\tilde\G_2(x)\otimes\g_1\sp\ast(y)$,
the only one
out of $\tilde\G_j(x)\otimes\g_k\sp\ast(y)$
which can be exponentially large for $x\ge y$, $y\to-\infty$,
drops out.
It follows that \eqref{fg} is bounded by
$C(\omega)\langle x\rangle$
for $y\le 0$, $x\ge y$.
The linear growth in $x$ may come
from
$\tilde\G\sb j(x)\otimes \g\sb j(y)$
for $y\le x\ll 0$
(when writing \eqref{fg}
as a linear combination of
$\g_j\otimes \g_k\sp\ast$, $\tilde\G_j\otimes\g_k\sp\ast$,
via the substitution \eqref{A-B}),
whenever $\lambda$ is near $i(1\pm\abs{\omega})$
so that the corresponding $\xi_j$ is near zero.
By \eqref{f-f-c}, as $\abs{\lambda}\to\infty$,
$\norm{\f\sb j(\cdot,\lambda,\omega)}_{L^\infty}$
and $\norm{\g\sb j(\cdot,\lambda,\omega)}\sb{L^\infty}$
are bounded by
$c(\omega)<\infty$.

We summarize the cases
$x\ge y$, $y\le 0$ and $x\ge y$, $x\ge 0$:
Thus, for some $c(\lambda,\omega)<\infty$,
\begin{equation}\label{fg-2}
\Norm{
\sum\sb{j,k}
\Gamma\sb{jk}\f_j(x)\otimes \g\sb k\sp\ast(y)
}\sb{\End(\C^4)}
\le c(\lambda,\omega)
\min(\langle x\rangle,\langle y\rangle),
\qquad
x\ge y.
\end{equation}

The case $x\le y$ follows from the above once
we notice that
$
\Delta(-y,\lambda,\omega)
=\bmupbeta\Delta(y,\lambda,\omega)\bmupbeta
$
and then
\[
G(-x,-y,\lambda,\omega)
=
-\bmupbeta G(x,y,\lambda,\omega)\bmupbeta;
\]
we arrive at the same bound
but now for $x\le y$:
\begin{equation}\label{fg-3}
\Norm{
\sum\sb{j,k}
\Gamma\sb{jk}\g_j(x)\otimes \f\sb k\sp\ast(y)
}\sb{\End(\C^4)}
\le c(\lambda,\omega)\min(\langle x\rangle,\langle y\rangle),
\qquad
x\le y.
\end{equation}

Let us study the contribution
of the matrix $\Delta(y,\lambda,\omega)$ defined
in \eqref{def-delta}.
By \eqref{fg-2} and \eqref{fg-3},
$\Delta(y,\lambda,\omega)$ satisfies
\begin{equation}\label{m-small}
\norm{\Delta(y,\lambda,\omega)}\sb{\End(\C^4)}
\le c(\lambda,\omega)\langle y\rangle,
\end{equation}
with the linear growth only for
$x\approx \pm i(1\pm\omega)$.

By \eqref{det-m} and \eqref{m-small},
there is $C(\lambda,\omega)<\infty$
such that
\begin{equation}\label{m-small-inv}
\norm{\Delta(y,\lambda,\omega)^{-1}}\sb{\End(\C^4)}
\le C(\lambda,\omega)\langle y\rangle.
\end{equation}
(Here, we need to argue that
the minors of $\Delta$ can not grow
faster than $\langle y\rangle$;
at most one of $\tilde\G_j(y)\otimes \g_j(y)\sp\ast$,
$j=1,\,2$ can grow linearly at a given value of $\lambda$,
hence, in the appropriate basis,
only one element of $\Delta$ grows linearly
while others are bounded uniformly in $y\in\R$.)
Combining \eqref{fg-2} and \eqref{fg-3}
with \eqref{m-small-inv},
we arrive at the bound \eqref{g-small}.

\medskip

Let us now study the behaviour of $G(x,y,\lambda,\omega)$
for $\lambda\in i\R$, $\abs{\lambda}\to\infty$.
By Proposition~\ref{prop-jost},
the Jost solutions $\f_j$, $\tilde\F_j$,
$\g_j$, $\tilde\G_j$
are bounded uniformly in $x$
as long as $\abs{\lambda}$ is sufficiently large.
By Lemma~\ref{lemma-large-lambda}
and
\eqref{det-m},
for $\lambda\in i\R$, $\abs{\lambda}\to\infty$,
one has
$\abs{\det \Delta(y,\lambda,\omega)}\to 1$,
while the components of $\Delta(y,\lambda,\omega)$
are uniformly bounded for $\lambda\to \pm i\infty$.
It follows that the components of the matrix
$G(x,y,\lambda,\omega)$
defined in \eqref{green}
are bounded uniformly in
$x$ and $y$ as long as $\abs{\lambda}$ is sufficiently large.

\medskip

Finally, the bounds
\eqref{t:102-0}
and \eqref{a230-0}
follow from the pointwise estimates
\eqref{g-small} and \eqref{G-bound-away}
for Green's function.
This concludes the proof of
Proposition~\ref{prop9}.
\end{proof}

\section{Dispersive estimates for the semigroup}\label{disp-section}
In this section, we develop set of dispersive estimates, which will be useful in the sequel for controlling the radiation portion of the perturbation.
\subsection{Weighted decay estimates}

\begin{proposition}
\label{prop12}
Let $\omega \in \varOmega$. Then there exists $C<\infty$
such that for all $t>0$, the following estimates hold:
\begin{eqnarray*}
&&\sup_x \langle x\rangle^{-3}
\norm{[e^{t\bJ\bL(\omega)}P_c(\omega) f](x)}_{L^2_t}
\leq C \norm{f}_{L^2_x},\\
&&\|\int_{-\infty}^\infty e^{t\bJ\bL(\omega)} P_c(\omega)F(t, \cdot)\,dt\|_{L^2_x}\leq C
\|F\|_{(L^1_3)_x  L^2_t}.
\end{eqnarray*}
\end{proposition}


\begin{remark}
The estimates in Proposition~\ref{prop12} can be upgraded to include derivatives.
For example,
$$
\sup_x \langle x\rangle^{-3} \norm{\p_x [e^{t\bJ\bL}P_c(\omega) f](x)}_{L^2_t}
\leq C \norm{f}_{H^1_x}.
$$
Note that the last estimate presents a challenge, since $\p_x e^{t\bJ\bL} \neq
e^{t\bJ\bL} \p_x $. Nevertheless,
since
\[
\bL(\omega)=\bD_m-\omega I_4+\bW(x,\omega),
\]
with $\bD_m$ from \eqref{def-DD},
we may essentially commute the derivative with $e^{t\bJ\bL}$ modulo low order error terms, whence the result generalizes to include derivatives.
\end{remark}

\begin{proof}[Proof of Proposition~\ref{prop12}]
Clearly, the two estimates in the claim of Proposition~\ref{prop12} are dual to each other, so it suffices to establish the first one.

Pick an even function
$\chi\in C^\infty\sb{\mathrm{comp}}(\R)$ such that
\begin{equation}\label{def-chi}
\supp\chi\subset [-4,4],
\qquad
\chi(\Lambda)=1
\quad\mbox{for $|\Lambda|\le 3$}.
\end{equation}
Decompose the evolution into two pieces:
$$
e^{t\bJ\bL}P_c(\omega) f=\chi(i\bJ\bL)
 e^{t\bJ\bL}P_c(\omega) f+(1-\chi(i\bJ\bL)) e^{t\bJ\bL}P_c(\omega) f.
$$
The required estimate will follow from
\begin{eqnarray}
\label{a50}
 \sup_x \|(1-\chi(i\bJ\bL))e^{t\bJ\bL}P_c(\omega) f\|_{L^2_t} &\leq&C \|f\|_{L^2_x}, \\
\label{a55}
 \sup_x \langle x\rangle^{-3} \|\chi(i\bJ\bL) e^{t\bJ\bL}P_c(\omega) f\|_{L^2_t} &\leq&C\|f\|_{L^2_x}.
\end{eqnarray}
By \eqref{e-p-f},
for a fixed value of $x$, the Fourier transform in $t$ of the function
$$
g_{x,t}=(1-\chi(i\bJ\bL))e^{t\bJ\bL}P_c(\omega) f
$$
is exactly
$$
g_x(\Lambda)
=-(1-\chi(\Lambda))
\big(
[R^+_{\bJ\bL}(i\Lambda)-R^-_{\bJ\bL}(i\Lambda)]f
\big)(x).
$$
Thus, \eqref{a50} will follow from
\begin{equation}
\label{a60}
\sup_x \|(1-\chi(\Lambda)) R^\pm_{\bJ\bL}(i\Lambda)f(x)\|_{L^2_\Lambda}\leq C \|f\|_{L^2_x}.
\end{equation}
Similarly, \eqref{a55} will follow from
\begin{equation}
\label{a65}
\sup_x \langle x\rangle^{-3}
\| \chi(\Lambda) R^\pm_{\bJ\bL}(i\Lambda)f(x)\|_{L^2_\Lambda}\leq C \|f\|_{L^2_x}.
\end{equation}

We now prove \eqref{a60} and \eqref{a65}.

\bigskip

\noindent
{\bf Proof of \eqref{a60}.}
For brevity, we denote
 $$
 R_\bW(\Lambda):= (\bD_m-\omega I_4-\Lambda\bJ^{-1}+\bW)^{-1}.
 $$
 From the resolvent identity, we have
 $R_\bW=R_0-R_\bW \bW R_0=R_0-R_0 \bW R_\bW$,
 whence the following Born expansion holds:
\begin{equation}
\label{a600}
 R_\bW=R_0-R_0 \bW R_0 +R_0 \bW R_\bW \bW R_0.
\end{equation}
Observe that
$
R_0=\begin{bmatrix}
(D_m-(\omega+\Lambda) I_2)^{-1}&0 \\
0&(D_m-(\omega-\Lambda) I_2)^{-1}
\end{bmatrix}.
$
The restrictions imposed by the cut-off $(1-\chi)$
(Cf. \eqref{def-chi})
implies that $|\om\pm \Lambda|>3$.
It follows that it is enough to show that
\begin{eqnarray}
\label{a70}
&&\sup_x \int_3^\infty |(D_m-\mu I_2)^{-1} f(x)|^2 d\mu\leq
C \|f\|^2_{L^2_x}; \\
\label{a75}
&&\sup_x \int_3^\infty |(D_m-\mu I_2)^{-1} W_\nu (D_m-\mu I_2)^{-1} f(x)|^2 d\mu\leq C \|\bW\|_{L^1_x}^2 \|f\|^2_{L^2_x}; \\
\label{a80}
&&\sup_x \int_3^\infty |(D_m-\mu I_2)^{-1} W_\nu R_\bW W_\nu (D_m-\mu I_2)^{-1} f(x)|^2 d\mu\leq
C \|\langle x\rangle^{\alpha} \bW\|_{L^2_x}^2 \|f\|^2_{L^2_x}.
\end{eqnarray}
Above,
$\alpha>3/2$ and $W_\nu$ is either of the potentials $W_1, W_0$.
Similar estimates were shown in \cite[Section VIII]{MR2985264}, but we provide the details here for completeness. Note that
$$
(D_m-\mu I_2)^{-1}=(1-\p_x^2-\mu^2)^{-1} \left(\begin{array}{cc}
1+\mu&\p_x \\
-\p_x&\mu-1
\end{array}
\right).
$$
Thus, setting $\mu=\sqrt{k^2+1}$, the operator $(D_m-\mu I_2)^{-1}$ is represented as a linear combination of operators with the following kernels:
$$
e^{\pm i k |x|} \sgn(x),
\qquad
\frac{e^{\pm i k |x|}}{k},
\qquad
\frac{e^{\pm i k |x|}\sqrt{k^2+1}}{k}.
$$
Clearly, for the purposes of showing \eqref{a70}, \eqref{a75},
\eqref{a80}, it is enough to consider the operator with kernel $e^{\pm i k |x|} \sgn(x)$.

For the proof of \eqref{a70}, we have by Plancherel's
\begin{eqnarray*}
&&\sup_{x\in\R} \int_{\sqrt{8}}^\infty \left|\int e^{\pm i k |x-y|} \sgn(x-y)
f(y) \,dy
\right|^2 \frac{k\,dk}{\sqrt{k^2+1}}
\\
&& \leq2
\sup\sb{x\in\R}
\int_{\sqrt{8}}^\infty
\left\{
\Big|\int_{-\infty}^x e^{\mp i k y} f(y) \,dy\Big|^2
+\Big|\int_{x}^\infty e^{\pm i k y} f(y) \,dy\Big|^2
\right\}
\,dk
\leq 4 \|f\|_{L^2}^2 .
\end{eqnarray*}
Similarly, for \eqref{a75} we have (by Minkowski's)
\begin{eqnarray*}
&&\sup_x \int_{\sqrt{8}}^\infty \left|\int e^{\pm i k |x-y|} \sgn(x-y)
\bW(y) [R_0(\sqrt{1+k^2}) f] (y) \,dy
\right|^2 \frac{k\,dk}{\sqrt{k^2+1}}\\
&&\leq \int_{\sqrt{8}}^\infty \left|
\int |\bW(y)|
\Big|[R_0(\sqrt{1+k^2}) f] (y)\Big|\,dy
\right|^2 \,dk \\
&&\leq \left(\int |\bW(y)|
\left( \int_{\sqrt{8}}^\infty |[R_0(\sqrt{1+k^2}) f] (y)|^2 \,dk\right)^{1/2} \,dy\right)^2\\
&&\leq \|\bW\|_{L^1}^2 \sup_y \int_{\sqrt{8}}^\infty
\Big|[R_0(\sqrt{1+k^2}) f] (y)\Big|^2 \,dk
\leq C \|\bW\|_{L^1}^2 \|f\|_{L^2}^2.
\end{eqnarray*}
This shows \eqref{a75}. Finally, for \eqref{a80}, we estimate
\begin{eqnarray*}
&&\sup_x \int_{\sqrt{8}}^\infty \left|\int e^{\pm i k |x-y|} \sgn(x-y)
\bW(y) R_\bW [\bW R_0 f](y) \,dy\right|^2 \,dk
\\
&&\leq 
\|\langle y\rangle^{3} \bW(y)\|_{L^2}^2
\int_{\sqrt{8}}^\infty
\|\langle y\rangle^{-3} R_\bW \langle y\rangle^{-3} [ \langle y\rangle^{3} \bW(y) [R_0(\sqrt{1+k^2}) f] (y)\|_{L^2_y}^2 \,dk \\
&&
\leq\|\langle y\rangle^{3} \bW(y)\|_{L^2}^4 \|R_\bW\|_{(L^2_3)_x\to(L^2_{-3})_x}^2
\sup_y \int_{\sqrt{8}}^\infty
\left|R_0(\sqrt{1+k^2}) f] (y)\right|^2 \,dk \\
&&
\leq C \|\langle y\rangle^{3} \bW(y)\|_{L^2}^4
\|f\|_{L^2_x}^2.
\end{eqnarray*}
In the last estimate, we have used the estimates
from Proposition~\ref{prop9}
which are uniform for large $\Lambda$
(for large values of the spectral parameter $\Lambda>\sqrt{8}$),
$R_\bW:L^2_3(\R)\to L^2_{-3}(\R)$.

\bigskip

\noindent
{\bf Proof of \eqref{a65}.} \\
The statement for low frequencies follows from the following result:

\begin{lemma}\label{lemmaa-y10}
Define
$A:\,C\sb 0(\R)\to C(\R\times\R)$ by
\begin{equation}\label{def-operator-a}
u\mapsto
A u(x,\Lambda)=\chi(\Lambda)\int\sb\R G(x,y,i\Lambda,\omega)u(y)\,dy.
\end{equation}
Then $A$ extends to a continuous operator
$L^2(\R)\to L^\infty\sb{\mathrm{loc}}(\R,L^2_\Lambda(\R))$,
and moreover there is $C<\infty$ such that
\begin{equation}\label{cube}
\sup_x
\langle x\rangle^{-3}
\norm{A u(x,\cdot)}\sb{L^2_\Lambda}
\le C\norm{u}\sb{L^2}.
\end{equation}
\end{lemma}

\begin{proof}
Let $u\in L^2(\R,\C^4)$.
Without loss of generality, we assume that
$\supp u\subset\R\sb{+}$,
so that in \eqref{green}
we have $y\ge 0$.

\medskip

\noindent
{\bf The case $x\ge 0$.}
We use the expression \eqref{green} for $G(x,y,i\Lambda,\omega)$;
expressing
in \eqref{green} the Jost solutions
$\g\sb j$ in terms of $\f\sb j$ and $\tilde\F\sb j$,
we see that it suffices to check that
the expressions
\begin{equation}\label{three}
\int_0^\infty
\Theta(\pm(x-y))\f\sb j(x)\f\sb k\sp\ast(y)u(y)\,dy,
\qquad
\int_0^\infty
\Theta(y-x)\tilde\F\sb j(x)\f\sb k\sp\ast(y)u(y)\,dy,
\qquad
\int_0^\infty
\Theta(x-y)\f\sb j(x)\tilde\F\sb k\sp\ast(y)u(y)\,dy,
\end{equation}
with $j,\,k=1,\,2$,
are bounded in $L^2$ as functions of $\Lambda$,
with an appropriate bound on the growth with $x$.
Above, we omitted the weight $\chi(\Lambda)$
present in \eqref{def-operator-a};
this weight will become important when we will integrate by parts.

In \eqref{three}
and in the rest of the proof,
the Jost solutions
are evaluated at $\lambda=i\Lambda$
and $\omega$,
which we usually do not indicate explicitly
to shorten the notations.
The first two terms in \eqref{three}
are analyzed similarly;
the more difficult being the second one,
so we focus on it.

\noindent
$\bullet$
Assume that $\f\sb k(y,i\Lambda,\omega)$ is exponentially decaying,
so that
\[
\f\sb k(y,i\Lambda,\omega)\sim e^{-\kappa\sb k y},
\qquad
y\gg 1,
\]
with $\kappa\sb k>0$
(Cf. \eqref{def-kappa}).

When
$\tilde\F\sb j(x,i\Lambda,\omega)$ remains bounded
or grows linearly in $x$
for $x\gg 1$,
\[
\Abs{
\int_0^\infty
\tilde\F\sb j(x)\f\sb k\sp\ast(y) u(y)\,dy
}
\le
C
\langle x\rangle
\int_0^\infty
\abs{\f\sb k(y)}
\abs{u(y)}\,dy
\le
C
\langle x\rangle
\norm{\Theta(\cdot)\f\sb k}\norm{u}
\le
\frac{C\langle x\rangle}{\sqrt{\kappa\sb k}}\norm{u}.
\]
Note that $\kappa\sb k^{-1/2}$ is $L^2$ in $\Lambda$
near the thresholds $\Lambda=\pm(1\pm\omega)$.

When $\tilde\F\sb j(x,i\Lambda,\omega)$
is exponentially growing,
by Lemma~\ref{lemma-f-tilde},
we have
$\abs{\tilde\F\sb j(x)}
\le C(\Lambda,\omega)\langle x\rangle e^{\kappa\sb j x}$
for $x\ge 0$,
and moreover we only need to consider
terms with
$\kappa_j\le \kappa\sb k$
due to our construction of $G$ in
Proposition~\ref{prop9}
(the term
$\tilde\F_2(x)\f_1\sp\ast(y)$ is absent
in the expansion of $G(x,y)$
over $\f_j(x)\f_k\sp\ast(y)$, $\tilde\F_j(x)\f_k\sp\ast(y)$,
and $\f_j(x)\tilde\F_k\sp\ast(y)$),
and with $C(\Lambda,\omega)$
locally bounded in $\Lambda$ and $\omega$,
with
$
\mathop{\lim\sup}_{\Lambda\to\pm\infty}
C(\Lambda,\omega)
\le C(\omega)<\infty$.
Then, again,
\[
\Abs{
\int_0^\infty \Theta(y-x)
\tilde\F\sb j(x)\f\sb k\sp\ast(y) u(y)\,dy
}
\le
C\langle x\rangle
\int_x^\infty
e^{\kappa_j x}e^{-\kappa_k y}
\abs{u(y)}\,dy
\le
\frac{C\langle x\rangle}{\sqrt{\kappa\sb k}}\norm{u}.
\]

\noindent
$\bullet$
Assume that $\f\sb k(y,i\Lambda,\omega)\sim e^{i\xi\sb k y}$
is oscillating:
\begin{equation}\label{is-oscillating}
\abs{\Lambda\pm\omega}>1,
\qquad
\xi\sb k(i\Lambda,\omega)=\sqrt{(\Lambda\pm\omega)^2-1}>0.
\end{equation}
(According to the construction of the Green function,
since $\f_k$ is oscillating,
we only need to consider the terms in \eqref{three}
with $\tilde\F\sb j(x)$ also oscillating:
$\xi_j>0$.)
In this case,
the integration in spatial variables
becomes possible after
integrating by parts with the aid of the operator
$L_\Lambda=\frac{1}{i(y-z)}\p\sb\Lambda$;
we only give a sketch,
substituting the Jost solutions
by their asymptotic behaviour
$\f_k(x)\sim e^{i\xi_k x}$
and
$\tilde\F_j(x)\sim e^{-i\xi_j x}
+\frac{e^{i\xi_j x}-e^{-i\xi_j x}}{2i\xi_j}$
(Cf. \eqref{def-f-tilde}).
Then the integration by parts yields
\begin{eqnarray}\label{square}
&&
\Abs{
\int\sb{\R}
\chi(\Lambda)\,d\Lambda
\int\sb{\R\times\R}
\abs{\tilde\F\sb j(x)}^2
\overline{\f\sb k\sp\ast(z)u(z)}
\f\sb k\sp\ast(y)u(y)
\,dy\,dz
}
\nonumber
\\
&&
=
\Abs{
\int\sb{\R}
\chi(\Lambda)\,d\Lambda
\int\sb{\R\times\R}
\abs{\tilde\F\sb j(x)}^2
L_\Lambda^2
\Big(
\overline{\f\sb k\sp\ast(z)u(z)}
\f\sb k\sp\ast(y)u(y)
\Big)
\,dy\,dz
}
\nonumber
\\
&&
\le
\langle x\rangle^2
\int\limits\sb{\R} \chi(\Lambda)\,d\Lambda
\int\limits\sb{\R\times\R}
\frac{C\abs{u(y)}\abs{u(z)}\,dy\,dz}
{1+\abs{\mu_j(x,\Lambda,\omega)^{-1}(y-z)\p\sb\Lambda\xi\sb k}^2}
\le
C\langle x\rangle^2
\int\limits\sb{\R} \chi(\Lambda)\,d\Lambda
\frac{\mu_j(x,\Lambda,\omega)\norm{u}^2}{\abs{\p\sb\Lambda\xi\sb k}}.
\end{eqnarray}
Above,
\begin{equation}\label{at-most}
\mu_j(x,\Lambda,\omega)
:=
C\max\Big(
1,
\ \abs{x}\abs{\p\sb\Lambda\xi_j},
\ \frac{\abs{\p\sb\Lambda^2\xi_j}}{\abs{\p\sb\Lambda\xi_j}}
\Big)
\end{equation}
is the bound on the contribution of $\p\sb\Lambda$
during the integration by parts
(the last term in \eqref{at-most}
is the contribution from the derivative
falling onto $\p\sb\Lambda\xi_j$ during the second
integration by parts).
In the last inequality in \eqref{square},
we used the Schur test.
Due to \eqref{is-oscillating},
one has
\[
\p\sb\Lambda\xi\sb j
=\frac{\Lambda\pm\omega}{\xi\sb j},
\qquad
\abs{\p\sb\Lambda^2\xi\sb j}
\le\frac{C\langle\Lambda\rangle^2}{\xi\sb j^3};
\]
hence, \eqref{at-most} can be continued as follows:
\[
\mu_j(x,\Lambda,\omega)
=C\max
\Big(
1,
\ \abs{x}\abs{\p\sb\Lambda\xi\sb j},
\ \frac{\abs{\p\sb\Lambda^2\xi\sb j}}{\abs{\p\sb\Lambda\xi\sb j}^2}
\Big)
\le
C\max\Big(1,
\ \frac{\langle x\rangle}{\xi_j}
\Big).
\]
It follows that
\[
\frac{\mu_j(x,\Lambda,\omega)}{\abs{\p\sb\Lambda\xi_k}}
\le
\frac{C\langle x\rangle}{\xi_k(i\Lambda,\omega)}
\]
is locally integrable in $\Lambda\in\supp\chi$
(and such that $\abs{\Lambda\pm\omega}>1$),
and moreover
$\langle x\rangle^{-3}
\int
\langle x\rangle^2
\frac{\mu_j(x,\Lambda,\omega)}{\abs{\p\sb\Lambda\xi\sb k}}
\chi(\Lambda)\,d\Lambda$
is bounded uniformly in $x$.
The factor $\langle x\rangle^2$
under the integral comes from the bound
$\abs{\tilde\F_j(x,\lambda,\omega)}\le C\langle x\rangle$
which remains valid uniformly in $\xi_j>0$
when $\xi_j\to 0+$ (Cf. Lemma~\ref{lemma-f-tilde}).
This leads to \eqref{cube}.

\medskip

Let us analyze the last term in \eqref{three}.
When $\tilde\F\sb k(y)$ is oscillating,
we use the same consideration as above,
in the case when $\f_k(y)$ was oscillating.
Let us consider the situation
when $\tilde\F\sb k(y)$ is exponentially growing
as $y\to+\infty$.
Since this growth is compensated by the decay of
$\Theta(x-y)\f\sb j(x)$
due to the choice of $B_{jk}(\lambda,\omega)$ in \eqref{def-gamma}
(as we mentioned above,
the construction of $G$ is such that
we only need to treat terms with
$\kappa\sb k\le\kappa\sb j$),
it suffices to consider the terms
$\Theta(x-y)\f\sb j(x)\tilde\F\sb k\sp\ast(y)$
which are bounded by
$
\Theta(x-y)\langle x\rangle e^{-\kappa_j\abs{x}}e^{\kappa_k\abs{y}},
$
with $\kappa\sb j\ge\kappa\sb k$.
We have:
\[
\langle x\rangle
\int
\Theta(x-y)e^{-\kappa_j\abs{x}}e^{\kappa_k\abs{y}}
\abs{u(y)}\,dy
\le
C\langle x\rangle
\int_0^x\abs{u(y)}\,dy
\le
C
\langle x\rangle^{3/2}
\norm{u},
\]
which immediately leads to \eqref{cube}.

\medskip

\noindent
{\bf The case $x\le 0$.}
This case is in fact much simpler.
In this case,
from \eqref{green},
we only need to consider the contribution from
$
\sum\sb{j,k=1}\sp{2}
\g_j(x)\Gamma\sb{jk}\f\sb k\sp\ast(y);
$
we need to prove that the expressions
\[
\int\sb{\R\sb{+}}
\g_j(x)\Gamma\sb{jk}\f\sb k\sp\ast(y)u(y)\,dy,
\]
with $j,\,k=1,\,2$,
are $L^2$-bounded
in $\Lambda$, for $\Lambda\in\supp\chi$.
Since $\g_j(x)$ are bounded for $x\le 0$,
the proof follows the lines of our
argument for the case $x\ge 0$,
except that we do not need to worry
whether the decay of $\f\sb k(x)$
compensates the growth $\g\sb j(x)$
since the latter terms are bounded for $x\le 0$.
This finishes the proof.
\end{proof}

This completes the proof of Proposition~\ref{prop12}.
\end{proof}

Next, we state and prove the estimate for the ``free'' Dirac operator, which is reminiscent of Proposition~\ref{prop12}.
{\it Surprisingly, however, Proposition~\ref{prop12} does not hold for $\bD_m$,
unless one adds a derivative correction
that takes care of the low frequency component of $f$.
Note also that there is no need of the exponential weight either,
but recall that this was added for the perturbed operator to counter exactly the same effect:
a somewhat pathological behavior of the low frequency component of the solution. }
\begin{lemma}
\label{lemma133}
We have the following estimate
for the evolution of the `` free'' Dirac operator:
\begin{eqnarray}
\label{a800}
&&\sup_x
\Big\|
e^{t\bJ\bL_0} f
\Big\|_{L^2_t}\leq \| M f\|_{L^2_x},
\\
\label{a801}
&&\|\int e^{t\bJ\bL_0} F(t, \cdot)\,dt\|\leq C \|M F\|_{L^1_x L^2_t},
\end{eqnarray}
where $M=\sqrt{\langle\nabla\rangle/|\nabla|}$ or more precisely $\widehat{M g}(\xi)= \frac{(1+\xi^2)^{1/4}}{|\xi|^{1/2}} \hat{g}(\xi)$.
In addition, by a simple duality argument, there is also
\begin{equation}
\label{a805}
\|\int e^{t\bJ\bL_0} F(t, \cdot)\,dt\|_{L^2_x}\leq C \| M F\|_{L^1_x L^2_t}.
\end{equation}
\end{lemma}
\begin{proof}
Clearly, \eqref{a801} is just a dual to \eqref{a800}, so we concentrate on \eqref{a800}.
Due to the block-diagonal structure of $\bD_m$,
the problem $i u_t = \bD_m u$ reduces to the following linear system:

%
%
%

\[
i\p\sb t
h
=D_m h,
\qquad
h\at{t=0}=h^0,
\]
which in the components of $h\in\C^2$ takes the following form:
$$
\left\{
\begin{array}{l}
i \p_t h_1=h_1+\p_x h_2, \\
i \p_t h_2=-\p_x h_1-h_2,\\
h_1(0)= h_1^0,\qquad h_2(0)= h_2^0.
\end{array}
\right.
$$

It follows that $h_1, h_2$ both satisfy the Klein--Gordon equation
$\p_{tt} h_{1,2}-\p_{xx} h_{1,2} +h_{1,2}=0$ with the corresponding initial data. Thus, \eqref{a800} reduces to
$$
\sup_x \|e^{i t \langle\nabla\rangle} f\|_{L^2_t}\leq C \|M f\|_{L^2},
$$
where $\widehat{\langle\nabla\rangle g}(\xi)=\sqrt{1+\xi^2} \hat{g}(\xi)$.
Changing the variables $\kappa=\sgn(\xi) \sqrt{1+\xi^2}$
and using Plancherel's theorem, we have:
\begin{eqnarray*}
\|e^{i t \langle\nabla\rangle} f\|_{L^2_t}^2 &=&
\int\Big|
\int e^{i t \sqrt{1+\xi^2}} \hat{f}(\xi) e^{i \xi x} d\xi
\Big|^2 dt \\
&=& \int
\Big|
\int_{|\kappa|>1} e^{i t \kappa} \hat{f}(\sqrt{\kappa^2-1})
e^{i x \sqrt{\kappa^2-1} } \frac{\kappa\,d\kappa}{\sqrt{\kappa^2-1}}
\Big|^2 dt \\
&=& \int_{|\kappa|>1} \frac{|\hat{f}(\sqrt{\kappa^2-1})|^2 \kappa^2\,d\kappa}{\kappa^2-1}  =
\int \frac{|\hat{f}(\xi)|^2 \sqrt{1+\xi^2} }{|\xi|} d\xi=\|M f\|_{L^2}^2.
\qedhere
\end{eqnarray*}
\end{proof}

Next, we present an estimate for the retarded term in the Duhamel representation,
in the spirit of Proposition~\ref{prop12}.
\begin{lemma}
\label{lemma20}
Let $\omega\in \varOmega$. There exists $C<\infty$ so that
\begin{equation}
\label{a510}
\sup_x \langle x\rangle^{-3}
\| \int_0^t e^{-(t-\tau)\bJ\bL} P_c(\omega) F(\tau, \cdot)\,d\tau\|_{L^2_t}
\leq C \| F\|_{(L^1_3)_x  L^2_t}.
\end{equation}
\end{lemma}
\begin{proof}
It is well-known that these type of estimates are essentially dual estimates to the one presented in Proposition~\ref{prop12}. In fact, recall that from Proposition~\ref{prop12},
$$
\|\int_0^\infty e^{\tau\bJ\bL} P_c(\omega) F(\tau, \cdot)\,d\tau\|_{L^2_x}
\leq C \|F\|_{(L^1_3)_x  L^2_t}.
$$
Thus, if one deals with the related quantity
$\int_0^\infty e^{-(t-\tau)\bJ\bL} P_c(\omega) F(\tau, \cdot)\,d\tau$,
we have, by virtue of Proposition~\ref{prop12} and its dual estimate,
\begin{eqnarray*}
\|\langle x\rangle^{-3}
\int_0^\infty e^{-(t-\tau)\bJ\bL} P_c(\omega) F(\tau, \cdot)\,d\tau\|_{L^\infty_x L^2_t}
&=& \|\langle x\rangle^{-3}
e^{-t \bJ\bL}\int_0^\infty e^{\tau\bJ\bL}P_c(\omega) F(\tau, \cdot)\,d\tau\|_{L^\infty_x L^2_t}
\\
&
\leq&
C \|\int_0^\infty e^{\tau\bJ\bL}P_c(\omega) F(\tau, \cdot)\,d\tau\|_{L^2_x}
\leq C \|F\|_{(L^1_3)_x L^2_t}.
\end{eqnarray*}
However, as one observes quickly, we have to deal with $\int_0^t$
in the retarded term in the Duhamel representation,
instead of $\int_0^\infty$ in our previous consideration.
This is a non-trivial issue, which has been resolved in the literature, see \cite[Lemma 11]{MR2511047}
and \cite[Lemma 2]{MR2985264}.
 In short, these results allows one to write for $F(t,x)=g_1(t) g_2(x)$,
\begin{eqnarray*}
U(t,\cdot) &=& 2\int_0^t e^{(t-\tau)\bJ\bL} P_c(\omega) F(\tau, \cdot)\,d\tau
+\left( \int_{-\infty}^0-\int_0^\infty \right)
e^{(t-\tau)\bJ\bL} P_c(\omega) F(\tau, \cdot)\,d\tau, \\
 U(t,x) &=& \frac{i}{\sqrt{2\pi}}
\int_{-\infty}^\infty e^{-i t \Lambda}\check{g_1}(\Lambda) \left(\left[
R^+_{\bJ\bL}(i\Lambda) + R^-_{\bJ\bL}(i\Lambda) \right] g_2\right)(x)\,d\Lambda.
\end{eqnarray*}
Since we have already shown the estimates for the term $\int_0^\infty \ldots$ (and the estimates for $\int_{-\infty}^0 \ldots$ are similar), it remains to show the appropriate estimates for $U$. By the Plancherel theorem in the $t$-variable,
\begin{eqnarray*}
&&
\|U(t,\cdot)\|_{(L^\infty_{-3})_x L^2_t}=
\Big\|
\langle x\rangle^{-3} \|\check{g_1}(\Lambda)[
R^+_{\bJ\bL}(i\Lambda) + R^-_{\bJ\bL}(i\Lambda)]g_2\|_{L^2_\Lambda}
\Big\|_{L^\infty_x}
\\
&&
\leq
C
\|
\check{g_1}\|_{L^2_\Lambda} \sup_{\Lambda\in\R}
\big\|R^\pm_{\bJ\bL}(i\Lambda)\big\|_{L^1_3\to L^\infty_{-3}}
\|g_2\|_{(L^1_3)_x}
\leq C \|g_1\|_{L^2_t} \|g_2\|_{(L^1_3)_x}.
\end{eqnarray*}
All in all, we have shown the required estimate \eqref{a510} for the case $F=g_1(t) g_2(x)$. Note however that the domain space
$(L^1_3)_x L^2_t$ may be embedded in the bigger space
$(\mathscr{M}_3)_x L^2_t$, where $\mathscr{M}_3$
is the space of Borel measures with the weight $\langle x\rangle^3$.
By the Krein--Milman theorem, elements of this space
may be represented as weak* limits of linear combinations of Dirac masses of the form
$\delta(x-a) g(t)$. Thus, to show bounds
of the form $T:\;(\mathscr{M}_3)_x L^2_t\to Y$
for any linear operator $T$ and Banach space $Y$,
it suffices to prove such an estimate for elements
$F=g_2(x) g_1(t)$,
with $g_2\in \mathscr{M}_3$, $g_1\in L^2$ as we have done above.
\end{proof}

\subsection{Further linear estimates for $e^{t\bJ\bL}$}
\label{sect-further-linear-estimates}

We will now state and derive the Strichartz estimates.
\begin{definition}
We say that a pair $(q,r)$ is Strichartz-admissible (for the Dirac equation in one spatial dimension), if
$$
q\geq 2,
\qquad
r\geq 2,
\qquad
\frac{2}{q}+\frac{1}{r}\leq \frac{1}{2}.
$$
Equivalently, the admissible set is a
triangle in the $(\frac{1}{q}, \frac{1}{r})$ plane,
with endpoints corresponding to $(q,r)=(4, \infty)$
and $(q,r)=(\infty,2)$.
\end{definition}
In view of the representation of the Strichartz-admissible set as a triangle in the $(\frac{1}{q}, \frac{1}{r})$ coordinates,
we will state the estimates only at the vertices,
with the estimates in the interior of the triangle
obtained by interpolation.

Next, before we can state our Strichartz type estimates, we need a variant of the well-known Christ--Kiselev lemma,
an abstract result which
allows one to pass between estimates for dual operators
and retarded terms in the Duhamel representation.
We state a version which is due to Smith and Sogge \cite{MR1789924}.
\begin{lemma}
\label{kiselev}
Let $X,\,Y$ be Banach spaces and
$\ck:L^p(\R;X)\to L^q(\R,Y)$ be a bounded linear operator such
that $\ck f(t)=\int_{-\infty}^\infty K(t,s) f(s)\,ds$.
Then the operator
\begin{equation}
\label{a700}
\tilde{\ck} f(t)=\int_0^t K(t,s) f(s)\,ds
\end{equation}
is bounded from $L^p(\R;X)$ to $L^q(\R,Y)$, provided that $p<q$. Moreover, there is
$C_{p,q} > 0$ such that
$$
\|\tilde{\ck}\|_{L^p(\R;X)\to L^q(\R,Y)} \leq C_{p,q}
\|\ck\|_{L^p(\R;X)\to L^q(\R,Y)}.
$$
\end{lemma}

\begin{lemma}
\label{lemma30}
Let $(q,r)$ be a Strichartz-admissible pair.
Then, for any $\epsilon>0$ and $s\geq 0$,
there is $C_\epsilon<\infty$ so that
\begin{eqnarray}
\label{a650}
&&\big\|
e^{t\bJ\bL} P_c(\omega)f
\big\|_{L^4_t L^\infty_x}\leq
C \|f\|_{H^{3/4+\epsilon}}, \\
\label{a660}
& &
\big\|e^{t\bJ\bL} P_c(\omega)f
\big\|_{L^\infty_t H^s_x}\leq
C \|f\|_{H^{s}}, \\
\label{a665}
&&
\Big\|
\int_{-\infty}^\infty e^{\tau\bJ\bL} P_c(\omega)F(\tau, \cdot)
\Big\|_{L^\infty_t H^1_x \cap L^q_t L^r_x}\leq \|F\|_{L^1_t H^1_x}, \\
\label{a670}
&&
\Big\|
\int_0^t e^{(t-\tau)\bJ\bL} P_c(\omega)F(\tau, \cdot)
\Big\|_{L^\infty_t H^1_x \cap L^q_t L^r_x}\leq \|F\|_{L^1_t H^1_x}.
\end{eqnarray}
\end{lemma}
\begin{proof}
We start with the estimates \eqref{a650} and \eqref{a660}. Let us note that we can easily upgrade \eqref{a650} to add derivatives on the evolution.
An interpolation between these two estimates
then yields (Cf. \ref{a850} below for the free Dirac case):
\begin{equation}
\label{a905}
\|e^{t\bJ\bL} P_c(\omega)f\|_{L^q_t W^{s,r}_x}\leq
C_\epsilon
\|f\|_{H^{s+\frac{1}{2}+\frac{1}{q}-\frac{1}{r}+\epsilon}},
\end{equation}
for $s\geq 0$ and for all Strichartz-admissible pairs $(q,r)$.

\medskip

The proof of \eqref{a670} is based on an application of the dual to \eqref{a905} and Lemma~\ref{kiselev}. Thus, it remains to show \eqref{a650} and \eqref{a670}. The approach follows what has become standard in recent years: we employ the available results for the ``free'' Dirac operator, in addition to the weighted decay estimates that we have proved in the previous section, namely Proposition~\ref{prop12} and Lemma~\ref{lemma20}. In fact, we follow closely the approach in \cite[Lemma 4]{MR2985264}.

Let us recall first the estimates for the free Dirac operator.
Let us prove the Strichartz estimates
for $e^{i t D_m}$ in the form \eqref{a650}, \eqref{a660}, \eqref{a670}.
The corresponding linear equations
$$
i\p_t h_1 = h_1+\p_x h_2,
\qquad
i \p_t h_2=-\p_x h_1-h_2
$$
reduce to the Klein--Gordon equation for each component $h_1, h_2$, as we have shown in the proof of Lemma~\ref{lemma133}. Thus, the ``free'' Dirac estimates follow from the respective estimates for the Klein--Gordon equation, which can be found in the recent work of Nakamura--Ozawa, \cite[Lemma 2.1]{MR1855424} (where one takes
$\theta = 1$, $\Lambda = 3/2$, $n = 1$).
These estimates read as follows: for every $\epsilon>0$,
\begin{equation}
\label{a850}
\|e^{-i t D_m} f\|_{L^q_t W^{s,r}_x}\leq C_\epsilon \|f\|_{H^{s+\frac{1}{2}+\frac{1}{q}-\frac{1}{r}+\epsilon}}.
\end{equation}
These are of course the variants of the estimates \eqref{a650} and \eqref{a660}; the estimate \eqref{a670} holds in a similar manner for the free Dirac case. One important improvement of \eqref{a850}, which is implicit in \cite{MR1855424},\footnote{This is the estimate $(2.15)$ in \cite{MR1855424},
which holds with the homogeneous Besov spaces version}
concerns the low frequency component of $f$. Namely, for the particular case $q=4, r=\infty$, we have:
\begin{equation}
\label{a851}
\|e^{-i t D_m} f\|_{L^4_t L^\infty_x}\leq C_\epsilon \||\p_x|^{3/4} f\|_{H^{\epsilon}}.
\end{equation}

Let us now consider
$\bJ\bL=\bJ(\bL_0+\bW)$, with a potential $\bW$ of Schwartz class.
We may write the perturbed evolution in terms of the free evolution as follows:
$$
e^{t \bJ\bL} f = e^{t \bJ\bL_0}f + \int_0^t e^{(t-s)\bJ\bL_0} \bJ\bW e^{s \bJ\bL} f\,ds.
 $$
 We now have to deal with the two endpoint cases of Strichartz pairs:
$(q,r)=(4, \infty)$ and $(q,r)=(\infty,2)$. We only present the first case, the second being similar.
To that end,
let $\bW(x)=V_1(x) V_2(x)$,
with
$V_1(x)=e^{-\updelta_\varOmega \langle x\rangle}$
and $V_2(x)=e^{\updelta_\varOmega \langle x\rangle} \bW(x)$,
with
\begin{equation}\label{def-updelta-varomega}
\updelta_\varOmega
=\inf\sb{\omega\in\varOmega}\updelta_\omega
=\inf\sb{\omega\in\varOmega}\sqrt{1-\omega^2}>0
\end{equation}
so that $V_2(x)$ is also exponentially decaying
(Cf. \eqref{w-small}).
For
$f\in H^{\frac{3}{4}+\epsilon}$,
we have:
 \begin{eqnarray*}
\|e^{t\bJ\bL} P_c(\omega) f\|_{L^4_t L^\infty_x}&\leq &
 \|e^{t\bJ\bL_0} f\|_{L^4_t L^\infty_x}+ \left\| \int_0^t e^{(t-s)\bJ\bL_0}
\bJ V_1 V_2 e^{s\bJ\bL} P_c(\omega) f\,ds \right\|_{L^4_t L^\infty_x} \\
&\leq&C_\ve \|f\|_{H_x^{3/4+\ve}}+ \left\| \int_0^t e^{(t-s)\bJ\bL_0}
 \bJ V_1 V_2 e^{s\bJ\bL} P_c(\omega) f\,ds \right\|_{L^4_t L^\infty_x}.
\end{eqnarray*}
We now use the Christ--Kiselev lemma (Lemma~\ref{kiselev})
with
$K(t,s)=e^{(t-s)\bJ\bL_0} \bJ V_1:
L^2_t H^{\frac{3}{4}+\epsilon}\to L^4_t L^\infty_x$.
Following \eqref{a700},
$$
\tilde{\ck} [V_2 e^{s\bJ\bL} P_c(\omega) f]
=
\int_0^t e^{(t-s)\bJ\bL_0}
\bJ V_1 V_2 e^{s\bJ\bL} P_c(\omega) f\,ds.
$$
According to Lemma~\ref{kiselev}, we have
\begin{eqnarray*}
\|\int_0^t e^{(t-s)\bJ\bL_0}
 \bJ V_1 V_2 e^{s \bJ\bL} P_c(\omega) f\,ds\|_{L^4_t L^\infty_x}
=
\|\tilde{\ck}
[V_2 e^{s\bJ\bL} P_c(\omega) f]\|_{L^4_t L^\infty_x}
\\
\leq
C \|\ck\|_{L^2_t H^{\frac{3}{4}+\epsilon}\to L^4_t L^\infty_x}
 \|V_2 e^{t\bJ\bL} P_c(\omega) f\|_{L^2_t H^{\frac{3}{4}+\epsilon}}.
\end{eqnarray*}
From the interpolation between the cases $s=0$ and $s=1$,
the decay and smoothness properties of $V_2$
and the weighted decay estimate from Proposition~\ref{prop12},
we conclude that
$
\|V_2 e^{t\bJ\bL} P_c(\omega) f\|_{L^2_t H^{s}_x}\leq C\|f\|_{H^s},
$
and we arrive at the estimate
$\|V_2 e^{t\bJ\bL} P_c(\omega) f\|_{L^2_t H^{\frac{3}{4}+\epsilon}}
\leq C \|f\|_{H^{\frac{3}{4}+\epsilon}}$.

\medskip

It remains to obtain the appropriate estimate for
$\|\ck\|_{L^2_t H^{\frac{3}{4}+\epsilon}\to L^4_t L^\infty_x} $. We have again by the Strichartz estimates for the free Dirac evolution
(more precisely, the version of \eqref{a851}):
\begin{eqnarray*}
 \|\int_{-\infty}^\infty e^{(t-s)\bJ\bL_0}\bJ V_1 G(s, \cdot)\,ds\|_{L^4_t L^\infty_x} &=& \|e^{t\bJ\bL_0} \int_{-\infty}^\infty e^{-s \bJ\bL_0} \bJ V_1 G(s, \cdot)\,ds\|_{L^4_t L^\infty_x}\\
&\leq&C\||\p_x|^{3/4}\int_{-\infty}^\infty e^{-s \bJ\bL_0} \bJ V_1 G(s, \cdot)\,ds\|_{H^{\epsilon}}.
\end{eqnarray*}
 From Lemma~\ref{lemma133} (and more precisely from \eqref{a801}), we have
$$
\||\p_x|^{3/4}\int_{-\infty}^\infty e^{-s\bJ\bL_0} \bJ V_1 G(s, \cdot)\,ds\|_{H^\epsilon}\leq
C
\||\p_x|^{3/4} M [\bJ V_1 G(s)]\|_{L^1_x H^{\epsilon}}.
$$
Note that in the low frequencies,
$|\p_x|^{3/4} M\sim |\p_x|^{1/4}$ is not singular anymore,
while in the high frequencies one has
$|\p_x|^{3/4} M\sim |\p_x|^{3/4}$.
Thus, with $V_1$ in the Besov space $B_2^{1,1}$, we have
$$
\||\p_x|^{3/4} M [\bJ V_1 G(s)]\|_{L^1_x H^{\epsilon}_s}\leq C \|V_1\|_{B_2^{1,1}} \|G\|_{L^2_x H^{3/4+\epsilon}}.
$$
With that, Lemma~\ref{lemma30} is proved in full.
\end{proof}

Our next lemma is another essential component of the fixed point arguments to be presented in Section~\ref{sec:pmt}. Namely, it connects the Strichartz estimates to the weighted decay estimates.
\begin{lemma}
\label{lemma50}
There is $C<\infty$ such that
\begin{eqnarray}
\label{a910}
&&\|\int_0^t e^{(t-\tau)\bJ\bL}
P_c(\omega) F(\tau, \cdot)\,d\tau\|_{L^\infty_t H^1_x \cap L^4_t L^\infty_x} \leq C
[\|F\|_{(L^1_3)_x L^2_t}+\|\p_x F\|_{(L^1_3)_x L^2_t}],
\\
\label{a970}
&&\sup_x \langle x\rangle^{-3}
\|\int_0^t e^{(t-\tau)\bJ\bL} P_c(\omega) F(\tau, \cdot)\,d\tau\|_{L^2_t}
\leq C \|F\|_{L^1_t L^2_x}.
\end{eqnarray}
\end{lemma}

\begin{proof}
For the proof of \eqref{a910}, by Lemma~\ref{kiselev},
we may consider the Duhamel's operator in the form $\int_{-\infty}^\infty \ldots$,
instead of the retarded term with $\int_0^t \ldots$ in the Duhamel representation.
By \eqref{a650} and \eqref{a660},
\begin{eqnarray*}
\Big\|\int_{-\infty}^\infty e^{(t-\tau)\bJ\bL} P_c(\omega)
F(\tau, \cdot)\,d\tau
\Big\|_{L^\infty_t H^1_x \cap L^4_t L^\infty_x}
&=&\Big\|
e^{t\bJ\bL}P_c(\omega)
\int_{-\infty}^\infty e^{-\tau \bJ\bL} F(\tau, \cdot)\,d\tau
\Big\|_{L^\infty_t H^1_x
\cap L^4_t L^\infty_x} \\
&\leq&
\Big\|
\int_{-\infty}^\infty e^{-\tau\bJ\bL} P_c(\omega) F(\tau, \cdot)\,d\tau
\Big\|_{H^1_x}.
\end{eqnarray*}
To prove \eqref{a910},
we need to estimate two terms:
one with a derivative and one without a derivative.
The term without a derivative
is dealt with by Proposition~\ref{prop12}:
\begin{equation}
\label{a950}
\Big\|
\int_{-\infty}^\infty e^{-\tau \bJ\bL}P_c(\omega)F(\tau, \cdot)\,d\tau
\Big\|_{L^2_x}
\leq C \| F\|_{ (L^1_3)_x L^2_t}.
\end{equation}
For the term
$\|\int_{-\infty}^\infty\p_x[e^{-\tau \bJ\bL}P_c(\omega)F(\tau,\cdot)]
d\tau\|_{L^2_x}$,
we are facing a difficulty since $\p_x e^{-\tau \bJ\bL}\neq e^{-\tau\bJ\bL} \p_x$. Nevertheless, due to the fact that
$
\bL
=\bD_m-\omega I_4+\bW,
$
we use the $L^2_x$ estimate \eqref{a950} to derive
\begin{eqnarray}
\label{aaa}
\Big\|\int\limits_{-\infty}^\infty \p_x
\big[e^{-\tau \bJ\bL} P_c(\omega) F(\tau, \cdot) \big]
d\tau
\Big\|_{L^2_x}
&\leq&
\Big\|\int\limits_{-\infty}^\infty
\big(
\bL-\bmupbeta+\omega I_4-\bW
\big)
\big[
e^{-\tau \bJ\bL} P_c(\omega) F(\tau, \cdot)
\big]
d\tau
\Big\|_{L^2_x}
\nonumber
\\
&\leq&
C
\Big\{
\|\bL F\|_{(L^1_3)_x L^2_t}
+
\big(1+\abs{\omega}+\|\bW\|_{L^\infty_x}\big)
\|F\|_{(L^1_3)_x L^2_t}
\Big\}.
\end{eqnarray}
Taking into account the specific form of $\bJ\bL$,
it follows from
\eqref{a950}
and
\eqref{aaa}
that
$$
\Big\|
\int_{-\infty}^\infty e^{-\tau \bJ\bL} g_1(\tau) P_c(\omega) g_2
d\tau
\Big\|_{H^1_x}\leq C
[\|F\|_{(L^1_3)_x  L^2_t}+\|\p_x F\|_{(L^1_3)_x L^2_t}].
$$
We now turn to proving \eqref{a970}
Because of the weak* density of linear combinations
$\{\delta(t-\tau_0)G(x): \tau_0\in\rone, G\in L^2_x(\rone)\}$ in $L^1_t L^2_x$,
it suffices to prove
\eqref{a970}
for $F(x)=\delta(t-\tau_0)G(x)$.
By Proposition~\ref{prop12},
\begin{eqnarray*}
\sup_x \langle x\rangle^{-3}
\Big\|
\int_0^t e^{(t-\tau)\bJ\bL} P_c(\omega)\delta(\tau-\tau_0) G(x)\,d\tau
\Big\|_{L^2_t}
&=&\sup_x \langle x\rangle^{-3} \| e^{(t-\tau_0)\bJ\bL} P_c(\omega) G(x) \|_{L^2_t}
\\
&=&\sup_x \langle x\rangle^{-3} \| e^{t \bJ\bL} P_c(\omega) G(x) \|_{L^2_t}
\leq C \|G\|_{L^2_x}.
\qedhere
\end{eqnarray*}
\end{proof}

\section{Proof of the Main Theorem}\label{sec:pmt}

In this section,
the constants $C$ may change from one instance to another;
they all depend only on $\varOmega$
and on the nonlinearity $f$ in \eqref{GN.eqn}.

\subsection{Modulation equations}\label{modul}
We consider the solution $\psi$ of equation \eqref{N.eqn.gn}
in the form
\begin{equation}\label{eqn.dec}
\psi(x, t) =
\big(\phi_{\omega(t)}(x) + \rho(x,t)\big)e^{-i\theta(t)}, \qquad
\text{with} \qquad \theta(t) = \int_0^t \omega(s)\,ds + \gamma(t),
\qquad x , t \in\R.
\end{equation}
Substituting this Ansatz into \eqref{N.eqn.gn}, we get
\begin{equation}\label{rho.eqn-2}
i\p_t \rho =
(D_m-I_2\omega -\dot{\gamma}I_2)\rho -\dot{\gamma} \phi -i\dot{\omega}\p\sb\omega\phi
+ \bN(\phi + \rho)-\bN(\phi),
\end{equation}
with $\bN$ defined in \eqref{def-n-gn}.
As in
\eqref{R.eqn.gn}, \eqref{def-upphi},
we use the notations
\[
R=\begin{bmatrix}\Re\rho\\\Im\rho\end{bmatrix},
\qquad
\bmupphi\sb\omega
=\begin{bmatrix}\Re\phi\sb\omega\\\Im\phi\sb\omega\end{bmatrix}
=\begin{bmatrix}\phi\sb\omega\\0\end{bmatrix}.
\]
Then equation \eqref{rho.eqn-2} takes the form
\begin{equation}\label{eqn:all}
\p_t R
=\bJ \bL R -\dot\gamma \bJ R
-\dot\gamma\bJ\bmupphi
-\dot\omega\p\sb\omega\bmupphi
+\bJ \bN_1,
\end{equation}
where
\begin{equation}\label{N1.def}
\bN_1(R, \omega) =
\begin{bmatrix}\Re \big(\bN(\phi +\rho)-\bN(\phi)\big)
\\\Im \big(\bN(\phi +\rho)-\bN(\phi)\big)\end{bmatrix}
- \bW R,
\end{equation}
with $\bW$ from \eqref{def-w.gn}.

\begin{remark}
Let us point out that since we take
the initial data of certain parity,
$\psi\at{t=0}\in X$,
then we also have
$\psi\in X$ for all $t\ge 0$,
so that $\rho\in X$;
therefore,  $R\in\bX$ and $\bJ\bN_1\in\bX$
(see Definitions~\ref{def-x},~\ref{def-xx}).
Moreover, the operators
$\bJ\bL(\omega)$, $P_d(\omega)$, and $P_c(\omega)$
act invariantly in $\bX$.
\end{remark}

We impose the requirement $R(t) \in \bX_c(\omega(t))$.
Together with the symplectic orthogonality condition
\eqref{symplectic-orth},
this requirement implies that
\begin{equation}\label{i-t}
\langle\bmupphi, R\rangle=
\langle
\bJ\p\sb\omega\bmupphi, R\rangle=0.
\end{equation}

Taking the time derivative of the relations \eqref{i-t},
we get
\begin{equation}
\label{dot-R-orth}
\langle\bmupphi, \dot{R} \rangle=
-\dot{\omega} \langle \partial_\omega \bmupphi, R \rangle = -\dot\omega \Re \langle
\varphi, \rho \rangle,
\qquad
\langle\bJ\p\sb\omega\bmupphi, \dot{R}\rangle=
-\dot{\omega} \langle\bJ\p\sb\omega^2\bmupphi, R\rangle =
\dot \omega \Im \langle \partial_\omega \varphi, \rho \rangle,
\end{equation}
where
\[
\varphi_\omega = \partial_\omega \phi_\omega.
\]
Coupling \eqref{eqn:all}
with $\bmupphi$ and with $\bJ\p\sb\omega\bmupphi$
and using the symplectic relations \eqref{symplectic-orth}
and the relations \eqref{dot-R-orth}, we obtain
\begin{equation}\label{para-eqns}
\mathcal{A}(t)
\begin{bmatrix} \dot{\omega}\\ \dot{\gamma}\end{bmatrix}
=
\begin{bmatrix}
 \wei{\bmupphi, \bJ \bN_1}\\
 \wei{\bJ \partial_\omega \bmupphi, \bN_1}
\end{bmatrix},
\end{equation}
where
\begin{equation}\label{A-B.def}
\mathcal{A}(t)
=
\begin{bmatrix}
 \wei{\bmupphi, \partial_\omega \bmupphi} - \wei{\partial_\omega \bmupphi, R} & \wei{\bmupphi, \bJ R} \\
-\wei{\bJ \p_\omega^2 \bmupphi, R} & \wei{\bJ \p_\omega \bmupphi, \bJ \bmupphi}
+\wei{\bJ \p_\omega \bmupphi, \bJ R}
\end{bmatrix},
\end{equation}
where
$\omega$ and $R$ are evaluated at the moment $t$.

Define
\begin{equation}\label{def-upmu}
\upmu(x)
:=
e^{-\updelta_\varOmega\langle x\rangle/(4k)},
\qquad
\updelta_\varOmega:=\inf\sb{\omega\in\varOmega}
\sqrt{1-\omega^2}>0;
\end{equation}
by \eqref{phi-bounded.gn} and \eqref{phi-bounded.gn-2},
there is $C<\infty$
such that for any $\omega\in\varOmega$, 
\begin{equation}\label{upmu-small}
\abs{\phi\sb\omega(x)}
+\abs{\p\sb\omega\phi\sb\omega(x)}
+\abs{\p\sb\omega^2\phi\sb\omega(x)}
\le C\upmu(x)^{2k},
\qquad
x\in\R,\quad\omega\in\varOmega.
\end{equation}

\begin{lemma}\label{lemma-a-invertible}
There is $\epsilon_0>0$ such that
if $\langle\upmu,\abs{R(t)}\rangle<\epsilon_0$,
then
\[
\norm{\mathcal{A}(t)^{-1}}
<
2\Big(
\inf\sb{\omega\in\varOmega}
\langle\phi\sb\omega,\p\sb\omega\phi\sb\omega\rangle
\Big)^{-1}<\infty.
\]
\end{lemma}

\begin{proof}
From \eqref{A-B.def}
and \eqref{upmu-small},
we have
\[
\mathcal{A}(t)= \begin{bmatrix}
\langle\phi\sb\omega,\p\sb\omega\phi\sb\omega\rangle
&0 \\
0&
\langle\phi\sb\omega,\p\sb\omega\phi\sb\omega\rangle
\end{bmatrix}
+\mathcal{O}\big(\langle\upmu^{2k},\abs{R}\rangle\big),
\]
where $\omega=\omega(t)$;
we took into account the bounds
\eqref{phi-bounded.gn}
and \eqref{phi-bounded.gn-2}. 
By Assumption~\ref{ass-1},
one has
$2\langle\phi_\omega,\p\sb\omega\phi\sb\omega\rangle= Q'(\omega)$
with $\inf\sb{\omega\in\varOmega}\abs{Q'(\omega)}>0$;
therefore,
one can choose $\epsilon_0>0$
so small
that $\mathcal{A}(t)$ is invertible
and satisfies the conclusion of the lemma.
\end{proof}

To control $\rho$ (or equivalently $R$), let us define
\begin{equation}\label{def-zt}
Z(t) = P_c(\omega_0)R(t),
\end{equation}
so that
\begin{equation}\label{def-z0}
Z(0) = P_c(\omega_0) R(0) = P_c(\omega_0) \begin{bmatrix} \Re \big( \Psi_0-\phi_0 \big) \\ \Im\big( \Psi_0-\phi_0 \big) \end{bmatrix}.
\end{equation}
Since
$Z=P_c(\omega_0) R$
and
$R=P_c(\omega) R$, and by \eqref{def-pd}, 
we have
\begin{equation}\label{Z_R}
Z-R
=P_c(\omega_0)R-P_c(\omega)R
=\big(
P_d(\omega)-P_d(\omega_0)
\big)R = \mathcal{O}(\omega - \omega_0) \wei{\upmu^{2k}, |R|} \upmu^{2k}.
\end{equation}
Therefore, if $|\omega -\omega_0|$ is sufficiently small, to control $R$, it suffices to control $Z$;
in particular, it follows from \eqref{Z_R} that
if either $Z$ or $R$ is from $H^1$ in $x$,
then so is the other function,
and moreover
\begin{equation}\label{r-z-same}
\norm{Z-R}\sb{H^1_x}
\le C\abs{\omega-\omega_0}
\langle\upmu^{2k},\abs{R}\rangle,
\end{equation}
with some constant $C<\infty$
which depends only on $\varOmega$
and on the nonlinearity $f$ in \eqref{GN.eqn}.
The weight
$
\upmu(x)^{2k}=e^{-\updelta_\varOmega\langle x\rangle/2}
$
(Cf. \eqref{def-upmu})
comes from the bounds \eqref{upmu-small}
on the eigenfunctions that span
the generalized null space
\eqref{n-ng} of the operator $\bJ\bL(\omega)$
and from the explicit form \eqref{def-pd}
of the projector $P_d(\omega)$.

Let us estimate
the right-hand side in \eqref{para-eqns}.

\subsection{Closing the estimates}
\label{sect-closing}

Now we will analyze
the modulation equations \eqref{para-eqns}
and the PDE \eqref{Z-eqn-new}.
We will assume that $\epsilon>0$ is sufficiently small
and that
\[
\psi_0 e^{i\theta_0} = \phi_{\omega_0} + \rho_0,
\qquad \rho_0 \in \X_c(\omega_0),
\qquad \theta_0 \in \mathbb{R},
\qquad
\|\rho_0\|_{H^1} \leq \epsilon^2.
\]
Without loss of generality, we assume that $\theta_0 =0$.

\begin{definition}\label{def-x-y}
For fixed $N > 10$ and $T>0$, let
\[
\begin{split}
\norm{Z}_{\mathcal{X}_T}&=\norm{Z}_{L^4_t L^\infty_x} + \norm{Z}_{L_t^\infty H^{1}_x}
+ \norm{\wei{x}^{-N}Z}_{L^\infty_x L^2_t} + \norm{\wei{x}^{-N}\p_x Z}_{L^\infty_x L^2_t}, \\
\norm{F}_{\mathcal{Y}_T}&=
\inf_{F = A + B} \Big[\norm{A}_{L^1_t H^1_x} +\norm{\wei{x}^{N}B}_{L_x^1L_t^2}
+ \norm{\wei{x}^{N} \p_x B}_{L_x^1L_t^2} \Big],
\end{split}
\]
where $L_t^\alpha = L^\alpha[0,T]$ and $L^\alpha_{x} = L^\alpha(\mathbb{R})$.
\end{definition}

\begin{lemma}\label{lemma-sim-nn1}
There is $C<\infty$
such that for each $\omega_0 \in\varOmega$
there is
$\epsilon_0\in\big(0,\mathop{\rm dist}(\omega_0,\p\varOmega)\big)$
such that if
$\omega$ and $Z\in H^1(\R,\C^4)$ satisfy
$
|\omega-\omega_0| < \epsilon_0,
$
$\norm{\langle x\rangle^{-N} Z}\sb{H^1_x}\le\epsilon_0$
with $N>10$ from Definition~\ref{def-x-y},
then
\[
|\wei{\bmupphi, \bJ \bN_1(R,\omega)}|
+|\wei{\bJ \p_\omega \bmupphi, \bN_1(R,\omega)}|
\leq C
\langle\upmu,\abs{Z}^2\rangle,
\]
where
$\bN_1(R,\omega)$ is from \eqref{N1.def},
$R=P_c(\omega_0)R$,
and $Z=P_c(\omega_0)R$.
\end{lemma}

\begin{proof}
    From \eqref{N1.def},
Taylor's expansion, and Young's inequality,
we see that
\begin{equation}\label{N1-est}
\bN_1 = \bN(\phi + \rho)-\bN(\phi)
-\bW R
=\mathcal{O}\big(\abs{\phi}^{2k-1}|R|^2+|R|^{2k+1}\big).
\end{equation}
Note that the above makes sense pointwise in $x\in\R$
since
$Z\in H^1(\R,\C^4)$,
and by \eqref{r-z-same} so is $R$.

By \eqref{upmu-small},
this leads to
\begin{equation}\label{l-t}
\langle\phi,\abs{\bN_1}\rangle
\le
C
\langle\upmu,\abs{R}^2\rangle
\big(1+\norm{\upmu R}_{L^\infty_x}^{2k-1})
\big)
\le
C
\langle\upmu,\abs{Z}^2\rangle
\big(1+\norm{\upmu Z}_{H^1_x}^{2k-1})
\big).
\end{equation}
Let us explain the last inequality.
By \eqref{Z_R} and the triangle inequality,
\[
-\abs{(P_d(\omega)-P_d(\omega_0))R}
\le
\abs{Z}-\abs{R}\le \abs{(P_d(\omega)-P_d(\omega_0))R},
\qquad
x\in\R;
\]
multiplying the above by $\abs{R}+\abs{Z}$
and
coupling the result with $\upmu$, we have
\[
-C\abs{\omega-\omega_0}\langle\upmu,\abs{R}^2+\abs{Z}^2\rangle
\le
\langle\upmu,\abs{Z}^2\rangle-\langle\upmu,\abs{R}^2\rangle
\le
C\abs{\omega-\omega_0}\langle\upmu,\abs{R}^2+\abs{Z}^2\rangle.
\]
It follows that if $\abs{\omega-\omega_0}$ is
sufficiently small, then
\[
\frac 1 2\langle\upmu,\abs{Z}^2\rangle
\le
\langle\upmu,\abs{R}^2\rangle
\le
2\langle\upmu,\abs{Z}^2\rangle.
\]
Since $\norm{\langle x\rangle^{-N}Z}_{H^1_x}\le\epsilon_0$,
we have
$\norm{\upmu Z}_{H^1_x}
\le C$;
therefore, the inequality \eqref{l-t}
finishes the proof.
\end{proof}

Applying the projection $P_c(\omega_0)$ to equation \eqref{eqn:all}, we obtain:
\begin{eqnarray}\label{Z-eqn}
&&
\p_t Z-\bJ\bL(\omega_0)Z
+
\big(
\dot{\gamma}(t) + \omega(t)-\omega_0
\big)
P_c(\omega_0)\bJ Z
\nonumber
\\
&&
= P_c(\omega_0)
\big(
\bJ (\bW(\omega)-\bW(\omega_0)) R
-\dot{\gamma} \bJ\bmupphi
-\dot{\omega}\p\sb\omega\bmupphi_\omega+\bJ \bN_1
\big).
\end{eqnarray}
We denote
\begin{equation}\label{def-alpha}
\alpha(t) =\dot{\gamma}(t) + \omega(t)-\omega_0
\end{equation}
and
\begin{equation}\label{F0.def}
F_0(t)
=
\bJ \big(\bW(\omega)-\bW(\omega_0)\big) R
-\dot{\gamma}\bJ\bmupphi
-\dot{\omega}\p\sb\omega\bmupphi_\omega+\bJ \bN_1;
\end{equation}
then \eqref{Z-eqn} takes the form
\begin{equation}\label{Z-eqn-new}
\p_t Z-\bJ\bL(\omega_0)Z
+\alpha(t)\bJ Z
=P_c(\omega_0)F_0+\alpha(t)[\bJ,P_c(\omega_0)]Z.
\end{equation}
We assume that there exist $T>0$ and $C_0>1$
depending on $\omega_0$ such that the solution
$(\omega(t), \gamma(t), Z(t))$
to the modulation equations \eqref{para-eqns}
and the PDE \eqref{Z-eqn-new}
exists on $[0, T]$ and
\begin{equation}\label{assum-close}
\|\dot{\omega}\|_{L^1[0,T]} + \| \dot{\gamma}\|_{L^1[0,T]} \leq C_0 \epsilon, \qquad
\norm{Z}_{\mathcal{X}_T} \leq C_0 \epsilon.
\end{equation}
\begin{lemma}\label{lemma-closing} Assume that \eqref{assum-close} holds. 
If $\epsilon>0$
is sufficiently small,
then the estimates \eqref{assum-close} can be improved
as follows:
\begin{equation}\label{closing-a}
\|\dot{\omega}\|_{L^1[0,T]} + \| \dot{\gamma}\|_{L^1[0,T]} \leq \epsilon,
\qquad \norm{Z}_{\mathcal{X}_T} \leq \epsilon.
\end{equation}
\end{lemma}

\begin{proof}
By \eqref{para-eqns}, the invertibility of
$\mathcal{A}(t)$ (Cf. Lemma~\ref{lemma-a-invertible}),
and the bounds from Lemma~\ref{lemma-sim-nn1},
we conclude that
\[
\abs{\dot\gamma}+\abs{\dot\omega}
\le C\langle\upmu,\abs{Z(t)}^2\rangle,
\]
hence, for small enough $\epsilon>0$,
\begin{equation}\label{para-dot-L1}
\norm{\dot\omega}_{L^1_t[0,T]}
+\norm{\dot\gamma}_{L^1_t[0,T]}
\le
C
\int_0^T
\langle\upmu,\abs{Z(t)}^2\rangle
\,dt
\leq C\norm{\upmu^{1/3}Z}_{L^\infty_x  L^2_t}^2
\leq C\norm{Z}_{\mathcal{X}_T}^2\leq C C_0^2 \epsilon^2
\leq \epsilon;
\end{equation}
we used the bound on $\norm{Z}\sb{\mathcal{X}_T}$
from \eqref{assum-close}.
This proves the first estimate in \eqref{closing-a}.

With \eqref{para-dot-L1}, we also have
\begin{equation}\label{para-Linfty}
\norm{\omega -\omega_0}_{L^\infty_t[0,T]}
\le
\norm{\dot\omega}_{L^1_t[0,T]}
\leq C\norm{Z}_{\mathcal{X}_T}^2
\leq C C_0^2\epsilon^2\le\epsilon.
\end{equation}
It follows from \eqref{Z-eqn-new}
with the initial data \eqref{def-z0},
\eqref{para-Linfty}, and from Lemma~\ref{X.est} below
that if $\epsilon>0$ is sufficiently small, then
\begin{equation}\label{Z-eqn.est}
\norm{Z}_{\mathcal{X}_T} \leq
C\Big[\norm{Z(0)}_{H^1} + \norm{F}_{\mathcal{Y}_T}\Big],
\qquad
F(t):=P_c(\omega_0)F_0(t)+\alpha(t)[\bJ,P_c(\omega_0)]Z(t).
\end{equation}
    From the definition \eqref{F0.def} of $F_0$,
we see that
\begin{equation}
\norm{F_0-\bJ\bN_1}_{\mathcal{Y}_T}
\leq
C\Big[
\norm{\omega-\omega_0}_{L^\infty_t[0,T]}
\norm{Z}_{\mathcal{X}_T}
+\norm{\dot{\gamma}}_{L^\infty_t[0,T]}
+\norm{\dot{\omega}}_{L^\infty_t[0,T]}
\Big].
\end{equation}
We used the bound
$\norm{R}\sb{\mathcal{X}_T}
\le
C\norm{Z}\sb{\mathcal{X}_T}
$
which follows from \eqref{r-z-same}.
Noting that $[\bJ, P_c(\omega_0)]$ is localized in space
and recalling that
$\alpha(t)=\dot\gamma(t)+\omega(t)-\omega_0$,
we also have
\begin{eqnarray}\label{w-a-h}
&&
\norm{P_c(\omega_0)(F_0-\bJ\bN_1)+\alpha(t)[\bJ,P_c(\omega_0)Z]
}_{\mathcal{Y}_T}
\nonumber
\\
&&
\leq C\Big[(\norm{\dot{\gamma}}_{L^\infty_t[0,T]}
+\norm{\omega-\omega_0}_{L^\infty_t[0,T]})\norm{Z}_{\mathcal{X}_T}
+\norm{\dot{\gamma}}_{L^\infty_t[0,T]}
+\norm{\dot{\omega}}_{L^\infty_t[0,T]}
 \Big].
\end{eqnarray}
By Lemma~\ref{lemma-sim-nn1},
\begin{equation}\label{para-dot-Linfty}
\abs{\dot\omega(t)}+\abs{\dot\gamma(t)}
\leq C\norm{Z(t)}_{L^2_x}^2
\leq C\norm{Z}_{\mathcal{X}_T}^2
\leq C C_0^2 \epsilon^2
\leq \epsilon,
\qquad
0\le t\le T,
\end{equation}
as long as $\epsilon>0$ is sufficiently small.
Applying \eqref{para-dot-Linfty} and \eqref{para-Linfty}
in \eqref{w-a-h},
we conclude that
there is $C<\infty$
such that
\begin{equation}\label{F1.est}
\norm{P_c(\omega_0)(F_0-\bJ\bN_1)+\alpha(t)[\bJ,P_c(\omega_0)]Z
}_{\mathcal{Y}_T}
\leq C\norm{Z}_{\mathcal{X}_T}^2.
\end{equation}
By
\eqref{r-z-same} and \eqref{N1-est},
using Young's inequality,
we see that
\[
\bN_1
=\mathcal{O}(\abs{\phi}^{2k-1}|R|^2+|R|^{2k+1})
=\mathcal{O}
\big(
\abs{\phi}^{2k-1}|Z|^2+|Z|^{2k+1}
+\upmu \abs{\omega-\omega_0}^2
\langle\upmu^{2k},\abs{R}\rangle^2
\big).
\]
Then, it follows from \eqref{Z_R} and \eqref{para-Linfty} that 
\begin{equation}\label{Omega.est}
\norm{\bN_1}_{\mathcal{Y}_T}
\leq C
\Big(
\norm{\abs{\phi}^{2k-1}\abs{Z}^2}_{\mathcal{Y}_T}
+\norm{\abs{Z}^{2k+1}}_{\mathcal{Y}_T}
+ C\norm{Z}_{\mathcal{X}_T}^4 \Big).
\end{equation}

On the other hand, from the definitions of
$\norm{\cdot}_{\mathcal{X}_T}, \norm{\cdot}_{\mathcal{Y}_T}$
(Cf. Definition~\ref{def-x-y}),
we observe that
\begin{equation}\label{localized-term}
\norm{\abs{\phi}^{2k-1} \abs{Z}^2}_{\mathcal{Y}_T}
\leq C\norm{\wei{x}^n \abs{\phi}^{2k-1}\abs{Z}^2}_{L^1_x L^2_t} +
C\norm{\wei{x}^N \p_{x}[\abs{\phi}^{2k-1} \abs{Z}^2]} \leq C\norm{Z}_{\mathcal{X}_T}^2 \leq
C\norm{Z}_{\mathcal{X}_T}^2.
\end{equation}
Similarly, we have
\[
\norm{\abs{Z}^{2k+1}}_{\mathcal{Y}_T} \leq C\norm{\abs{Z}^{2k+1}}_{L^1_t H^1_x} \leq
C\norm{(|Z| + |\p_x Z|)\abs{Z}^{2k}}_{L^1_t L^2_x} \leq C\norm{Z}_{L_t^\infty H^1_x}\norm{Z}_{L^{2k}_t L^\infty_x}^{2k}.
\]
We note that
$
\norm{Z}_{L^\infty_t H^1_x}
\leq \norm{Z}_{\mathcal{X}_T};
$
since $k \geq 2$, we arrive at
\[
\norm{Z}_{L^{2k}_t L_x^\infty} \leq \norm{Z}_{L^4_t L^\infty_x}^{2/k} \norm{Z}_{L_t^\infty L_x^\infty}^{1-2/k} \leq C\norm{Z}_{\mathcal{X}_T}.
\]
Therefore,
\begin{equation}\label{2k+1}
\norm{\abs{Z}^{2k+1}}_{\mathcal{Y}_T} \leq C\norm{Z}_{\mathcal{X}_T}^{2k+1}.
\end{equation}
In summary, it follows from \eqref{F1.est}, \eqref{Omega.est}, \eqref{localized-term}, and \eqref{2k+1} that
there is $C<\infty$ such that
\begin{equation*} 
\norm{P_c(\omega_0)F_0+\alpha(t)[\bJ,P_c(\omega_0)]Z}_{\mathcal{Y}_T}
\leq C\norm{Z}_{\mathcal{X}_T}^{2}.
\end{equation*}
    From this and \eqref{Z-eqn.est}, we infer that if $\epsilon>0$ is
sufficiently small, then we have
\begin{equation}\label{Z.last.est}
\norm{Z}_{\mathcal{X}_T} \leq C[\norm{Z(\cdot, 0)}_{H^1} + \norm{Z}_{\mathcal{X}_T}^2] \leq C[1+C_0]\epsilon^2 \leq \epsilon.
\end{equation}
This proves the last estimate of \eqref{closing-a},
completing the proof of the lemma.
\end{proof}

\medskip

From Lemma \ref{lemma-closing} and the local existence theory \cite{MR2883845}, it follows that there exists unique
global solution to equation \eqref{N.eqn.gn},
\[
\psi(x, t)=\big(\phi\sb{\omega(t)}(x)+\rho(x,t)\big)
e^{-i\left(\int_0^t\omega(s)\,ds+\gamma(t)\right)},
\qquad
t\ge 0,
\]
with $\omega$, $\gamma$, and $\rho$ satisfying the estimates
\[
\norm{\dot\omega}_{L^1(\R\sb{+})}
+\norm{\dot\gamma}_{L^1(\R\sb{+})}\leq \epsilon,
\qquad
\norm{Z}_{\mathcal{X}_\infty} \leq \epsilon.
\]
    From this, we infer that there exist
$\omega_\infty,\,\gamma_\infty\in\R$ such that
\[
\lim_{t\rightarrow \infty}\omega(t) = \omega_\infty, \qquad \lim_{t \rightarrow \infty}\gamma(t) = \gamma_\infty,
\qquad
\lim_{t\rightarrow \infty}\norm{Z(t)}_{L^\infty_x} =0.
\]
The last relation is due to
$\norm{Z}\sb{L^4_t L^\infty_x}
\le \norm{Z}\sb{\mathcal{X}_\infty}$.
Due to \eqref{r-z-same} and \eqref{para-Linfty},
assuming that $\epsilon>0$ is sufficiently small,
we also have
\[
\lim_{t\rightarrow \infty}\norm{\rho(t)}_{L^\infty_x}=0.
\]
This completes the proof of the main theorem.

\appendix

\section{Appendix: Estimates for the linear perturbed equation}

This subsection proves the estimate \eqref{Z-eqn.est} on $Z$.
The main result is the following lemma:

\begin{lemma}\label{X.est}
Fix $\omega_0 \in\varOmega$.
Let $Z(t) \in\bX_c(\omega_0)$ be a solution to the equation
\[
\left\{
\begin{array}{lll}
& \p_t {Z}-\bJ\bL (\omega_0) {Z}
+\alpha(t) \bJ{Z} = F,
&\quad t >0, \\
& {Z}(0) = {Z}_0 \in\bX_c(\omega_0).
\end{array} \right.
\]
Then there exist
$c_0 >0$ and $C<\infty$
such that if $|\alpha(t)| \leq c_0$, we have
\[
\norm{{Z}}_{\mathcal{X}} \leq C\Big[\norm{{Z}(0)}_{H^1}
+\norm{F}_{\mathcal{Y}} \Big].
\]
\end{lemma}

We recall that $\bX_c(\omega_0)$
is defined in \eqref{def-x-decomposition}.

\begin{proof} It follows from our linear estimates in Section~\ref{disp-section} that Lemma~\ref{X.est}
holds when $\alpha =0$. The proof therefore is a perturbative argument.
 We base our argument on \cite[Appendix B]{MR2898769},
which originates in \cite{MR2831875}. In the perturbation argument, instead of using the
free operator as in \cite{MR2831875, MR2898769}, we shall make use of the operator
\[
\bL_\nu
=\begin{bmatrix}H_\nu&0\\0& H_\nu\end{bmatrix}, \qquad \text{with} \quad
H_\nu := D_m -\omega_0+ V_\nu,
\]
where $V_\nu$ is a fixed matrix-valued potential which is sufficiently small
and decays exponentially,
and such that the point spectrum $\sigma_d(H_\nu)$
of $H_\nu$ is empty
and there is no resonance at thresholds $\Lambda=\pm m-\omega_0$. The advantage of using $\bL_\nu$ is that it has stronger decay estimates \eqref{ref-est} which essentially follow from \cite[Theorem 3.7]{MR2857360}.

We now denote
$\bW_\nu=\bL(\omega_0)-\bL_\nu$, the exponentially decaying
matrix potential; thus,
\[
\bL(\omega_0) = \bL_\nu + \bW_\nu.
\]
For fixed $\varkappa >0$ and for $P_d(\omega_0)
:= \text{Id}-P_c(\omega_0)$, we consider the auxiliary equation
\begin{equation}\label{au-est}
\p_t \Psi-\bJ\bL(\omega_0) P_c(\omega_0) \Psi + \varkappa P_d(\omega_0) \Psi + \alpha \bJ P_c(\omega_0) \Psi
=F,
\qquad \Psi(0) = {Z}(0).
\end{equation}
We note that ${Z}=P_c(\omega_0)\Psi$,
therefore it suffices to prove the estimate for $\Psi$. Let us denote
\[
\beta(t) = \int_0^t \alpha(s)\,ds,
\qquad U(t) = e^{\beta(t)\bJ},
\qquad \Psi(t) = U(t) \Phi.
\]
Then it follows from \eqref{au-est} that
\[
\p_t \Phi + U^{-1}(-\bJ\bL(\omega_0)+\varkappa P_d(\omega_0))U\Phi=
G, 
\qquad
G: = U^{-1}F + \alpha(t) U^{-1}\bJ P_d(\omega_0) U\Phi.
\]
Since $\bJ$ commutes with $\bL_\nu$, we obtain
\begin{equation}\label{Phi.eqn}
\p_t \Phi -\bJ\bL_\nu \Phi =-U^{-1}(\bW-\bJ\bL(\omega_0) P_d(\omega_0) + \varkappa P_d(\omega_0)) U\Phi + G.
\end{equation}
Now, we choose $V_2$ a smooth, exponentially decaying, invertible matrix potential such that the matrix
\[
V_1 = (\bW-\bJ\bL P_d + \varkappa P_d) V_2^{-1}
\]
is also smooth and exponentially decaying.
Then, note that $\Phi(0) =\Psi(0)$, and
$\bL_\nu$ commutes with $\bJ$. Therefore, applying $U(t)$
to both sides of equation \eqref{Phi.eqn}, we infer that
\begin{equation}\label{R.eqn-est}
\begin{split}
\Psi(t)&= U(t)e^{-t \bJ\bL_\nu}\Psi(0) + \int_0^t e^{-(t-s)\bJ\bL_\nu}\Big[U(t)U^{-1}(s) V_1V_2 \Psi(s)-U(t) G(s) \Big]ds \\
&=U(t)e^{-t \bJ\bL_\nu}\Psi(0) + \int_0^t e^{-(t-s) \bJ\bL_\nu}U(t)U^{-1}(s) \Big[(V_1-\alpha(s) \bJ P_d(\omega_0) V_2^{-1} )V_2 \Psi-F(s)\Big]ds.
\end{split}
\end{equation}
On the other hand, it follows from \cite[Section VIII]{MR2985264} that
\begin{equation*}
\norm{\Psi}_{\mathcal{X}} \leq
C\Big[\norm{\Psi(0)}_{H^{1}} +\norm{F}_{\mathcal{Y}} + \norm{V_2 \Psi}_{L^2_t H^1_x} +\norm{\alpha}_{L^\infty}\norm{P_d \Psi} _{\mathcal{Y}}\Big].
\end{equation*}
Note that
$\norm{P_d \Psi}_{\mathcal{Y}} \leq C\norm{\Psi}_{\mathcal{X}}$.
Therefore, if $\norm{\alpha}_{L^\infty}$ is sufficiently small, we obtain
\begin{equation}\label{R-est.1}
\norm{\Psi}_{\mathcal{X}} \leq
C\Big[\norm{\Psi(0)}_{H^{1}} +\norm{F}_{\mathcal{Y}} + \norm{V_2 \Psi}_{L^2_t H^1_x}\Big].
\end{equation}
Next, we need to control $\norm{V_2 \Psi}_{L^2_t H^1_x}$. We denote
\[
T_0 f(t) = V_2\int_0^t e^{-(t-s)\bJ\bL_\nu} U(t) U^{-1}(s) V_1 f(\cdot, s)\,ds.
\]
    From Lemma~\ref{invert-T0} below, we see that
the mapping
$I-T_0
:\;L^{2}_{t}H^1_x\to L^2_{t}H^1_x$
is invertible
and there exists $C<\infty$ such that
$
\norm{(I- T_0)^{-1}}_{L^2_{t}H^1_x \rightarrow L^2_{t}H^1_x} \leq C.
$
By \eqref{R.eqn-est}, we see that
\begin{equation}
(I-T_0) V_2 \Psi = V_2 U(t) e^{-t\bJ\bL_\nu}\Psi(0)
-V_2 \int_0^t e^{-(t-s)\bJ\bL_\nu} U(t)U^{-1}(s)[F(s) +\alpha(s)\bJ P_d \Psi(s)]\,ds.
\end{equation}
Therefore, using again the linear estimates from \cite{MR2985264},
we obtain:
\[
\begin{split}
\norm{V_2 \Psi}_{L^2_t H^1_{x}}&\leq \norm{V_2 U(t) e^{-t \bJ\bL_\nu}\Psi(0)}_{L^2_{t}H^1_x}
+ \norm{V_2 \int_0^t e^{-(t-s) \bJ\bL_\nu} U(t)U^{-1}(s)[F(s) +\alpha(s)\bJ P_d \Psi(s)]\,ds}_{L^2_{t}H^1_x}\\
& \leq
C\Big[ \norm{\Psi(0)}_{H^1_x} +\norm{F}_{\mathcal{Y}} + \norm{\alpha}_{L^\infty} \norm{\Psi}_{\mathcal{X}}\Big].
\end{split}
\]
    From this and \eqref{R-est.1}, we see that there is $c_0>0$ sufficiently small such that if $\norm{\alpha}_{L^\infty} \leq c_0$, then one has
$
\norm{\Psi}_{\mathcal{X}} \leq C\Big[\norm{\Psi(0)}_{H^1} + \norm{F}_{\mathcal{Y}} \Big].
$
Since ${Z}=P_c(\omega_0)\Psi$,
this completes the proof of the lemma.
\end{proof}
\begin{lemma}\label{invert-T0}
For $k =0,\,1$, the map $I-T_0 : L^2_{t}H^k_x \mapsto L^2_{t}H^k_x$ is invertible and therefore there exists
$C<\infty$ such that
\[ \norm{(I- T_0)^{-1}}_{L^2_{t}H^k_x \rightarrow L^2_{t}H^k_x} \leq C. \]
\end{lemma}
\begin{proof}
First, note that it follows from the linear estimates in \cite[Section VIII]{MR2985264} that $T_0$ is well-defined
as an operator from $L^2_{t}H^k_x$ to $L^2_{t}H^k_x$, with $k =0,\,1$. We now let
\[
T_1f(t) =V_2\int_0^t e^{-(t-s) \bJ\bL_\nu}V_1 f(\cdot, s)\,ds.
\]
It follows from our linear estimates in Section~\ref{disp-section} that $T_1$ is also well-defined from $L^2_{t}H^1_x$ to $L^2_{t}H^1_x$.
Also, note that
\[
(T_1 -T_0)f =V_2\int_0^t e^{-(t-s)\bJ\bL_\nu}
\Big(e^{\bJ \int_{t}^s\alpha(\tau)\,d\tau} -1 \Big)V_1 f(\cdot, s)\,ds.
\]
By \cite[Theorem 3.7]{MR2857360}, we have
\[
\norm{e^{-t\bJ\bL_\nu}}_{L^2_{\sigma} \rightarrow L^2_{-\sigma}} \leq C_{\sigma}\wei{t}^{-3/2}, \qquad
\sigma > 5/2.
\]
    From this, we further infer that
$
\norm{\bL_\nu e^{-t\bJ\bL_\nu} f}_{L^2_{-\sigma}} \leq C_\sigma \wei{t}^{-3/2} \norm{\bL_\nu f}_{L^2_\sigma}.
$
Since
$
\norm{f}_{H^1_x} \sim \norm{f}_{L^2} + \norm{\bL_\nu f}_{L^2},
$
we see that
\begin{equation}\label{ref-est}
\norm{e^{-t\bJ\bL_\nu} }_{H^k_{\sigma} \rightarrow H^k_{-\sigma}} \leq C_\sigma \wei{t}^{-3/2}, \qquad \sigma > 5/2, \qquad k =0,\,1.
\end{equation}
Using \eqref{ref-est} and the fact that
\[
\Big| e^{\bJ \int_{t}^s\alpha(\tau)\,d\tau} -1 \Big| \leq
\min\left(
1,\ \norm{\alpha}_{L^\infty}(t-s)
\right),
\]
we obtain:
\[
\norm{V_2 e^{-(t-s) \bJ\bL_\nu}
\Big[e^{\bJ \int_{t}^s\alpha(\tau)\,d\tau} -1 \Big]V_1 f(\cdot, s)\,ds}_{H^k}
\leq C\norm{\alpha}_{L^\infty}^{1/4} \wei{t-s}^{-5/4}\norm{f(\cdot,s)}_{H^k}.
\]
Thus, if $\norm{\alpha}_{L^\infty}$ is sufficiently small, we see that
\[ \norm{T_1- T_0}_{L^2_t H^k_{x} \rightarrow L^2_{t}H^k_x} \leq C\norm{\alpha}_{L^\infty}^{1/4} <1.
\]
Therefore, it suffices to prove that $I- T_1$ is invertible.
The lemma then follows exactly as in \cite[Lemma B.2]{MR2898769}
by using the linear estimates on $e^{-t\bJ\bL(\omega_0)}$
from Section~\ref{disp-section}.
\end{proof}

\bibliographystyle{sima-doi}
\bibliography{all}
\end{document}